\documentclass[a4paper,11pt]{article}
\usepackage[utf8]{inputenc}
\usepackage{setspace} 
\usepackage{authblk}
 \usepackage[lmargin=4cm, rmargin=4cm, top= 4cm, bottom = 4cm]{geometry}
\usepackage{mathtools}
\usepackage{commath}
\usepackage{amsmath}
\usepackage{amsthm}
\usepackage{amssymb}
\usepackage{amsfonts}
\usepackage{graphicx}
\graphicspath{ {./images/} } 
\usepackage{color}
\usepackage{calc}
\usepackage{IEEEtrantools}
\usepackage[percent]{overpic}
\usepackage{xcolor}
\usepackage{float}
\usepackage{enumerate}
\usepackage{tikz-cd}
\usetikzlibrary{decorations.pathmorphing}
\tikzset{snake it/.style={decorate, decoration=snake}}
\usepackage{lmodern}
\usetikzlibrary{arrows}
\def\E{\mathbb{E}}

\def\P{\mathbb{P}}
\def\pr{\mathbb{P}}

\def\det{\mathrm{det}}

\def\uppone{I^{\xi_1}}
\def\upptwo{I^{\xi_2}}
\def\uppm{I^{\xi_m}}
\def\lowone{I_{\xi_1}}
\def\lowm{I_{\xi_m}}
\def\har{\mathfrak{h}}
\def\K{\mathbb{K}}
\def\Pfaff{\mathrm{pf}}

\newcommand{\bracedincludegraphics}[2][]{%
  \sbox0{$\vcenter{\hbox{\includegraphics[#1]{#2}}}$}%
  \left\lbrace
    \vphantom{\copy0}
  \right.\kern-\nulldelimiterspace
  \underbrace{\vrule width0pt depth \dimexpr\dp0 + .3ex\relax\box0}}

\newtheorem{theorem}{Theorem}
\newtheorem{lemma}[theorem]{Lemma}
\newtheorem{proposition}[theorem]{Proposition}

\theoremstyle{definition}

\usepackage[nottoc,notlot,notlof]{tocbibind}

\title{Ordered exponential random walks}
\author{Denis Denisov and Will FitzGerald
\thanks{
  {\it Emails:} denis.denisov@manchester.ac.uk and william.fitzgerald@manchester.ac.uk. \newline
    This research was funded by a Leverhulme Trust Research Project Grant  RPG-2021-105.\newline  
}
}
\affil{University of Manchester}

\begin{document}
\maketitle

\begin{abstract}
We study a $d$-dimensional random walk with exponentially distributed increments conditioned so that the components stay ordered (in the sense of Doob). 
We find explicitly a positive harmonic function $h$ for the killed process
and then construct an ordered process using Doob's $h$-transform.
Since these random walks are not nearest-neighbour, the 
harmonic function is not the Vandermonde determinant. 
The ordered process is related to the departure process 
of M/M/1 queues in tandem. 
We find asymptotics for the tail probabilities of the time until the components in exponential random walks become disordered 
and a local limit theorem. 
We find the distribution of the processes of smallest and largest particles as 
Fredholm determinants. 
%If $\eta_1= 0$ and $(\eta_j)_{2 \leq j \leq d}$ are independent Gamma$(j-1, 1)$ random variables then $h(x) = \E[\prod_{1 \leq i < j \leq d} x_j +\eta_j - x_i -\eta_i]$ is harmonic for an exponential random walk with equal rates $1$ that is killed when the components become disordered. 
%We also consider distinct rates, analyse tail asymptotics for the time until the components become disordered and relate the largest and smallest particles in the ordered process to Fredholm determinants. 
%An exponential random walk conditioned (in the sense of Doob) so that the position of the walk interlaces with its position at the previous time step has been well studied due to an abundance of connections to last passage percolation and Young tableaux.

\end{abstract}

\section{Introduction} 
\label{intro}
Random walks in Weyl chambers have many connections. 
In random matrix theory, the eigenvalues of a Brownian motion on the space of complex Hermitian matrices evolve as a non-colliding system of Brownian motions called Dyson Brownian motion, while certain non-colliding random walks are related to orthogonal polynomial ensembles \cite{konig2005}.
The analysis of many interacting particle systems in the Kardar-Parisi-Zhang (KPZ) universality class involves the construction of processes on Gelfand-Tsetlin patterns where the bottom layer is a process in a Weyl chamber, eg. 
\cite{bfps, johansson_2, warren}. Furthermore there are connections to tandem queueing networks \cite{glynn_whitt, o_connell2002} which we discuss in Appendix \ref{queueing}.
%The output process of applying the Robinson-Schensted-Knuth (RSK) correspondence to exponential data is an $h$-transform of a random walk with exponential increments killed when it fails to interlace with its position at the previous time step. This plays a crucial role in the analysis of last passage percolation 
%and is connected to Young tableaux, Schur polynomials and queueing theory 
% \cite{johansson_2}. 
A variety of physical phenomena are modelled by ordered random walks in Fisher \cite{fisher}. 

Nearest-neighbour $d$-dimensional random walks with zero mean which are conditioned so that the components stay ordered for all time (in the sense of Doob) are well understood. The Karlin-McGregor formula gives the transition density in the form of a determinant and the Vandermonde determinant is a harmonic function for the random walk killed when the components become disordered. 
There has been recent progress when the jumps are no longer nearest-neighbour based around using Brownian approximations,
for example 
\cite{denisov_wachtel10, eichelsbacher_konig}. This has led to generalisations to different Weyl chambers, random walks in cones
\cite{denisov_wachtel15} and integrated random walks. In general, the harmonic functions are more 
complicated and explicit calculations are not possible. 
  
%. The first problem is that there is no (simple) Karlin-McGregor formula for the transition density. 
%Moreover, constructing a harmonic function for the random walk that is zero at the boundary of the Weyl chamber is no longer sufficient to construct a harmonic function for the random walk killed when the components become disordered.
%Much of the progress beyond nearest neighbour random walks has been .
 %Exponential random walks are an example where this obstacle occurs but nonetheless concise formulae remain. 

In the analysis of last passage percolation an important role is played by an $h$-transform of a $d$-dimensional random walk with exponential increments killed when it fails to interlace with its position at the previous time step. 
This is the output process of applying the Robinson-Schensted-Knuth (RSK) correspondence to last passage percolation with exponential data. 
The largest particle in the $h$-transformed process satisfies a number of process-level identities with sequences of last passage percolation times \cite{johansson_2}, the sequence of departure times from the last queue in a tandem queueing network (see Appendix \ref{queueing}) and the largest eigenvalues of a sequence of minors of the Laguerre Unitary Ensemble \cite{borodin_peche, dieker_warren}. 
This process is a random walk conditioned to satisfy an interlacing rather than ordering condition. When started from zero there is an exact coupling that relates the two types of conditioning (see \cite{o_connell_2003} or Section \ref{coupling} and Appendix \ref{zero_section}). For general starting positions the relationship is more complicated.

%This plays a crucial role in the analysis of last passage percolation 
%and is connected to Young tableaux, Schur polynomials and queueing theory 
% \cite{johansson_2}. 

In this paper, we analyse certain stopping times and $h$-transforms of $d$-dimensional random walks with exponential increments. 
This connects the example above, arising in the study of last passage percolation, with the general theory of ordered random walks. 
Moreover, this provides an example of an ordered random walk where the increment distribution is not nearest-neighbour but where explicit calculations are still possible.

Let $(X_{ij})_{i \geq 1, 1 \leq j \leq d}$ be independent exponential random variables with rates $\lambda_j > 0$. 
Let $W^d = \{(x_1, \ldots, x_d) \in \mathbb{R}^d : x_1 \leq x_2 \leq \ldots \leq x_d\}$ denote the Weyl chamber.
We define a $d$-dimensional random walk $(S(n))_{n \geq 0} = (S_1(n), \ldots, S_d(n))_{n \geq 0}$ started from $S(0) = x^0 = (x_1^0, \ldots, x_d^0) \in W^d$ by 
$S_j(n) = x_j^0 + \sum_{i=1}^n X_{ij}$ for $n \geq 1$ and $j = 1, \ldots, d$. 
Vectors $a = (a_1, \ldots, a_d)$ and  $b = (b_1, \ldots, b_d) \in W^d$ interlace written as $a \prec b$ if $a_1 \leq b_1 \leq \ldots \leq a_d \leq b_d$. 
We can define two stopping times: 
\begin{align*}
\rho & := \inf\{n \geq 1: S(n-1) \nprec  S(n)\}\\
\tau & := \inf\{n \geq 1:  S(n) \notin W^d\}.
\end{align*}
In the case $\lambda_1 > \ldots > \lambda_d,$ it is easy to construct $(S(n))_{n \geq 0}$ conditioned on $\{\rho = \infty\}$ or $\{\tau = \infty\}$ as these events have non-zero probability of occurring. 
For equal rates $\lambda_1 = \ldots = \lambda_d = 1$, a natural approach is to construct Doob $h$-transforms
which requires finding strictly positive functions $\har$ on 
$\text{int}(W^d)$ and $h$ on $W^d$ such that 
\begin{align*}
\E_x[\har( S(1)) 1_{\{\rho > 1\}}] & = \har(x), \quad x \in 
\text{int}(W^d) \\
\E_x[h( S(1)) 1_{\{\tau > 1\}}] & = h(x), \quad x \in W^d.  
\end{align*} 
The reason that we define $\mathfrak{h}$ and $h$ on $\text{int}(W^d)$ and $W^d$ respectively is due to their interpretation
as Doob $h$-transforms, see Appendix \ref{h_transform}.
A solution is given for distinct rates by
$\har(x_1, \ldots, x_d) = e^{\sum_{i=1}^d \lambda_i x_i} \det(e^{-\lambda_i x_j})_{i, j=1}^d$
which is strictly positive on $\text{int}(W^d)$ if $\lambda_1 > \ldots > \lambda_d$ and for equal rates by $\har(x) = \Delta(x)$ where 
\[
\Delta(x) = \prod_{1 \leq i < j \leq d} (x_j - x_i)
\] denotes the Vandermonde determinant throughout the paper. 
This corresponds to the fact that the output process of the RSK correspondence applied to exponential data is an honest Markov chain, for example see \cite{o_connell_2003}. 
Our first result is that a harmonic function for $(S(n))_{n \geq 0}$ killed when $\tau$ occurs can be given as follows.

In the case $\lambda_1 = \ldots = \lambda_d = \lambda > 0$, let $\eta_1 := 0$ and for $j = 2, \ldots, d$ let $\eta_j$ be a sequence of independent Gamma$(j-1, \lambda)$ 
random variables. We define
\[
h(x_1, \ldots, x_d) = \E[\Delta(x_1 + \eta_1, \ldots, x_d + \eta_d)], \quad x = (x_1, \ldots, x_d) \in W^d.
\] 
For distinct $\lambda_1, \ldots, \lambda_d$ and $x = (x_1, \ldots, x_d) \in W^d$, we define
 \[h(x_1, \ldots, x_d)  =  e^{\sum_{i=1}^d \lambda_i x_i} \det(\lambda_i^{i-j} e^{-\lambda_i x_j})_{i, j=1}^d.
 \]
We will only specify the dependency on the rates as $h^{(\lambda_1, \ldots, \lambda_d)}$ when they differ from the rates used in the exponential random variables. 
\begin{theorem}
\label{harmonic}
In the case $\lambda_1 = \ldots = \lambda_d = \lambda > 0$ and $\lambda_1 > \ldots > \lambda_d$ 
then $h$ defined above is a solution to $\E_x[h(S(1)) 1_{\{\tau > 1\}}] = h(x)$ 
satisfying
$h(x) > 0$ for all $x \in W^d$.
%\begin{enumerate}[(i)]
%\item 
%In the case $\lambda_1 = \ldots = \lambda_d = \lambda > 0$, let $\eta_1 := 0$ and for $j = 2, \ldots, d$ let $\eta_j$ be a sequence of independent Gamma$(j-1, \lambda)$ 
%random variables. Then
%\[
%h(x_1, \ldots, x_d) = \E[\Delta(x_1 + \eta_1, \ldots, x_d + \eta_d)], \quad x = (x_1, \ldots, x_d) \in W^d.
%\] 
%This satisfies $h(x) > 0$ for all $x \in W^d$. 
%\item For distinct $\lambda_1, \ldots, \lambda_d$ and $x = (x_1, \ldots, x_d) \in W^d$,
% \[h(x_1, \ldots, x_d)  =  e^{\sum_{i=1}^d \lambda_i x_i} \det(\lambda_i^{i-j} e^{-\lambda_i x_j})_{i, j=1}^d.
% \]
%If $\lambda_1 > \ldots > \lambda_d$ then $h^{(\lambda_1, \ldots, \lambda_d)}(x) > 0$ for all $x \in W^d$.  
%\end{enumerate}
\end{theorem}

In the case of equal drifts, this can be compared to the less explicit but more general formula for such a harmonic function from \cite{denisov_wachtel10, eichelsbacher_konig}, 
\[
 h(x) = \Delta(x) - \E_x(\Delta(S({\tau}))).
\]
The disadvantage of this formula is that $\E_x(\Delta(S({\tau})))$ is unknown. 
In \cite{denisov_wachtel10, eichelsbacher_konig} it was shown from such a formula that $h(x) \sim \Delta(x)$ as $x_{i+1} - x_i \rightarrow \infty$ for each $i = 1, \ldots, d-1$ and this is sufficient to prove weak convergence of ordered random walks to Dyson Brownian motion with $d$ fixed. Nevertheless there are interesting questions about ordered random walks which require a more detailed understanding of $h$. One example is where $d$ is allowed to grow with $n$, a problem of significant interest in understanding universality within the KPZ class. 
The new formula in Theorem \ref{harmonic} is helpful in such questions, for example it leads to a Fredholm determinant formula in Theorem \ref{largest_particle} that could be used to understand process-level asymptotic behaviour in 
various regimes.

The tail asymptotics of $\P(\tau > n)$ and $\P(\rho > n)$ are given in terms of these harmonic functions as follows. For sequences $(a_n)_{n \geq 1}$ and $(b_n)_{n \geq 1}$ we write $a(n) \sim b(n)$ if 
$a(n)/b(n) \rightarrow 1$ as $n \rightarrow \infty$. We define
\[\hat{h}(x_1, \ldots, x_d) = h(-x_d, \ldots, -x_1), \quad x \in W^d. \]  
\begin{theorem}
\label{tail_asy}
\begin{enumerate}[(i)]
\item If $\lambda_1 > \ldots > \lambda_ d$ then  
\begin{align*}
\P_x(\tau = \infty) & = h(x), \quad x \in W^d. 
 \end{align*}
\item If $\lambda_1 = \ldots = \lambda_d = \lambda > 0$ then uniformly for $x \in W^d$ and $x \in \text{int}(W^d)$ respectively with $x_d - x_1 = o(\sqrt{n})$, 
\begin{align*}
\P_x(\tau > n) & \sim \mathfrak{X} \lambda^{d(d-1)/2} h(x)n^{-d(d-1)/4}, \quad n \rightarrow \infty \\
\P_x(\rho > n) & \sim \mathfrak{X} \lambda^{d(d-1)/2} \Delta(x)n^{-d(d-1)/4}, \quad n \rightarrow \infty
\end{align*}
with
\begin{align*}
\mathfrak{X} 
%& =\frac{1}{(2\pi)^{d/2}\prod_{j=1}^{d-1} j!} \int_{W^d} e^{-y^2/2} \Delta(y) dy 
 = \frac{\prod_{j=1}^d \Gamma(j/2)}{\pi^{d/2}\prod_{j=1}^{d-1} j! }. 
\end{align*}
\item Suppose $\lambda_1 < \ldots < \lambda_d$ and let $\bar{\lambda} = \sum_{j=1}^d \lambda_j/d$ and 
$\lambda^* = (\prod_{i=1}^d \lambda_i)^{1/d}$.  Uniformly for $x \in W^d$ and $x \in \text{int}(W^d)$ respectively with $x_d - x_1 = o(\sqrt{n})$
\begin{align*}
\P_x(\tau > n)  & \sim K_{\lambda} n^{-\alpha} e^{-\gamma n} e^{\sum_{i=1}^d ( \lambda_i- \bar{\lambda} ) x_i} h^{(\bar{\lambda}, \ldots, \bar{\lambda})}(x), \quad n \rightarrow \infty \\
\P_x(\rho > n) & \sim   C_{\lambda} n^{ - \alpha} e^{-\gamma n} e^{\sum_{i=1}^d ( \lambda_i- \bar{\lambda} ) x_i} \Delta(x), 
\quad n \rightarrow \infty
\end{align*}
where $\gamma = d \log(\bar{\lambda}/\lambda^*) \geq 0$ and $\alpha = \frac{(d-1)(d+1)}{2}$
and the constant factors $K_{\lambda}$ and $C_{\lambda}$ are defined in Equations \eqref{factor1} and \eqref{factor2}.
%and \[
%K_{\lambda}  = \frac{1}{(2\pi)^{d/2}\prod_{j=1}^{d-1} j!}\int_{W^d} e^{-\sum_{j=1}^d (\lambda_j - \bar{\lambda} )z_j} \hat{h}^{(\bar{\lambda}, \ldots, \bar{\lambda})}(z) dz_1 \ldots dz_d.
%\]
%\begin{align*}
%K_{\lambda} & = \frac{1}{(2\pi)^{d/2}\prod_{j=1}^{d-1} j!}\int_{W^d} e^{-\sum_{j=1}^d (\lambda_j - \bar{\lambda} )z_j} \hat{h}^{(\bar{\lambda}, \ldots, \bar{\lambda})}(z) e^{-\frac{d z_1^2}{2}} dz_1 \ldots dz_d, \\
%C_{\lambda} & = \frac{\bar{\lambda}^{d(d-1)/2}}{(2\pi)^{d/2}\prod_{j=1}^{d-1} j!} \int_{W^d} e^{-\sum_{j=1}^d (\lambda_j - \bar{\lambda} )z_j} \Delta(z) e^{-\frac{d z_1^2}{2}}  dz_1 \ldots dz_d. 
%\end{align*}
\end{enumerate}
\end{theorem}
With equal drifts, tail asymptotics have been considered in \cite{denisov_wachtel10, eichelsbacher_konig} and in works which require some smoothness on the cone~\cite{denisov_wachtel21}. 
The theorem above extends existing results in various ways: considering different drifts, $\rho$ along with $\tau$ and uniformity in the starting positions. We also prove local limit theorems in Section \ref{sec:tail_asy} and believe that our arguments could be extended to 
give some information about next order terms in the asymptotic expansion. 
For completeness, it is known \cite{o_connell_2003} that $\P_x(\rho = \infty)  = \mathfrak{h}(x)$ for $x \in \text{int}(W^d)$.

%Note that (ii) improves both~\cite{denisov_wachtel10}, 
%where only fixed starting points have been considered and~\cite{denisov_wachtel21}, which requires some smoothness of the cone  
%and starting points which need to grow slower. 

A step in the proof of its own interest is that we find an explicit transition density for the random walk killed at $\tau$ in the form of a determinant. This is not a consequence of the Karlin-McGregor formula since the jumps are not nearest-neighbour and the functions appearing in the matrix have a dependency on the rows and columns in the matrix. 
\begin{proposition}
\label{transition_density}
For $x, z \in W^d$,
\[
\P_x(S(n) \in dz, \tau > n) = \prod_{j=1}^d \lambda_j^n e^{-\sum_{j=1}^d \lambda_j(z_j - x_j)}  \det(q_{n+i-j}(z_j - x_i))_{i, j = 1}^d dz
\]
where $q_n(x) = \frac{1}{(n-1)!}x^{n-1} 1_{\{x > 0\}} \text{ for } n \geq 1 \text{ and } q_n \equiv 0 \text{ for } n \leq 0.$
\end{proposition}

In the proof of Theorem \ref{tail_asy}, case (i) can be analysed directly. For case (ii) we use a formulation for 
the exit probability as a Pfaffian. The use of Pfaffians in this context is related to their appearance in plane partitions \cite{stembridge}, exit times from finite reflection groups \cite{doumerc_o_connell} and 
coalescing and annihilating particle systems \cite{FTZ, TZ}. For case (iii) we first prove a local limit theorem in the case of equal rates in Theorem \ref{LLT} and then apply 
a change of measure. 
We believe that it is possible to obtain tail asymptotics for all cases
$(\lambda_1, \ldots, \lambda_d) \in \mathbb{R}^d$ by combining the methods used here with those in \cite{puchala_rolski}. This would require introducing the notion of a stable partition and we do not pursue this here. 
%The evaluation of the integral expression for $\mathfrak{X}$ is (2.5.4) in~\cite{anderson_guionnet_zeitouni_09} 
%and the constant agrees with (1.2) and (1.3) in~\cite{puchala_rolski_05} which is expected since the constant does not depend on the increment distribution~\cite{denisov_wachtel10}.

In the case $\lambda_1 > \ldots > \lambda_d$ and $\lambda_1 = \ldots = \lambda_d = \lambda$ we define an ordered exponential random walk $(Z(n))_{n \geq 0} = (Z_1(n), \ldots, Z_d(n))_{n \geq 0}$ as a Doob $h$-transform of 
$(S(n))_{n \geq 0}$ killed when $\tau$ occurs using the harmonic function from Theorem \ref{harmonic}. We give a description of this construction in Appendix \ref{h_transform}.

When an ordered exponential random walk is started from zero then the largest particle satisfies several process level identities. We summarise some identities and their proofs in Appendix \ref{exact_identities}. One example involves last passage percolation times.
Let $(e_{ij})_{i \geq 1, 1 \leq j \leq d}$ be an independent collection of exponential random variables with rates $\lambda_j > 0$. We define last passage percolation times for $n \geq 1$ and $1 \leq k \leq d$ by $L(0, k) = 0$
and 
\begin{equation*}
L(n, k) = \max_{\pi \in \Pi(n, k)} \sum_{(i, j) \in \pi} e_{ij}
\end{equation*}
where $\Pi(n, k)$ is the collection of up-right paths from the point $(1, 1)$ to the point $(n, k)$. 
%On the other hand, we define $(A_{ij})_{1 \leq i \leq d, j \geq 1}$ to be an independent collection of complex Gaussians 
%with mean zero and variance $1/\lambda_j$. Let $A(n)$ be the $d \times n$ matrix formed from $(A_{ij})_{1 \leq i \leq d, 1 \leq j \leq n}$ and let $(M(n))_{n \geq 0}$ be defined by setting $M(0)$ to be the zero-matrix and 
%$M(n) = A(n) A(n)^*$ for $n \geq 1$. This is a matrix-valued stochastic process considered in Warren/Dieker.  
%Let $(\lambda_{\max}(M(n))_{n \geq 0}$ denote the process of the largest eigenvalue of $M(n)$. 
Then 
%a consequence of the application of RSK to last passage percolation combined with \eqref{identity_cond_laws} is that
\begin{equation}
\label{process_identities}
(Z_d(n))_{n \geq 0} \stackrel{d}{=} (L(n, d))_{n \geq 0}.
\end{equation}
The proofs of various identities of a similar form to \eqref{process_identities} often involve the construction of a process on a Gelfand-Tsetlin pattern where the bottom layer is a process satisfying an interlacing condition. In Appendix \ref{exact_identities} we show that there are also natural processes on Gelfand-Tsetlin patterns where the bottom layer satisfies an ordering condition.

%The identity \eqref{process_identities} is usually phrased from the viewpoint of interlaced exponential random walks. In Section \ref{coupling} we discuss the relationship to ordered exponential random walks and how this leads to some new perspectives on \eqref{process_identities}.

%These identities have usually been phrased in terms of exponential random walks conditioned to stay interlaced, 
%however, the perspective of ordered exponential random walks also seems relevant. For example, in Section \ref{coupling} we show that a natural consequence of the work of O'Connell and Yor \cite{o_connell2002} is that the largest particle in an ordered exponential random walk has the same distribution as the process of departure times in a tandem queueing network. Moreover, this leads to new variants of constructions of processes on Gelfand-Tsetlin patterns. 

%We introduce
%variants of these constructions that work with ordered random walks rather than interlaced random walks.   

When $(Z(n))_{n \geq 0}$ is not started from zero the identity \eqref{process_identities} no longer holds. Our next result shows that for general initial conditions the distribution of the processes of the largest and smallest particles can be described by a Fredholm determinant. 
Let $f_n$ be the probability density function of a $\mathrm{Gamma}(n, 1)$ random variable.
\begin{theorem}
\label{largest_particle}
Let $Z_1(0) = x_1, \ldots, Z_d(0) = x_d$ and $\lambda_1 = \ldots = \lambda_d = 1$.
Let $A$ be an invertible matrix with entries given for $k, l = 1, \ldots, d$ by
 \begin{align*}
 A_{kl}  = \int_{x_k}^{\infty} f_{n-d+k}(z - x_k) z^{l-1} dz = \E((x_k + \eta_{n-d+k})^{l-1})
 \end{align*}
 where $\eta_j \sim \mathrm{Gamma}(j-1, 1)$.
If $n_1 \geq d-1$, the largest particle satisfies
\begin{align*}
\P_x(Z_d(n_1) \leq \xi_1, \ldots, Z_d(n_m) \leq \xi_m) = 
\det(I - \bar{\chi}_\xi K \bar{\chi}_{\xi})_{l^2(\{n_1, \ldots n_k\} \times \mathbb{N}}
\end{align*}
where $\bar{\chi}_{\xi}(n_j, y) = 1_{\{y > \xi_j\}}$ and the extended kernel $K$ is given by 
\begin{align*}
 K(n_i, y; n_j, z) & = -  f_{n_j - n_i}(z - y)  1_{\{i < j\}} \\
& \quad + \sum_{k, l = 1}^d \int_{y}^{\infty} f_{n_m - n_i}(u-y)u^{k-1} du (A^{-1})_{lk} f_{n_j-d+l}(z - x_l).
\end{align*}
Let $B$ be an invertible matrix with entries given for $k, l = 1, \ldots, d$ by
 \[
 B_{kl} = \int_{x_k}^{\infty} f_{n-1+k}(z - x_k) z^{l-1} dz = \E((x_k + \eta_{n-1+k})^{l-1}).
 \]
 If $n_m - n_{m-1} \geq d-1$ then the smallest particle satisfies
\begin{align*}
\P_x(Z_1(n_1) \geq \xi_1, \ldots, Z_1(n_m) \geq \xi_m) = 
\det(I - \chi_\xi K \chi_{\xi})_{l^2(\{n_1, \ldots n_k\} \times \mathbb{N}}
\end{align*}
where $\chi_{\xi}(n_j, y) = 1_{\{y < \xi_j\}}$ and the extended kernel $K$ is given by 
\begin{align*}
 K(n_i, y; n_j, z) & = -  f_{n_j - n_i}(z - y)  1_{\{i < j\}} \\
& \quad + \sum_{k, l = 1}^d \int_{y}^{\infty} f_{n_m - n_i}(u-y)u^{k-1} du (B^{-1})_{lk} f_{n_j-1+l}(z - x_l).
\end{align*}
\end{theorem}
The distribution of the conditioned process can be expressed in terms of the harmonic function in Theorem \ref{harmonic}
and the transition density in Proposition \ref{transition_density}. However, the usual route to obtain a Fredholm determinant 
by the Eynard-Mehta theorem does not apply since neither are in the right form as determinants. 
The main idea to circumvent this difficulty is that the harmonic function in Theorem \ref{harmonic}
and transition density in Proposition \ref{transition_density} both have expressions as determinants where the functions appearing in the matrix satisfy \emph{derivative and integral relations}. 
This is more reminiscent of the study of interacting particle systems with local interactions in the KPZ universality class,
eg. \cite{bf_anisotropic, schutz, warren} and it is surprising to see this idea appear in ordered random walks. 
%\begin{corollary}
%Let $x_i = 0$ for all $i = 1, \ldots, n.$ The above statement holds with the extended kernels: 
%for the largest particle,  
%\[
% K_n(y, z)  =  \sum_{k = 1}^d L_k^{(n-d)}(y) L_k^{(n-d)}(z) (yz)^{(n-d)/2} e^{-(y+z)/2}
%\]
%and for the smallest particle, 
%
%\end{corollary}

The rest of the paper is structured as follows. In Section \ref{harmonic_properties} we prove Theorem \ref{harmonic} along with further properties of the harmonic function. In Section \ref{section.transition.density} we give an expression for the transition density of exponential random walks killed when $\tau$ occurs and prove uniform bounds. In Section \ref{sec:tail_asy} we prove Theorem \ref{tail_asy} along with local limit theorems. In Section \ref{sec.largest.particle} we prove Theorem \ref{largest_particle}. In Appendix \ref{h_transform} we give a brief recap on Doob $h$-transforms. In Appendix \ref{exact_identities} we consider the connections between ordered exponential random walks, last passage percolation, tandem queueing networks and push-block dynamics.

\section{Harmonic functions}
\label{harmonic_properties}

\subsection{Proof of Theorem \ref{harmonic}}
%Compute by averaging over the positions at time $n/2$,
%\begin{align*}
%\lim_{n \rightarrow \infty} n^{d(d-1)/4} \P_x(\tau > n) & = \lim_{n \rightarrow \infty} n^{d(d-1)/4}
% \E_x (\P_{S_{n/2}}(\tau > n); \tau_n > n/2) \\
% & = C \lim_{n \rightarrow \infty} \E(V^{\mathrm{BM}}(S_{n/2}); \tau_n > n/2).
%\end{align*}
%(How to justify... Lemma \ref{bounds} should be useful.) 
Suppose that $h$ solves $\E_x[h(S(1)) 1_{\{\tau > 1\}}] = h(x).$
The defining equation for $h$ can be written as
\begin{align*}
  h(x) & = \left( \prod_{j=1}^d \lambda _j \right)\int_0^\infty da_d \int_0^{x_d-x_{d-1}+a_d} da_{d-1} \ldots
  \int_0^{x_{2}-x_1+a_{2}}da_1 \\
  & \times \left( e^{-\sum_{i=1}^d \lambda_i a_i}
  h(x_1+a_1,\ldots, x_d+a_d) \right).
\end{align*}
After a substitution $b_1 = x_1 + a_1, \ldots, b_d = x_d + a_d$ then $h$ solves
\begin{align*}
% \label{eq0}
  h(x) = \left( \prod_{j=1}^d \lambda_j \right)
 \int_{x_d}^\infty
  db_d \int_{x_{d-1}}^{b_{d}} db_{d-1} \cdots \int_{x_1}^{b_{2}}  db_1
  e^{\sum_{i=1}^d \lambda_i (x_i - b_i)} h(b_1,\ldots,b_d).
\end{align*}
Letting $g(x_1,\ldots,x_d):= e^{-\sum_{i=1}^d \lambda_i x_i} h(x_1,\ldots,x_d)$
we can rewrite this as
\begin{equation}
  \label{eq1}
    g(x) = \left( \prod_{j=1}^d \lambda _j \right) \int_{x_d}^\infty
  db_d \int_{x_{d-1}}^{b_{d}} db_{d-1} \cdots \int_{x_1}^{b_{2}}  db_1
  g(b_1,\ldots,b_d). 
\end{equation}
Differentiating with respect to $x_1,x_2,\ldots,x_d$ we
obtain that $g$ satisfies the differential equation 
\begin{equation}\label{eq3}
    g_{x_1x_2\ldots x_d}={(-1)}^d \left( \prod_{j=1}^d \lambda _j \right) g(x_1,\ldots,x_d)
\end{equation}
along with the boundary conditions
\begin{align}\label{eq3.1}
  g_{x_d}(x)&=0,\quad x_d=x_{d-1}\\
  g_{x_{d-1}}(x)&=0,\quad x_{d-1}=x_{d-2} \nonumber \\ 
   &\ldots \nonumber \\
 g_{x_2}(x)&=0,\quad  x_{2}=x_1\nonumber.
\end{align}
We can also formulate the above as the following equation for $h$: 
\begin{equation}\label{eq4}
    \left(\lambda_1 I-\frac{\partial}{\partial x_1}\right)\cdots \left(\lambda_d I-\frac{\partial}{\partial x_d}\right)
    h(x) = \left( \prod_{j=1}^d \lambda _j \right) h(x)    
\end{equation}
with boundary conditions
\begin{align}\label{eq5}
  \lambda_d h(x)&=h_{x_d}(x), \quad x_d=x_{d-1}\\\nonumber
  \lambda_{d-1} h(x)  &=h_{x_{d-1}}(x), \quad x_{d-1}=x_{d-2} \\ \nonumber
   &\ldots \\
   \lambda_2 h(x)&= h_{x_{2}}(x),\quad  x_{2}=x_1  \nonumber.
\end{align}
Direct substitution shows that if $g$ or $h$ satisfies the 
partial differential equations and boundary condition above then they
solve~\eqref{eq1}.

\begin{proof}[Proof of Theorem \ref{harmonic}]

Let $\lambda_1 > \ldots > \lambda_d$. 
By differentiating the Leibniz formula for the determinant, $g(x) = \text{det}(\lambda_i^{i-j}\mathrm{e}^{-\lambda_i x_j})_{i, j=1}^d$
solves
\[
  g_{x_1x_2\ldots x_d} 
  =(-1)^d \left( \prod_{j=1}^d \lambda _j \right) g.
\]
To show the boundary condition \eqref{eq3.1} note that for $j \geq 2$ the $j$-th and $(j-1)$-th columns of $g_{x_j}$ are equal up to a sign on $\{x_j=x_{j-1}\}$. 
This proves Theorem \ref{harmonic} for distinct rates apart from the strict positivity which we defer to 
Lemma \ref{harmonic_simplify2}. This expresses $h$ as an expectation over a strictly positive random variable.

Consider now the case of equal rates $\lambda_1=\ldots=\lambda_d = \lambda > 0$. We set $\lambda = 1$ and can recover the general case by scaling.  
Our plan is to verify~\eqref{eq4} and boundary conditions~\eqref{eq3.1}.

To verify~\eqref{eq4} let $L=\left(I-\frac{\partial}{\partial x_1}\right)\cdots \left(I-\frac{\partial}{\partial x_d}\right)$
and apply $I-\frac{\partial}{\partial x_j}$ to the corresponding row to obtain 
that 
\[
L\Delta(x) = \begin{vmatrix}
  1 & x_1-1 & \cdots & x_1^{d-2}-(d-2)x_1^{d-3} & x_1^{d-1}-(d-1)x_1^{d-2} \\
\vdots & \vdots  & \ddots & \vdots & \vdots \\
  1 & x_d-1  & \cdots & x_{d}^{d-2} -(d-2)x_{d}^{d-3} & x_d^{d-1}-(d-1)x_d^{d-2}
\end{vmatrix}
\]
After applying column operations the right hand side equals $\Delta(x)$. Hence, 
\begin{align*}
    Lh_1(x) &= L\E[\Delta (x_1+\eta_1, \ldots,x_d+\eta_d)]
    = \E[L\Delta (x_1+\eta_1, \ldots,x_d+\eta_d)]\\ 
          &= \E[\Delta (x_1+\eta_1, \ldots,x_d+\eta_d)]=h_1(x).    
\end{align*}
This part of the argument works for any choice of $\eta_j$ and 
we choose independent random variables $\eta_j \sim \mathrm{Gamma}(j-1, 1)$ 
to satisfy the boundary conditions.

We show the formulation in \eqref{eq3.1}. It is convenient to rewrite the expectations over Gamma random variables as
integrals. For $j \geq 2$
\begin{align}
\label{gamma_integral}
\E[(x_j + \eta_j)^{i-1}]
%& = \frac{1}{(j-1)!} \int_0^{\infty}(x_j + y)^{i-1} y^{j-1} e^{-y} dy \\
& =  \frac{1}{(j-2)!} \int_{x_j}^{\infty}u^{i-1} (u-x_j)^{j-2} e^{-u+x_j} du \nonumber \\
& = (-1)^{j-1} e^{x_j} \phi^{(1-j)}_i(x_j).
\end{align}
Therefore 
\begin{equation}
\label{defn_g}
g(x_1, \ldots, x_d) = \mathrm{det}((-1)^{j-1} \phi_i^{(1-j)}(x_j))_{i, j = 1}^d
\end{equation}
where the exchange of the determinant and expectation uses the independence of the $\eta_j$.
For $j \geq 2$
\[
g_{x_j}(x) = 0, \quad x_j = x_{j-1} 
\]
since two columns in the matrix are equal and hence the determinant is zero. 
Therefore \eqref{eq3.1} holds.
Again we defer the positivity of $h$ to Lemma \ref{harmonic_simplify2}.
\end{proof}

\subsection{Alternative representations for the harmonic function}

With equal rates $\lambda_1 = \ldots = \lambda_d = 1$ we have three different representations for a strictly positive harmonic function on $W^d$ satisfying $h(x) = \E_x[h(S(n)); \tau > n]$. 
For $i \geq 1$ let $\phi_i(x) = x^{i-1}e^{-x}$ and for $j \geq 1$ let 
$\phi^{(j)}_i$ be the $j$-th derivative of $\phi_i$ and 
$\phi^{(-j)}_i(x) = (-1)^j \int_x^{\infty} \frac{(u - x)^{j-1}}{(j-1)!} \phi_i(u) du$.
Here, note that this notation is consistent, that is  $\frac{d}{dx} \phi^{(-j)}_i(x)=\phi^{(1-j)}_i(x)$. 
For $x \in W^d$, 
\begin{align*}
h_1(x) & = \E_x[\Delta(x_1 + \eta_1, \ldots, x_d + \eta_d)] \\
h_2(x) & = e^{\sum_{j = 1}^d x_j} \det((-1)^{d-j} \phi^{(d-j)}_i(x_j))_{i, j = 1}^d \\
h_3(x) & = \Delta(x) - \E_x[\Delta(S({\tau}))] = \lim_{n \rightarrow \infty} \E_x[\Delta(S(n)); \tau > n].  
\end{align*}
%It is not immediately obvious that these three definitions coincide but this will be proven in Lemma \ref{h_12} and 
%Section \ref{further_properties}. 
The two expressions for $h_3$ are equal, see~\cite{denisov_wachtel10}.
%The coupling between random walks conditioned to stay ordered and interlaced suggests the choice $h_1$,
%the limit as $\lambda_i \rightarrow 1$ of the harmonic function with distinct rates leads to the choice $h_2$ and the general theory of ordered random walks leads to the choice $h_3$. 
%All of these are representations of the same harmonic function.
%In the proof we have also shown the following.
\begin{lemma}
\label{h_12}
$h_1(x) = h_2(x)$ for all $x \in W^d$. 
\end{lemma}

\begin{proof}
From \eqref{gamma_integral} we have
\[
h_1(x) = e^{\sum_{j=1}^d x_j} \mathrm{det} ((-1)^{j-1} \phi^{(1-j)}_i(x_j))_{i, j=1}^d, \quad x \in W^d.
\]
Reformulating $Lh = h$ as  $g_{x_1x_2\ldots x_d}={(-1)}^d  g(x_1,\ldots,x_d)$  
we obtain that 
\[g_{x_1^{d-1}x_2^{d-1}\ldots x_d^{d-1}}=   g(x_1,\ldots,x_d).\]
Recall the expression for $g$ in \eqref{defn_g} and 
bring the derivatives in $x_j$ into the $j$-th column of the matrix 
(as well as redistributing negative signs)
to obtain
\[
e^{\sum_{j = 1}^d x_j} \det((-1)^{d-j}\phi^{(d-j)}_i(x_j))_{i, j = 1}^d
 = e^{\sum_{j = 1}^d x_j} \det((-1)^{j-1} \phi^{(1-j)}_i(x_j))_{i, j = 1}^d. \]
Therefore $h_1 = h_2$.
\end{proof}
It can be shown that $h_1 = h_3$. This relates our work to the general work on ordered random walks in \cite{denisov_wachtel10, eichelsbacher_konig} and cones in \cite{denisov_wachtel15}.
%\textcolor{blue}{for example this gives another proof of the strict positivity of $h$ as a consequence of Proposition 4 in \cite{denisov_wachtel10}. One idea which gives a direct proof that $h_1(x) = h_3(x)$ for all $x \in W^d$ is to use the coupling along with the fact that $\Delta$ is harmonic for the $S(n)$ killed at $\rho$.} 
We omit a direct proof since it is not needed in our arguments and there are some tedious details in the proof. Instead the fact that $h_1 = h_3$ can be observed once Theorem \ref{tail_asy} is established by comparing with Theorem 1 in \cite{denisov_wachtel10}. 

We briefly remark that much of the above also holds for ordered random walks with geometric increments. For $j = 1, \ldots, d$, let $X_j \sim \text{Geom}(1-q_j)$ with the convention  $\P(X_j = k) = (1-q_j)q_j^k$ for $k \in \mathbb{N}_0$. In this case, the corresponding harmonic function in Theorem \ref{harmonic} is given for distinct rates by 
\[
\prod_{j=1}^d q_j^{-x_j} \text{det}\left(\left(\frac{q_i}{1-q_i}\right)^{j-1} q_i^{x_j}\right)_{i, j= 1}^d.
\]

\subsection{Coupling between ordered and interlaced random walks}
\label{coupling}

 Let $(e_j^i : 1 \leq i < j \leq d)$ be an independent collection of exponential random variables such that 
 $e_j^i$ has rate $\lambda_j > 0$ for all $1 \leq i < j \leq d$.  Let $(V_j^i : 0 \leq i < j \leq d)$ be defined inductively by $V_j^0 := 0$ and $V_j^i = V_j^{i-1} + e_j^i.$ 
Let $A$ denote the event that $x_j + V_j^i \leq x_{j+1} + V_{j+1}^{i}$ for all $1 \leq i < j < d$. 
Let $\Psi = (0, V_2^1, \ldots, V_d^{d-1})$.

We now define two different random walks from the same independent family of exponential random variables 
$(X_{ij})_{i \geq 1, 1 \leq j \leq d}$ with rates $\lambda_j > 0$.
Define a random walk $(\mathcal{S}(n))_{n \geq 0} = (\mathcal{S}_1(n), \ldots, \mathcal{S}_d(n))_{n \geq 0} $ starting from the random initial condition $\mathcal{S}(0) = x + \Psi$ for $1 \leq j \leq d$ and $k \geq 1$ by
\begin{align*}
\mathcal{S}_j(k) = \mathcal{S}_j(k-1) + X_{kj}. 
\end{align*}
Secondly define a random walk $(S(n))_{n \geq 0} = (S_1(n), \ldots, S_d(n))_{n \geq 0}$
by $S_j(0) = x_j$ for $j = 1, \ldots, d$, 
\begin{align*}
S_j(i) & = x_j + V_j^i \text{ for } 1 \leq i < j \leq d \\
 S_j(k) & = S_j(k-1) + X_{k-j+1, j} \text{ for } 1 \leq j \leq d, k \geq j. 
 \end{align*}
These random walks are related by 
\[
\mathcal{S}_j(k) = S_j(k+j-1).
\]
For any $ 1 \leq j \leq d - 1$ the condition that 
\[
\mathcal{S}_j(k) \leq \mathcal{S}_{j+1}(k-1)
\]
is equivalent to the condition that 
\[
S_j(k+j-1) \leq S_{j+1}(k+j-1).
\]
Therefore the event that $(S(n))_{n \geq 0}$ started from $x_1 \leq \ldots \leq x_d$ is ordered for all time is equivalent to the event that $A$ holds and  
$(\mathcal{S}(n))_{n \geq 0}$ started from the random initial condition $x + \Psi$ interlaces for all time.  
Recall that $A$ is an ordering condition associated to the definition of $\Psi$.
It is possible to define other variants of these couplings which become particularly simple in the case when $x_ j = 0$ for all $j = 1, \ldots, d$, see Appendix \ref{exact_identities}.

We now apply this idea to the representation of $\P(\tau > n)$ and $\P(S(n) \in dy, \tau > n)$ 
which will be used in Section \ref{sec:tail_asy}.

%
%We can then extend to the ordered case by using the coupling in Section and
% interchanging limits using Proposition \ref{prop.uniform.bounds}.
 Let $(\gamma_j^i : i+j \leq d)$ be an independent collection of exponential random variables such that $\gamma_j^i$ has rate $\lambda_j > 0$.  
 Let $(U_j^i : i+j  \leq d)$ be defined inductively by $U_j^0 := 0$ and $U_j^i = U_j^{i-1} + \gamma_j^i.$ 
Let $B$ denote the event that $z_j - U_j^i \leq z_{j+1} - U_{j+1}^{i}$ for all $i+ j < d$. 
Let $\Phi = (U_1^{d-1}, \ldots, U_{d-1}^1, 0)$. If we reverse signs then the series of inequalities become
$-z_{j+1} + U_{j+1}^i \leq -z_j + U_j^i$.  
These inequalities correspond to 
the event $A$
 with the choices that 
$x_j = -z_{d+1-j}$ along with $V_j^i = U_{d+1-j}^i$ and 
$\Psi_j = \Phi_{d+1-j}$ for $j =1, \ldots, d$.
\begin{lemma}
\label{coupling_lemma}
\begin{enumerate}[(i)]
\item For $n \geq d$, 
\[
\P_x(S(n) \in dz, \tau > n) = \E_x[\P_{x + \Psi}(S(n-d+1) + \Phi \in dz, \rho > n-d+1); A, B]  
\]
\item For $n \geq d$, 
\[
 \E_x[\P_{x + \Psi}(\rho > n); A]  \leq \P_x(\tau > n) \leq  \E_x[\P_{x + \Psi}(\rho > n-d+1); A] 
\]
\end{enumerate}
\end{lemma}

\begin{proof}
Part (i) follows directly from the coupling described in this Section. $\Psi$ is a random initial condition associated with the ordering condition $A$. We then run an exponential random walk for time $n-d-1$ where the ordering condition has been shifted into an interlacing condition. At the end we need to add on a random variable $\Phi$ 
in order to recover the particle positions at a fixed time in the original random walk. The event $B$ is an ordering condition associated to $\Phi$.  

Part (ii) is similar. Instead of adding on $\Phi$, we impose the interlacing condition for either $n$ or $n-d+1$ steps to give lower and upper bounds. 
\end{proof}

\begin{figure}
\centering
 \begin{tikzpicture}[scale = 1]
 \node at (0, -2) {$\mathcal{S}_{1}(1) $};
 \node at (0, -1) {$\mathcal{S}_{1}(2) $};
  \node at (3, -1) {$\mathcal{S}_{2}(1) $};
    \node at (0, 0) {$\mathcal{S}_{1}(3)$};
        \node at (3, 0) {$\mathcal{S}_{2}(2)$};
            \node at (6, 0) {$\mathcal{S}_{3}(1)$};
                        \node at (0, 1) {$\mathcal{S}_{1}(4)$};
                                    \node at (3, 1) {$\mathcal{S}_{2}(3)$};
                                                \node at (6, 1) {$\mathcal{S}_{3}(2)$};
\node at (1.5, 0) {$\leq$};
\node at (4.5, 0) {$\leq$};
\node at (1.5, 1) {$\leq$};
\node at (4.5, 1) {$\leq$};
\node at (4.5, -1) {$\leq$};
\node at (1.5, -1) {$\leq$};
\node at (1.5, -2) {$\leq$};
\node at (4.5, -2) {$\leq$};

\node at (6, -1) {$x_3 + V_3^2$};
\node at (6, -2) {$x_3 + V_3^1$};
\node at (3, -2){$x_2 + V_2^1$};

% 
% \node at (4, 4) {$X_{11}$};
% \draw[->] (3.3, 4) -- (3.7, 4);
%  \draw[->] (4, 3.7) -- (4, 3.3);
% \node at (3, 4) {$X_{12}$};
% \draw[->] (2.3, 4) -- (2.7, 4);
%   \draw[->] (4, 2.7) -- (4, 2.3);
% \node at (2, 4) {$X_{13}$};
% \draw[->] (1.3, 4) -- (1.7, 4);
%   \draw[->] (4, 1.7) -- (4, 1.3);
% \node at (1, 4) {$X_{14}$};
% \draw[->] (3.3, 3) -- (3.7, 3);
%   \draw[->] (3, 3.7) -- (3, 3.3);
% \node at (4, 3) {$X_{21}$};
% \draw[->] (2.3, 3) -- (2.7, 3);
%   \draw[->] (3, 2.7) -- (3, 2.3);
% \node at (3, 3) {$X_{22}$};
% \draw[->] (3.3, 2) -- (3.7, 2);
% \node at (2, 3) {$X_{23}$};
%   \draw[->] (2, 3.7) -- (2, 3.3);
% \node at (4, 2) {$X_{31}$};
% \node at (3, 2) {$X_{32}$};
% \node at (4, 1) {$X_{41}$};
% \draw (0, 4) -- (4, 0);
%  \draw[->] (0.6, 3.6) -- (0.8, 3.8);
%  \draw[->] (1.6, 2.6) -- (1.8, 2.8);
%   \draw[->] (2.6, 1.6) -- (2.8, 1.8);
%    \draw[->] (3.6, 0.6) -- (3.8, 0.8);
%  \draw[snake=coil,segment aspect=1]           (3.7,1.3)  -- (3.3,1.7); 
 \end{tikzpicture}
\caption{An ordered random walk represented as an interlaced random walk with a random initial condition. The columns correspond to particles in both processes. A fixed row gives the fixed time positions of the ordered random walk and time increases upwards.} 
%The inequalities give the conditioning that is present in \emph{both} interlaced and ordered processes. The $V_j^i$ are a random initial condition which makes the conditionings match.}
\label{par_seq_harmonic}
\end{figure}

\subsection{Relations between harmonic functions}
%For non-zero initial conditions, the exact identities in Section \ref{coupling} do not hold. 
%Instead we can represent an ordered random walk as an interlaced random walk that is started from a random initial state, see Figure \ref{par_seq_harmonic}. To use this we will need the relationships between harmonic functions given in Lemma \ref{harmonic_simplify2}. 

To use the coupling in Section \ref{coupling} we need the following relationships between harmonic functions.
% Let $(e_j^i : 1 \leq i < j \leq d)$ be an independent collection of exponential random variables such that 
% $e_j^i$ has rate $\lambda_j > 0$ for all $1 \leq i < j \leq d$.  Let $(V_j^i : 0 \leq i < j \leq d)$ be defined inductively by $V_j^0 := 0$ and $V_j^i = V_j^{i-1} + e_j^i.$ 
%Let $A$ denote the event that $x_j + V_j^i \leq x_{j+1} + V_{j+1}^{i}$ for all $1 \leq i < j < d$. 
%Let $\Psi = (0, V_2^1, \ldots, V_d^{d-1})$.
\begin{lemma}
\label{harmonic_simplify2}
\begin{enumerate}[(i)]
\item If $\lambda_1 = \ldots = \lambda_d = 1$ then
$\E[\Delta(x + \Psi); A] = h(x)$ for $x \in W^d$.
\item  If $\lambda_1, \ldots, \lambda_d$ are distinct then 
\[
\E\big[ e^{\sum_{j=1}^d \lambda_j (x_j + \Psi_j)} \det(e^{-\lambda_i (x_j + \Psi_j)})_{i, j = 1}^d ; A\big] =  h^{(\lambda_1, \ldots, \lambda_d)}(x), \quad x \in W^d.
\] 
\item For simplicity set $\bar{\lambda} = 1$. Then for $x \in W^d$,
\[
\prod_{j=1}^d \lambda_j^{1-j} \E\big[e^{\sum_{i=1}^d (\lambda_i - 1)(x_i + \Psi_i)} \Delta(x + \Psi); A\big]= e^{\sum_{i=1}^d (\lambda_i - 1)x_i}  h(x).\]
%\[\bar{\lambda}^{d(d-1)/2}\prod_{j=1}^d \lambda_j^{1-j} \E\big[e^{\sum_{i=1}^d (\lambda_i - \bar{\lambda})(x_i + \Psi_i)} \Delta(x + \Psi); A\big]= e^{\sum_{i=1}^d (\lambda_i - \bar{\lambda})x_i}  h(x).\]
\end{enumerate}
\end{lemma}

\begin{proof}
We start by proving (i). 
We can remove the indicator functions appearing in the expectation on the left hand side of (i) using the following argument based on row operations in the determinant. Define a sequence of sets $(J_k)_{k=0}^{d(d-1)/2}$ by $J_0 = \{(i, j): 1 \leq i < j \leq d\}$ and inductively defining
$J_k = J_{k-1} \setminus \{(r, s)\}$ where $(r, s)$ is the maximal element in $J_k$ under an ordering in which $(i, j) > (k, l)$ if either $i > k$ or if $i = k$ and $j > l$. Thus $J_{d(d-1)/2} = \emptyset$. Let $D(\Psi) = \det((x_i + \Psi_{i})^{j-1})_{i, j = 1}^d$. Then
\begin{multline}
\label{split_ind2}
 \E\big[D(\Psi) \prod_{(i, j ) \in J_{k-1}} 1_{\{x_j + V_j^i \leq x_{j+1} + V_{j+1}^{i}\}} \big] =  \E\big[ D(\Psi)  \prod_{(i, j ) \in J_{k}} 1_{\{x_j + V_j^i \leq x_{j+1} + V_{j+1}^{i}\}} \big] \\
 \quad - \E\big[D(\Psi)  1_{\{x_s + V_{s}^{r} > x_{s+1} + V_{s+1}^{r}\}}
\prod_{(i, j ) \in J_{k}} 1_{\{x_j + V_j^i \leq x_{j+1} + V_{j+1}^{i}\}}  \big].
\end{multline}
By construction, there is no indicator function in the product over $J_{k}$ involving any of the random variables 
$V_s^{r}, \ldots, V_s^{s-1}, V_{s+1}^{r}, \ldots, V_{s+1}^{s}.$
On the event, $x_s + V_{s}^{r} > x_{s+1} + V_{s+1}^{r}$ using lack of memory 
$x_s + \Psi_s  \stackrel{d}{=} x_{s+1} + V_{s+1}^{r} + \zeta_{s-r}^{(1)}$
where $\zeta_{s-r}^{(1)} \sim \mathrm{Gamma}(s-r, 1).$
By definition,
$x_{s+1} + \Psi_{s+1} \stackrel{d}{=} x_{s+1} + V_{s+1}^{r} + \zeta_{s-r}^{(2)}$ where $\zeta_{s-r}^{(2)} \sim \mathrm{Gamma}(s-r, 1)$. Both $\zeta_{s-r}^{(1)}$ and $\zeta_{s-r}^{(2)}$ are independent of all other random variables and after taking expectations the $s$-th and $(s+1)$-th rows agree and the final term in \eqref{split_ind2} vanishes.  
This means we can successively remove all of the indicator functions from $\E[\Delta(x + \Psi); A]$.  Once the indicator functions have been removed $\Psi_j \sim \mathrm{Gamma}(j-1, 1)$ are independent random variables so that
\begin{align*}
\E[\Delta(x + \Psi); A]
= \E[\Delta(x + \Psi)]
 = h(x).
\end{align*}

For part (ii) we can remove the indicator function on $A$ by a similar argument. Equation \eqref{split_ind2} holds with 
$D(\Psi) = e^{\sum_{i=1}^d \lambda_i(x_i+\Psi_i) } \det(e^{-\lambda_i (x_j+\Psi_j)})_{i, j = 1}^d.$ 
By a similar argument, on the event $\{x_s + V_s^r > x_{s+1} + V_{s+1}^{r}\}$,
\begin{align*}
x_s + \Psi_s & \stackrel{d}{=} x_{s+1} + V_{s+1}^{r} + \zeta^{(1)} \\
x_{s+1} + \Psi_{s+1} & \stackrel{d}{=} x_{s+1} + V_{s+1}^{r} + \zeta^{(2)}
\end{align*}
where $\zeta^{(1)} \sim \mathrm{Gamma}(s-r, \lambda_s)$ and $\zeta^{(2)} \sim \mathrm{Gamma}(s-r, \lambda_{s+1})$ are independent of all other random variables. 
Therefore the $(i, s)$ and $(i, s+1)$ entries in the matrix defining the determinant
\begin{equation}
\label{vanishing_det2}
\E\big[D(\Psi)  1_{\{x_s + V_{s}^{r} > x_{s+1} + V_{s+1}^{r}\}}
\prod_{(i, j ) \in J_{k}} 1_{\{x_j + V_j^i \leq x_{j+1} + V_{j+1}^{i}\}}\big]  
\end{equation}
are given by 
\begin{align*}
 & e^{(\lambda_s - \lambda_i)(x_{s+1} + V_{s+1}^r + \zeta^{(1)})}, \\
 & e^{(\lambda_{s+1} - \lambda_i)(x_{s+1} + V_{s+1}^r + \zeta^{(2)})}.
\end{align*}
The random variables $\zeta^{(1)}$ and $\zeta^{(2)}$ are independent of the remaining random variables 
and we can find the expectations 
\begin{align*}
& \E (e^{(\lambda_s - \lambda_i) \zeta^{(1)}}) = \lambda_s^{s-r} \lambda_i^{r-s} \\
& \E (e^{(\lambda_{s+1} - \lambda_i) \zeta^{(2)}}) = \lambda_{s+1}^{s-r} \lambda_i^{r-s}.
\end{align*}
We now take the factors $\lambda_s^{s-r} e^{\lambda_s(x_{s+1} + V_{s+1}^r)}$ and $\lambda_{s+1}^{s-r} e^{\lambda_{s+1}(x_{s+1} + V_{s+1}^r)}$ which only depend on the index of the column outside of the determinant as prefactors. 
After doing this the $s$-th and $(s+1)$-th column both have $(i, s)$ and $(i, s+1)$ entry given by  
$\lambda_i^{r-s} e^{-\lambda_i(x_{s+1} + V_{s+1}^{r})}$.
Therefore \eqref{vanishing_det2} vanishes. 
This means the indicator function on $A$ can be removed after which we can compute
\begin{align*}
& \E\big[e^{\sum_{i=1}^d \lambda_i(x_i+\Psi_i) } \det(e^{-\lambda_i (x_j+\Psi_j)})_{i, j = 1}^d\big] \\
%& =  \E(e^{\sum_{i=1}^d \lambda_i(x_i + \Phi_i)} \det(\lambda_i^{d-j}e^{-\lambda_i(x_j + \Phi_j)}); A) \\
& \quad =  e^{\sum_{i=1}^d \lambda_ix_i } \det\big(e^{-\lambda_i x_j} \E[e^{(\lambda_j - \lambda_i)\Psi_j}]\big)_{i, j = 1}^d \\
& \quad =  e^{\sum_{i=1}^d \lambda_ix_i } \det\left(e^{-\lambda_i x_j} \left(\frac{\lambda_j}{\lambda_i}\right)^{j-1}\right)_{i, j = 1}^d \\
%& = \prod_{j=1}^d \lambda_j^{j-1} e^{\sum_{i=1}^d \lambda_ix_i } \det( \lambda_i^{1-j}e^{-\lambda_i x_j})_{i, j = 1}^d \\
& \quad =  e^{\sum_{i=1}^d \lambda_ix_i } \det\big(e^{-\lambda_i x_j} \lambda_i^{i-j} \big)_{i, j = 1}^d  \\
& \quad  = h(x).
\end{align*}
Part (iii) follows from the fact that 
$\prod_{j=1}^d \lambda_j^{1-j} e^{\sum_{i=1}^d (\lambda_i - 1) \Psi_i}$ can be viewed as a change of measure after which the $V_j^i$ all have rates $\bar{\lambda} = 1$. Therefore the proof follows as in part (i). 
\end{proof}

\subsection{Further properties of $h$.}
\label{further_properties}

\begin{lemma} 
\label{h_continuity}
For all $x \in W^d$
\begin{equation}
\lim_{\lambda_1, \ldots, \lambda_d \rightarrow 1} \frac{h^{(\lambda_1, \ldots, \lambda_d)}(x)}{\Delta(\lambda_d, \ldots, \lambda_1)} = \frac{h(x)}{\prod_{j=1}^{d-1} j!}.
\end{equation}
\end{lemma}

\begin{proof}
We have 
\[
h^{(\lambda_1, \ldots, \lambda_d)}(x) = \left( \prod_{i=1}^d \lambda_i^{i-d} \right) e^{\sum_{i=1}^d \lambda_i x_i} \det\big((-1)^{d-j} \big(\frac{d}{dx}\big)^{d-j} e^{-\lambda_i x_j}\big)_{i, j = 1}^d. 
\]
We fix $x$ and view $h^{(\lambda_1, \ldots, \lambda_d)}(x)$ as a function in the $\lambda_i$.
A standard fact is that for functions $\varphi_1, \ldots, \varphi_d$ which are differentiable $d-1$ times at $-\lambda$ 
we have
\begin{equation}
\label{singular_det_limit}
\lim_{\lambda_1, \ldots, \lambda_d \rightarrow \lambda}\frac{\det(\varphi_i(-\lambda_j))_{i, j = 1}^d}{\Delta(\lambda_d, \ldots, \lambda_1)} = \frac{\det(\varphi_i^{(j-1)}(-\lambda))_{i, j = 1}^d}{\prod_{j=1}^{d-1} j!}.
\end{equation}
Commuting the derivatives in $x_j$ and $\lambda_j$ in the determinant gives 
\[
\lim_{\lambda_1, \ldots, \lambda_d \rightarrow 1} \frac{h^{(\lambda_1, \ldots, \lambda_d)}(x)}{\Delta(\lambda_d, \ldots, \lambda_1)} = 
\frac{e^{\sum_{j = 1}^d x_j}}{\prod_{j=1}^{d-1} j!} \det((-1)^{d-j} \phi^{(d-j)}_i(x_j))_{i, j = 1}^d. \qedhere
\]
\end{proof}

Lemma \ref{h_continuity} is useful for proving convergence of $h$-transformed processes. It is not clear how it could be used in Theorem \ref{tail_asy}, for example to deduce part (ii) from part (i), since this would require commuting limits.
%
%\subsection{Related models}
%\textcolor{blue}{
%We briefly remark that much of the above also holds for ordered random walks with geometric increments and for continuous time random walks with exponential increments. 
%\begin{itemize}
%\item For $j = 1, \ldots, d$, let $X_j \sim \text{Geom}(1-q_j)$ with the convention  $\P(X_j = k) = (1-q_j)q_j^k$ for $k \in \mathbb{N}_0$. In this case, the corresponding harmonic function in Theorem \ref{harmonic} is given for distinct rates by 
%\[
%\prod_{j=1}^d q_j^{-x_j} \text{det}\left(\left(\frac{q_i}{1-q_i}\right)^{j-1} q_i^{x_j}\right)_{i, j= 1}^d.
%\]
%\item Consider a continuous time exponential random walk where each of the components have independent exponential clocks of rate $1$ and after a clock rings that component makes an exponential jump of rate $1$ to the right.  We show that $h$ is also harmonic for this continuous time random walk killed when the components become disordered. 
%For the process without killing, Corollary 2.2 of \cite{konig2002} shows that the Vandermonde is harmonic; thus $h$ is also harmonic by commuting the expectation and generator as in the proof of Theorem \ref{harmonic}. 
%For the boundary condition, we observe that if $Y \sim \mathrm{Exp}(1)$ then conditionally on the event $\{x_j + Y > x_{j+1}\}$ 
%the overshoot is an independent $\mathrm{Exp}(1)$ and 
%\[
% \E\big[\Delta(x_1 + \eta_1, \ldots, x_{j} + Y + \eta_{j}, x_{j+1} + \eta_{j+1}, \ldots, x_d + \eta_{d}) 1_{\{x_j + Y > x_{j+1}\}}\big]
% = 0
%\]
%since the $j$-th and $(j+1)$-th columns will coincide after taking expectations. 
%\end{itemize}
%}
%

\section{Transition densities and uniform bounds}
\label{section.transition.density}
Although the Karlin-McGregor formula does not apply in this setting, the condition $x \prec z$ and hence the transition density of $S(n)$ killed at $\rho$ can be expressed in terms of a determinant. 
Let 
\[q_n(x) = \frac{1}{(n-1)!}x^{n-1} 1_{\{x > 0\}} \text{ for } n \geq 1 \text{ and } q_n \equiv 0 \text{ for } n \leq 0.
\] 
Then for $x, z \in \text{int}(W^d)$, 
\begin{align}\label{eq.km.rho}
\widetilde G_n(x,z)dz & :=\P_x(S(n) \in dz; \rho > n) \nonumber \\
& =    
\left( \prod_{j=1}^d \lambda_j^n \right) e^{-\sum_{j=1}^d \lambda_j(z_j - x_j)} \det(q_n(z_j - x_i))_{i, j = 1}^d dz.
\end{align}
Let $f_n$ denote the probability density function of a $\mathrm{Gamma}(n,1)$ random variable for $n \geq 0$ 
and $f_n = 0$ for $n < 0$. 
In the case where $\lambda_1 = \ldots = \lambda_d = 1$ an alternative expression for the transition density is
\begin{equation}\label{eq:g.tilde.alternative}
\widetilde G_n(x,z) = \det(f_n(z_j - x_i))_{i, j = 1}^d. 
\end{equation}

Equation \eqref{eq.km.rho} is closely related to some of the arguments used in \cite{bfps}. It can be proven by starting with the case $n=1$ and then applying the Andr{\'e}ief (or Cauchy-Binet) identity: for a Borel measure $\nu$ and functions 
$f_i, g_i \in L^2(\mathbb{R}, \nu)$ for $1 \leq i \leq d,$
\[
\int_{W^d} \text{det}(f_i(x_j))_{i, j = 1}^d \text{det}(g_i(x_j))_{i, j = 1}^d \prod_{i=1}^d \nu(dx_i) = \text{det}\left(\int_{\mathbb{R}} f_i(x) g_j(x) \nu(dx)\right)_{i, j =1}^d. 
\] 
Proposition \ref{transition_density} states that the transition density can also be written as a determinant when $\rho$ is replaced by $\tau$. For all $n \geq 1$ define \[
G_n(x, z)dz  := \P_x(S(n) \in dz, \tau > n), \quad x, z \in W^d.
\]
We may specify the dependency on the rates using $G_n^{(\lambda_1, \ldots, \lambda_d)}$ and $\widetilde G_n^{(\lambda_1, \ldots, \lambda_d)}.$
 We observe the following integral and derivative relations which will be useful in proving Theorem 
 \ref{largest_particle}: for all $k, n \geq 1$
\begin{align}
\label{derivative_relations}
\frac{d^k}{dx^k} q_n(x) & = q_{n-k}(x), \quad x > 0, \\
\label{integral_relations}
\int_0^x \frac{(x - u)^{k-1}}{(k-1)!} q_n(u) du & = q_{n+k}(x), \quad x > 0. 
\end{align}
%
%For $k = 1, \ldots, n-1$ let $q_n^{(k)}(x)$ be given by the $k$-th derivative of $q$ for $x < 0$ and $q_n^{(k)} = 0$ for $x \leq 0$. For $k < 0$ 
%let $q^{(-k)}$ be the $k$-th iterated integral of $q$ given by
%$q_n^{(-k)}(x) = \int_0^x \frac{(x - u)^{k-1}}{(k-1)!} q(u) du$ for $x > 0$ and $q^{(-k)}(x) = 0$ for $x \leq 0$.
%

Define independent random variables $\chi_d=0$ and $\chi_{j}\sim \mathrm{Gamma}(d-j,1)$ for $j=1,\ldots,d-1$ 
and $\eta_1=0$ and $\eta_{j}\sim \mathrm{Gamma}(j-1,1)$ for  $j=2,\ldots,d.$ 
%\begin{proposition}
%\label{transition_density}
% Let $x = (x_1, \ldots, x_d)$ and $z = (z_1, \ldots, z_d)$, then for $x, z \in W^d$,
%\begin{equation}\label{eq:transition.density}
%G_n(x, z) = \prod_{j=1}^d \lambda_j^n e^{-\sum_{j=1}^d \lambda_j(z_j - x_j)}  \det(q_{n+i-j}(z_j - x_i))_{i, j = 1}^d.
%\end{equation}
An alternative form for the transition density in Proposition \ref{transition_density} with $n \geq d$ and $x, z \in W^d$ is
\begin{equation*}
    G_n(x,z)   = \prod_{j=1}^d \lambda_j^n e^{-\sum_{j=1}^d (\lambda_j-1)(z_j - x_j)}  
    \E\big[\mathrm{det}\left( f_{n-d+1}(z_j-\chi_{j}-x_i-\eta_i)\right)_{i,j=1}^d\big].
\end{equation*}
To prove the alternative form will follow from Proposition \ref{transition_density} note first that 
%\[
%f_{n+m}(t)  = \E f_n(t-\eta_m),  
%\]
%which immediately follows from 
\begin{equation}\label{eq:conv.gamma}
  \E f_n(t-\eta_m) = \int_0^t dz f_{m-1}(z) f_n(t-z) = 
  f_{n+m-1}(t). 
\end{equation}
We can rewrite first 
\[
  G_n(x,z)   = \prod_{j=1}^d \lambda_j^n e^{-\sum_{j=1}^d (\lambda_j-1)(z_j - x_j)}  
  \mathrm{det}\left( f_{n+i-j}(z_j-x_i)\right)_{i,j=1}^d.
\] 
Then, using~\eqref{eq:conv.gamma} two times  one can see that 
\[
  f_{n+i-j}(z_j-x_i) 
  =
  \E f_{n+1-j}(z_j-x_i-\eta_i)
  = 
  \E f_{n-d+1}(z_j-\chi_j-x_i-\eta_i). 
\]
Therefore, 
\begin{align*}
  G_n(x,z)   &= \prod_{j=1}^d \lambda_j^n e^{-\sum_{j=1}^d (\lambda_j-1)(z_j - x_j)}  
  \mathrm{det}\left(\E f_{n-d+1}(z_j-\chi_j-x_i-\eta_i)\right)_{i,j=1}^d\\ 
  &= \prod_{j=1}^d \lambda_j^n e^{-\sum_{j=1}^d (\lambda_j-1)(z_j - x_j)}  
  \E  [\mathrm{det}\left(f_{n-d+1}(z_j-\chi_j-x_i-\eta_i)\right)_{i,j=1}^d].  
\end{align*}   
When $\lambda_1=\ldots=\lambda_d=1$ we obtain the following 
connection between $G$ and $\widetilde G$ 
\begin{align}
G_n(x,z)   &= \E[\text{det} \left( f_{n-d+1}(z_j-\chi_{j}-x_i-\eta_i)\right)_{i,j=1}^d ]
\nonumber 
\\ 
&=\E\left[\widetilde G_{n-d+1}(x+(\eta_1,\dots,\eta_d),z-(\chi_1,\dots,\chi_d)
\right],
\label{eq.gn.second}
\end{align}
where we have also made use of~\eqref{eq:g.tilde.alternative}.

To prove Proposition \ref{transition_density} we need the following Lemma. 
\begin{lemma}
\label{lem.ibp}
For any $x, z \in W^d$ and any $n, m \geq 1$,
\begin{align*}
& \int_{W^d} \det(q_{n+i -j}(y_j - x_i))_{i, j = 1}^d \det(q_{m+i-j}(z_j - y_i))_{i, j =1}^d dy_1 \ldots dy_d \\
& =  \det( q_{n+m+i-j}(z_j - x_i))_{i, j = 1}^d.
\end{align*}
\end{lemma}

\begin{proof}
The $q_n$ satisfy derivative and integral relations \eqref{derivative_relations} and \eqref{integral_relations}. Therefore this is Lemma 5 (ii) of \cite{FW}. 
\end{proof}

\begin{proof}[Proof of Proposition \ref{transition_density}]
The one-step transition density for ordered exponential random walks is given for $x, y \in W^d$ by
\[
G_1(x, y) = \prod_{i=1}^d \lambda_i e^{-\sum_{i=1}^d \lambda_i (y_i - x_i)}\det(q_{1+i-j}(y_j - x_i))_{i, j = 1}^d
\]
since the matrix is lower triangular. This is simply rewriting the transition density of independent random walks with the ordering condition then imposed by constraining that $y \in W^d$. The 
advantage of this rewriting is that we can apply Lemma \ref{lem.ibp} to conveniently integrate over $y \in W^d$ and find the two-step transition density
\begin{align*}
G_2(x, z)  & = \int_{W^d} 
\left( \prod_{i=1}^d \lambda_i^2 \right) e^{-\sum_{i=1}^d \lambda_i (z_i - x_i)} 
\det(q_{1+i-j}(y_j - x_i))_{i, j = 1}^d \\
& \qquad \cdot \det(q_{1+i-j}(z_j - y_i))_{i, j = 1}^d dy_1 \ldots dy_d \\
& = \left( \prod_{i=1}^d \lambda_i^2 \right) e^{-\sum_{i=1}^d \lambda_i (z_i - x_i)} 
\det(q_{2+i-j}(z_j - x_i))_{i, j = 1}^d.
\end{align*}
The statement can then be proved inductively by using Lemma \ref{lem.ibp}.
%We proceed by induction on $n$. 
%We apply Lemma \ref{lem.ibp} to obtain the required formula for the $(n+1)$-step transition density,
%\begin{align*}
%& \P_x(S(n+1) \in dz, \tau > n+1) \\
%& = e^{-\sum_{i=1}^d \lambda_i (z_i - x_i)} \int_{W^d} \det(q_{n+i-j}(y_j - x_i))_{i, j = 1}^d \det(q_{1+i-j}(z_j - y_i))_{i, j = 1}^d  dy_1 \ldots dy_d \\
%& = e^{-\sum_{i=1}^d \lambda_i (z_i - x_i)} \det( q_{n+1+i-j}(z_j - x_i))_{i, j = 1}^d. \qedhere
%\end{align*}
%\begin{align*}
%& = e^{-\sum_{i=1}^d \lambda_i (z_i - x_i)} \int_{W^d} \det(q_n^{(1-i)}(y_j - x_i))_{i, j = 1}^d \det(q_1^{(j-1)}(z_j - y_i) dy_1 \ldots dy_d. 
%\end{align*}
%(This justification for this integration by parts step is a bit fiddly but is already in the text in Section 6.)
%Then we apply the Andr{\'e}ief (or Cauchy-Binet) identity to obtain that 
%\begin{align*}
%& P_x(S(n+1) \in dz, \tau > n+1) \\
%&  \quad = e^{-\sum_{i=1}^d \lambda_i (z_i - x_i)} \det\big(\int_{-\infty}^{\infty} q_n^{(1-i)}(y - x_i) q_1^{(j-1)}(z_j - y) dy\big)_{i, j = 1}^d \\
%& \quad =  e^{-\sum_{i=1}^d \lambda_i (z_i - x_i)} \det( q_{n+1}^{(j-i)}(z_j - x_i))_{i, j = 1}^d. \qedhere
%\end{align*}
\end{proof}

\begin{proof}[Alternative proof of Proposition \ref{transition_density}] 
We now give an alternative  proof  for $n\ge d$.  
This argument is inspired by the proof 
of the LGV lemma, see e.g. Theorem~1 in \cite{gessel_viennot89}. 
We will give this proof for  $\lambda_1=\cdots=\lambda_d=1$. 
The general case can be treated by using the change of measure. 
In this case one can rewrite
the proposed transition density as
\begin{equation}\label{eq:transition.density.2}
  \mathrm{det}\left( f_{n+i-j}(z_j-x_i)\right)_{i,j=1}^d, \qquad x, z \in W^d.
\end{equation}

We will now construct an auxillary model 
%corresponding to  
%Figure~\ref{fig:switch}
. 
Here we have $d$  random walks 
$\widehat S_i(n) $ starting at $x_i$ at time 
$d-i$ and 
arriving at $z_i$ at time $n+d-i$, 
which correspond to $(S_i(k))_{k=0}^n$ 
starting  at $x_i$ and arriving  at $z_i$.  
We  let  $\widehat S_i(k)=\partial$ 
for some fictitious state $\partial$ 
when  $k<d-i$ or  $k>n+d-i$.  
We denote the corresponding probability measure with $\pr_x$. 

More generally for a permutation 
$\pi \in \mathcal S_d$ we consider 
a random walk 
$\widehat S_i(n) $ that starts at $x_i$ 
at time 
$d-i$
and 
arrives at $z_{\pi(i)}$ 
at time $n+d-\pi(i)$, which has 
$n+i-\pi(i)$ steps 
and has the same distribution 
as $(S_i(k))_{k=0}^{n+i-\pi(i)}$ 
 starting  at $x_i$ and arriving  at $z_{\pi(i)}$.  

Let $\widehat \tau$ be the following stopping time 
\[
  \widehat \tau 
  :=\min\{k\ge 1\colon \widehat S_i(k)>\widehat S_{i+1}(k-1)\text{ for some } i=1,\ldots, d-1\},
\]
where as usual $\widehat \tau=\infty$ if the minimum is taken over the empty set. 
Then
\begin{align*}
  &\P_x(S(n)\in dz,\tau>n)
  =\P_x(\widehat S_i(n+d-i)\in dz_i,i=1,\ldots,d,\widehat \tau=\infty) \\ 
  &=\sum_{\pi\in\mathcal S_d} 
  \text{sgn}(\pi)
  \P_x(\widehat S_i(n+d-\pi(i))\in dz_{\pi(i)},
  i=1,\ldots, d, \widehat \tau=\infty).  
\end{align*}
The second equality holds since 
$z\in W^d$ and hence all probabilities 
are equal to zero unless $\pi$ is the identity 
permutation.  Note also that 
by the construction of $\widehat S(n)$
\[
G_n(x,z) dz  = 
\sum_{\pi\in\mathcal S_d} 
  \text{sgn}(\pi)
  \P_x(\widehat S_i(n+d-\pi(i))\in dz_{\pi(i)},
  i=1,\ldots, d). 
\]
Hence, we are left to prove that 
\begin{equation}\label{eq6}
  \sum_{\pi\in\mathcal S_d} 
  \text{sgn}(\pi)
  \P_x(\widehat S_i(n+d-\pi(i))\in dz_{\pi(i)},
  i=1,\ldots, d, \widehat \tau<\infty)=0.
\end{equation}

On the event $\widehat \tau<\infty$ we have two cases: 
one case when  
the edges of $(\widehat S(n)) $ 
have non-empty intersections, see Figure~\ref{fig:switch.overshoot},   
and the second case when the last value of a path exceeds the last value of another path. 

We will consider the first case carefully; the second case 
can be considered  similarly. 
On the event 
\[
  \{\widehat S_i(n+d-\pi(i))\in dz_{\pi(i)},
  i=1,\ldots, d, \widehat \tau<\infty\}   
\]
let $i$ be the smallest integer for which 
$(\widehat S_i)$ has a non-empty intersection with another path. 
Let $A$ be the first vertex, where this intersection happens 
and $i'>i$ be the smallest number 
corresponding to the path $(\widehat S_{i'})$, which 
intersected $(\widehat S_i)$.  
Denote as $O$ the second vertex corresponding to the 
path $i'$ and as $B$ the second vertex corresponding 
to the path $i$, see Figure~\ref{fig:switch.overshoot}.  
\begin{figure}
  \centering 
   \begin{tikzpicture}[scale = 1]
   \node at (0, 1) {\textbullet};
   \node at (1, 0) {\textbullet};
    \draw(0, 1) -- (1, 1) -- (1, 2) -- (2, 2) -- (2,2.5) -- (3,2.5);
    \draw(1, 0) -- (2, 0) -- (2, 4) -- (3, 4) -- (3,5) -- (4,5);
     \node at (3, 2.5) {\textbullet};
     \node at (2, 2) {\textbullet};
     \node at (2.3, 2){$A$};
     \node at (2, 2.5) {\textbullet};
     \node at (1.7, 2.5){$O$};
     \node at (2, 4) {\textbullet};
     \node at (1.7, 4){$B$};
   \node at (4, 5) {\textbullet};
   \node at (-0.5, 1){$x_{i'}$};
    \node at (0.5, 0){$x_i$};
      \node at (3, 3){$z_{\pi(i')}$};
       \node at (4.5, 5){$z_{\pi(i)}$};
    \end{tikzpicture}
  \caption{Construction of one-to-one correspondence in the alternative proof of Proposition \ref{transition_density}. 
  \label{fig:switch.overshoot}
%  This Figure and Figure \ref{non_int_paths} give two slightly different mappings to non-intersecting paths which can both be used to prove Proposition \ref{transition_density}.
}
  \end{figure}
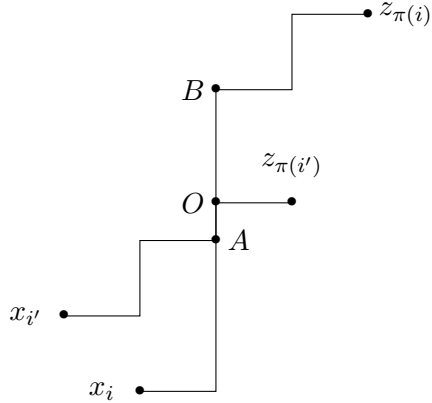

Then $|AB|$ is an overshoot 
of  random walk, which has 
exponential distribution in view of the 
memoryless property of the exponential distribution 
and is independent of anything else. 
$|OA|$ also has an exponential distribution 
independent  
of anything else. 
Hence we can swap the trajectories of 
the paths $i$ and $i'$ after point $A$ without 
affecting the distribution. 
This gives a one-to-one correspondence between 
$\pi$ and $\pi'$ with $i$ and $i'$ permuted. 
As $\text{sgn}(\pi)=-\text{sgn}(\pi')$
this implies~\eqref{eq6}.  
\end{proof}

%
%We now turn to uniform bounds for tails and 
%transition densities.
\begin{proposition}\label{prop.uniform.bounds}
  Let $\lambda_1=\cdots=\lambda_d=1$. 
  Let $x = (x_1, \ldots, x_d)\in W^d $.  
  \begin{enumerate}[(i)] 
  \item There exists a constant $C_d$ such that for $n\ge 2d$, 
  \begin{align}
    \label{eq.uniform.tau}\P_x(\tau>n) \le C_d\frac{h(x)}{n^{d(d-1)/4}}\\ 
    \label{eq.uniform.rho} \P_x(\rho>n) \le C_d\frac{\Delta(x)}{n^{d(d-1)/4}}.
  \end{align}  
  \item In addition, let $z = (z_1, \ldots, z_d)\in W^d.$ 
  There exists a constant 
  $C_d$ that does not depend on $x$ and $z$ such that 
  for $n\ge 2d$,  
  \begin{align} 
    \label{eq.uniform.G} G_{n+d}(x,z)\le  
  \frac{C_d}{n^{d^2/2}}
 h(x) \hat{h}(z), \\ 
 \label{eq.uniform.G.tilde} \widetilde G_{n}(x,z)\le  
  \frac{C_d}{n^{d^2/2}}
 \Delta(x) \Delta(-z). 
  \end{align}
\end{enumerate}
\end{proposition} 
%Note that \eqref{eq.uniform.rho} and \eqref{eq.uniform.G.tilde} are clear at the boundary so we assume in these cases that $x, z \in \text{int}(W^d)$.
We prove this by a sequence of Lemmas 
and start with part (ii). 
In view of~\eqref{eq.gn.second} 
 to estimate $G_n(x,z)$ 
it is sufficient to  estimate $\widetilde G_{n}(x,z)$. 
Let 
\[
\varphi(\theta):=
\frac{\lambda}{\lambda-i\theta}
\]
be the characteristic  function 
of $\Gamma(1,\lambda)$ distribution. 
 We have the following representation for $\widetilde G_n(x,y).$ 
\begin{lemma}\label{lem:g.via.characteristic} 
  Let $\lambda_1=\cdots=\lambda_d=\lambda$. 
  For any $x,y \in W^d$, 
  \begin{multline*}
    \widetilde G_{n}(x,y) = 
    \left(\frac{1}{2\pi}\right)^{d}
  \int_{W^d} 
  \mathrm{det}
\left(
        e^{-i\theta_{j}y_{k}}
 \right)_{j, k = 1}^d
   \mathrm{det}
  \left(
        e^{i\theta_j x_{k}}
 \right)_{j, k = 1}^d
   \prod_{k=1}^d  (\varphi(\theta_k))^n d\theta_k.
\end{multline*}
\end{lemma}  
\begin{proof} 
  Using the inversion formula for characteristic functions we obtain 
  \begin{align*}
    \widetilde G_n(x,y) 
    %&= 
   % \sum_{\sigma}(-1)^{\text{sign}(\sigma)} \prod_{i=1}^d 
   % f_n(y_{\sigma(i)}-x_{i})\\
%    &=\left(\frac{1}{2\pi}\right)^{d/2}\int_{\mathbb{R}^d}
%    \sum_{\sigma}(-1)^{\text{sgn}(\sigma)}
%    \prod_{i=1}^d e^{-i\theta_i(y_{\sigma(i)}-x_{i})}
%    (\varphi(\theta_i))^n d\theta_i \\
    &= 
    \left(\frac{1}{2\pi}\right)^{d}\int_{\mathbb{R}^d}
    \mbox {det}
 \left(
          e^{-i\theta_j(y_{k}-x_{j})}
   \right)_{j, k = 1}^d
     \prod_{j=1}^d  (\varphi(\theta_j))^n d\theta_j. 
  \end{align*} 
 %Then the Andr{\'e}ief identity gives the stated formula. 
Using the standard properties of the determinant we can write 
\begin{align*} 
  \widetilde G_{n}(x,y) &= 
  \left(\frac{1}{2\pi}\right)^{d}
  \int_{\mathbb{R}^d}
  \text{det}
       {\left(
        e^{-i\theta_jy_{k}}
 \right)}_{j,k=1..d}
   e^{i\sum_{j=1}^d\theta_j x_{j}}
   \prod_{k=1}^d  (\varphi(\theta_k))^n d\theta_k. 
\end{align*} 
Next we split the $d$-dimensional cube to obtain that  $\widetilde G_{n}(x,y)$ equals
\begin{align*} 
 &
  \left(\frac{1}{2\pi}\right)^{d}
  \sum_{\sigma}  
  \int_{\theta_{\sigma(1)}<\ldots<\theta_{\sigma{(d)}}} 
  \mbox {det}
\left(
        e^{-i\theta_jy_{k}}
 \right)_{j, k = 1}^d
   e^{\sum_{j=1}^di\theta_j x_{j}}
   \prod_{k=1}^d  (\varphi(\theta_k))^n d\theta_k \\ 
   &= 
  \left(\frac{1}{2\pi}\right)^{d}
  \sum_{\sigma}  
  \int_{\theta_1<\ldots<\theta_d} 
  \mbox {det}
   \left(
        e^{-i\theta_{\sigma(j)}y_{k}}
 \right)_{j, k = 1}^d
   e^{i\sum_{j=1}^d\theta_{\sigma(j)} x_{j}}
   \prod_{k=1}^d  (\varphi(\theta_k))^n d\theta_k\\ 
   &= 
  \left(\frac{1}{2\pi}\right)^{d}
  \sum_{\sigma}  
  (-1)^\sigma
  \int_{\theta_1<\ldots<\theta_d} 
  \mbox {det}
  \left(
        e^{-i\theta_{j}y_{k}}
 \right)_{j, k = 1}^d
   e^{i\sum_{j=1}^d\theta_{\sigma(j)} x_{j}}
   \prod_{k=1}^d  (\varphi(\theta_k))^n d\theta_k\\ 
   &= 
  \left(\frac{1}{2\pi}\right)^{d}
  \int_{\theta_1<\ldots<\theta_d} 
  \mbox {det}
       \left(
        e^{-i\theta_{j}y_{k}}
 \right)_{j, k = 1}^d
   \mbox {det}
       \left(
        e^{i\theta_j x_{k}}
 \right)_{j, k = 1}^d
   \prod_{k=1}^d  (\varphi(\theta_k))^n 
   d\theta_k. \qedhere
\end{align*} 
\end{proof}  

%Next we have the following estimate. 
\begin{lemma}\label{lem:exp.vandermonde}
  For any real $x_1, \ldots, x_d$ and $\theta_1<\ldots<\theta_d$ we have 
  \[
    \left|\mathrm{det}
        {\left(
         e^{-i\theta_j x_k}
  \right)}_{j,k=1}^{d}\right| 
    \le C_d \Delta(\theta)\Delta(x). 
    \]
\end{lemma}  
\begin{proof} 
  The proof follows by observing that 
  combination of formulae~(3.2) and (3.4) in~\cite{shatashvili_93} gives a representation 
as a product of Vandermonde determinant $\Delta(i\theta)$ and 
an integral over the Gelfand-Tsetlin polytope. Then noting that the integrand is bounded we arrive at the conclusion. 
\end{proof} 
\begin{lemma}\label{lem:pre.conc2}
  Let $\lambda_1=\cdots=\lambda_d=1$. 
  There exists a constant $C_d$ such that  
  \[
    \widetilde G_{n}(x,y)\le 
    C_d  \frac{\Delta(x)\Delta(y)}{n^{d^2/2}}, 
    \quad x,y\in W^d, n\ge 2d. 
  \]
\end{lemma}  
\begin{proof} 
  Combining Lemma~\ref{lem:g.via.characteristic} and Lemma~\ref{lem:exp.vandermonde} we 
  obtain that 
  \begin{align*} 
  \widetilde G_{n}(x,y) & \le 
  C_d 
  \Delta(x)\Delta(y)
  \int_{\theta_1<\ldots<\theta_d} 
\Delta(\theta)^2 
   \prod_{j=1}^d  \left|\varphi(\theta_j)\right|^n d\theta_j\\ 
   & =
   C_d \frac{\Delta(x)\Delta(y)}{n^{d^2/2}}
   \int_{\theta_1<\ldots<\theta_d} 
\Delta(\theta)^2 
   \prod_{j=1}^d  \left|\varphi\left(\frac{\theta_j}{\sqrt n}\right)\right|^n d\theta_j
   \\
   & =
   C_d \frac{\Delta(x)\Delta(y)}{n^{d^2/2}}
   \int_{\theta_1<\ldots<\theta_d} 
\Delta(\theta)^2 
   \prod_{j=1}^d  
   \frac{1}{(1+\theta_j^2/n)^{n/2}}
   d\theta_j.
  \end{align*} 
  Here and in the rest of the proof $C_d$ denotes constants which might change from line to line. 
  %where we used 
  %the substitution $\theta_j = \frac{\theta_j'}{\sqrt n}$ and 
  %the equality  
  %\[
  %|\phi(\theta)| 
  %=\frac{1}{|1-i\theta/\lambda|} 
  %=\left(1+\frac{\theta^2}{\lambda^2}\right)^{-1/2}.
  %\]
  Analysis of the integral shows that it is uniformly bounded. 
  %Case $n=1$ can be analysed directly similarly to Lemma~\ref{lem:exp.vandermonde}.    
  Indeed, first  note that 
  \[
    \Delta(\theta) 
    \le C_d 
    \left(\prod_{j=1}^{d}\max(|\theta_j|,1)\right)^{d-1}. 
  \]
  Then, the integral is bounded by 
  \begin{multline*} 
    \int_{\theta_1<\ldots<\theta_d} 
    \Delta(\theta)^2 
       \prod_{j=1}^d  
       \frac{1}{(1+\theta_j^2/n)^{n/2}}
       d\theta_j
       \le
       C_d 
       \int 
       \prod_{j=1}^d  
       \frac{\max(|\theta_j|,1)^{2d-2}}{(1+\theta_j^2/n)^{n/2}}
       d\theta_j
      \\
      = 
      C_d
      \left(\int_{-\infty}^\infty  
      \frac{\max(|\theta|,1)^{2d-2}}{(1+\theta^2/n)^{n/2}}
      d\theta 
      \right)^d 
      \le 
      2^dC_d 
      \left(
        1+\int_{1}^\infty 
        \frac{\theta^{2d-2}}{(1+\theta^2/n)^{n/2}} 
      d\theta 
      \right)^d
   \end{multline*}
  Next we make use of the inequality 
  $\ln (1+t)\ge t-t^2, 
t>-\frac{1}{2}$ to obtain  
\begin{multline*} 
  \int_{1}^{\sqrt {n/2}} 
        \frac{\theta^{2d-2}}{(1+\theta^2/n)^{n/2}} 
      d\theta 
      \le 
      \int_{1}^{\sqrt {n/2}} 
      \theta^{2d-2}\exp\left(-\frac{n}{2}\ln\left(1+\frac{\theta^2}{ n}\right)\right)d\theta \\ 
      \le 
      \int_{1}^{\sqrt {n/2}} 
      \theta^{2d-2} 
      \exp\left(-\frac{\theta^2}{2}
      +\frac{\theta^4}{2n}
      \right)
      d\theta
      \le 
      \int_{1}^{\infty} 
      \theta^{2d-2} 
      \exp\left(-\frac{\theta^2}{4}
      \right)
      d\theta
\end{multline*}  
Next, we estimate  the remaining part of the integral 
\[ 
  \int_{ \sqrt {n/2}}^\infty 
        \frac{\theta^{2d-2}}{(1+\theta^2/n)^{n/2}}
      d\theta 
     =  
 n^{d-1/2} 
  \int_{\sqrt {1/2}}^\infty 
        \frac{\theta^{2d-2}}{(1+\theta^2)^{n/2}}d\theta. 
  \] 
We can further  estimate 
\begin{multline*}
  n^{d-1/2} 
  \int_{\sqrt {1/2}}^\infty 
        \frac{\theta^{2d-2}}{(1+\theta^2)^{n/2}}d\theta
        \le 
        \frac{n^{d-1/2}}{(3/2)^{n/2-d+1/4}} 
  \int_{\sqrt {1/2}}^\infty 
              \frac{\theta^{2d-2}}{(1+\theta^2)^{d-1/4}}d\theta\\ 
              \le 
              \frac{n^{d-1/2}}{(3/2)^{n/2-d+1/4}} 
  \int_{\sqrt {1/2}}^\infty  
  \frac{d\theta}{\theta^{3/2}},
\end{multline*}
which is uniformly (in $n$) bounded. 
\end{proof}  

 \begin{proof}[Proof of Proposition~\ref{prop.uniform.bounds}~(ii)] 
The required uniform bound 
for $\widetilde G_n(x,y)$ is contained in Lemma~\ref{lem:pre.conc2}.  
Then using~\eqref{eq.gn.second} we obtain 
  \[
     G_{n+d-1}(x,z)\le  
     \frac{C_d}{n^{d^2/2}}
     \E \Delta(x_1+\eta_1,\ldots,x_d+\eta_d)
     \E \Delta(z_1-\chi_1,\ldots,z_d-\chi_d). 
\]
This proves~\eqref{eq.uniform.G} 
by using that
\begin{align*}
& \E[\Delta(z_1-\chi_1,\ldots,z_d-\chi_d)] = \E[\Delta(-z_d+\chi_d,\ldots,-z_1+\chi_1)] \\
& \quad =
h(-z_d, \ldots, -z_1) = \hat{h}(z). \qedhere
\end{align*}

\end{proof}

\begin{lemma}\label{lem.large.deviations}
  Let $\lambda_1=\cdots=\lambda_d=1$.  
  \begin{enumerate}[(i)]
  \item
  Then, for $x,y\in W^d$ and $n\ge 2d$, 
  \begin{align*} 
    \widetilde G_{n}(x,y) &\le C_d 
    \mathrm{e}^{-\frac{\left|\sum_{i=1}^d(y_i-x_i)-dn\right|}{\sqrt n}}
    \frac{\Delta(x)\Delta(y)}{n^{d^2/2}}\\ 
    G_{n}(x,y) &\le C_d 
    \mathrm{e}^{-\frac{\left|\sum_{i=1}^d(y_i-x_i)-dn\right|}{\sqrt n}}
    \frac{h(x)\hat{h}(y)}{n^{d^2/2}}.
\end{align*}
\item
  If, in addition, 
  $\max_{j}(y_j-y_{j-1})\le n^{1/2}$ 
  and $\max_{j}(x_j-x_{j-1})\le n^{1/2}$ then  
  \begin{align*} 
    \widetilde G_{n}(x,y) &\le C_d 
    \mathrm{e}^{-d\frac{|y_1-x_1-n|}{\sqrt n}}
    \frac{\Delta(x)\Delta(y)}{n^{d^2/2}}\\ 
    G_{n}(x,y) &\le C_d 
    \mathrm{e}^{-d\frac{|y_1-x_1-n|}{\sqrt n}}
    \frac{h(x)\hat{h}(y)}{n^{d^2/2}}.
  \end{align*} 
  \end{enumerate}
\end{lemma}  
\begin{proof}
  Fix $\lambda>0$. We will start with the change of measure. 
  Let $f_n^{(\lambda)}$ be the density of the 
  $\Gamma(n,\lambda)$ distribution 
  and let $\widetilde G_{n}^{(\lambda)}(x,y) 
  =\mathrm{det}
  \left(
       f_n^{(\lambda)}(y_j-x_i)
  \right)_{i, j = 1}^d $. 
  We have, for $\lambda>-1$,  
  \begin{align*} 
  \widetilde G_{n}(x,y) &= 
\mathrm{det}
\left(
     f_n^{(1)}(y_j-x_i)
\right)_{i, j = 1}^d\\
&=
\mathrm{det}
\left(
\mathrm{e}^{\lambda(y_j-x_i)}  
(1+\lambda)^{-n}
f_n^{(1+\lambda)}(y_j-x_i)
\right)_{i, j = 1}^d\\
&= 
\mathrm{e}^{\lambda\sum_{i=1}^d (y_i-x_i)-dn\ln(1+\lambda)}
\widetilde G_{n}^{(1+\lambda)}(x,y).
\end{align*}
%Note that 
%\[
%  \lambda\sum_{i=1}^d (y_i-x_i)
%  =\lambda d(y_1-x_1) 
%+\lambda\sum_{i=2}^d ((y_i-y_{1})-(x_i-x_1))
%\le \lambda d(y_1-x_1) +|\lambda| d^2\sqrt n.
%\]
Now  
make use of the inequality $\ln (1+\lambda)\ge \lambda-\lambda^2, 
\lambda>-\frac{1}{2}$ 
 to obtain  
\begin{align*}
  \mathrm{e}^{\lambda\sum_{i=1}^d (y_i-x_i)-dn\ln(1+\lambda)} 
  \le 
  \mathrm{e}^{ \lambda\sum_{i=1}^d (y_i-x_i)-\lambda dn  + 
   +nd\frac{\lambda^2}{2}}
  \le C_d \mathrm{e}^{-\frac{\left|\sum_{i=1}^d (y_i-x_i)-dn\right|}{\sqrt n}} 
\end{align*}  
after we put 
$\lambda=-\frac{1}{\sqrt n}$ 
when 
$\sum_{i=1}^d (y_i-x_i)>dn$ 
and 
$\lambda=\frac{1}{\sqrt n}$ 
when 
$\sum_{i=1}^d (y_i-x_i)\le dn$.  

 Using this bound and the uniform bound for 
$\widetilde G_{n}^{(1+\lambda)}(x,y)$ from Lemma~\ref{lem:pre.conc2} 
we arrive at the conclusion. The same argument holds for $G_n$.

To check the second statement it is sufficient 
to note that 
\begin{align*}
  \sum_{i=1}^d |y_i-x_i - n|
  &\le d|y_1-x_1-n|   
+\sum_{i=2}^d |(y_i-y_{1})-(x_i-x_1)|\\ 
&\le d|y_1-x_1-n|  
+
\sum_{i=2}^d \left(|y_i-y_{1}|+|x_i-x_1|\right)\\ 
&\le d|y_1-x_1-n|   
+ 2\sqrt n\sum_{i=2}^d(i-1) \\ 
&=d|y_1-x_1-n|   + d(d-1)\sqrt n. 
\end{align*}
The rest of the proof can be done in exactly the same way. 
\end{proof}

%\begin{lemma}\label{lem.uniform.bounds}
%  There exists a constant $C$ such that~\eqref{eq.uniform.tau} and~\eqref{eq.uniform.rho} hold. 
%\end{lemma}  
\begin{proof}[Proof of Proposition \ref{prop.uniform.bounds}~(i)]

We will proceed by induction. 
For $d=2$ we can argue similarly to Lemma~25 in~\cite{denisov_sakhanenko_wachtel18} 
or use directly the exact formula  for $\P_x(\rho>n)$ given in 
Lemma~\ref{lem:rho.2}.  

Assume now that the statement~\eqref{eq.uniform.rho} holds 
for values of $j\le d$ and prove it for $d+1$. 
We first consider the case $\max_j(x_{j}-x_{j-1})\le n^{1/2}$. 
By the total probability formula  
\begin{align*}
\P_x(\rho>n) &=   
\int_{W^d} \P_x(\rho>[n/2], S_{[n/2]}\in dy ) 
\P_y(\rho>n-[n/2])\\
&\le 
\int_{W^d\cap \{\max_j (y_j-y_{j-1})\le \sqrt n\}} \P_x(\rho>[n/2], S_{[n/2]}\in dy ) 
\P_y(\rho>n-[n/2])\\
&+\sum_{j=2}^d 
\int_{W^d\cap \{ (y_j-y_{j-1})> \sqrt n\} } \P_x(\rho>[n/2], S_{[n/2]}\in dy ) 
\P_y(\rho>n-[n/2])\\
&=: P_1+\sum_{j=2}^d P_j.
\end{align*} 
We will split the first probability in $2$ parts, $P_1\le P_{11}+P_{12}$, 
where 
\begin{align*}
  P_{11}&:=  
  \int_{W^d\cap \{\max_j (y_j-y_{j-1})\le \sqrt n, |y_1-x_1-n|\le \sqrt n\}} \P_x(\rho>[n/2], S_{[n/2]}\in dy ) 
\P_y(\rho>n-[n/2])\\
P_{12}&:=  
\int_{W^d\cap \{\max_j (y_j-y_{j-1})\le \sqrt n, |y_1-x_1-n|>\sqrt n\}} \P_x(\rho>[n/2], S_{[n/2]}\in dy ) 
\P_y(\rho>n-[n/2])
\end{align*}

For the first probability it follows from the 
definition~\eqref{eq.km.rho} of $\widetilde G_n(x,z)$ and 
the uniform bound in Lemma~\ref{lem:pre.conc2}, 
\begin{align*}
  P_{11}&
  \le 
  \int_{W^d\cap \{\max_j (y_j-y_{j-1})\le \sqrt n,|y_1-x_1-n|\le \sqrt n\}} dy
  \widetilde G_{[n/2]}(x,y)\\ 
  &\le  
  \frac{C\Delta(x)}{n^{d^2/2}}
  \int_{W^d\cap \{\max_j (y_j-y_{j-1})\le \sqrt n,|y_1-x_1-n|\le \sqrt n\}} dy\Delta(y) \\ 
  &\le 
  \frac{C\Delta(x)}{n^{d^2/2}} 
  \int_{W^d\cap \{\max_j (y_j-y_{j-1})\le \sqrt n,|y_1-x_1-n|\le \sqrt n\}} dy 
  \prod_{1\le k<l\le d} 
  \left(
    (l-k)n^{1/2} 
  \right)
  \\
  &\le 
  \frac{C\Delta(x)}{n^{d^2/2}} 
  n^{\frac{d(d-1)}{4}} 
  \int_{W^d\cap \{\max_j (y_j-y_{j-1})\le \sqrt n\},|y_1-x_1-n|\le \sqrt n\}} dy
  \le \frac{2C\Delta(x)}{n^{\frac{d(d-1)}{4}}}
\end{align*}  
since 
\begin{multline*} 
  \int_{W^d\cap \{\max_j (y_j-y_{j-1})\le \sqrt n\},|y_1-x_1-n|\le \sqrt n} dy 
  \\
  \le \int_{x_1+n-\sqrt n}^{x_1+n+\sqrt n}
  \int_{y_1}^{y_1+\sqrt n}  
  \ldots 
  \int_{y_{d-1}}^{y_{d-1}+\sqrt n} dy_d\ldots dy_2dy_1  
  \le 2n^{d/2}.  
\end{multline*}  
To analyse $P_{12}$ we apply Lemma~\ref{lem.large.deviations} to obtain 
\begin{align*}
P_{12}&\le 
\frac{C\Delta(x)}{n^{\frac{d^2}{2}}}
  \int_{W^d\cap \{\max_j (y_j-y_{j-1})\le \sqrt n,|y_1-x_1-n|>\sqrt n\}} dy
  \mathrm{e}^{-d\frac{|y_1-x_1-n|}{\sqrt{[n/2]}}}
  \Delta(y) \\ 
  &\le \frac{C\Delta(x)}{n^{\frac{d(d-1)}{4}}} 
\end{align*} 
since 
\begin{multline*} 
  \int_{W^d\cap \{\max_j (y_j-y_{j-1})\le \sqrt n\},y_1-x_1-n>\sqrt n} 
  \mathrm{e}^{-d\frac{|y_1-x_1-n|}{\sqrt{[n/2]}}}
  dy 
  \\
  \le \int_{x_1+n+\sqrt n}^{\infty}
  \int_{y_1}^{y_1+\sqrt n}  
  \ldots 
  \int_{y_{d-1}}^{y_{d-1}+\sqrt n} 
  \mathrm{e}^{-d\frac{|y_1-x_1-n|}{\sqrt{[n/2]}}}
  dy_d\ldots dy_2dy_1  \\ 
  \le n^{(d-1)/2}
  \int_{x_1+n+\sqrt n}^{\infty}
  \mathrm{e}^{-d\frac{y_1-x_1-n}{\sqrt{[n/2]}}}dy_1 
  =
  n^{(d-1)/2}
  \int_{\sqrt n}^{\infty}
  \mathrm{e}^{-d\frac{y_1}{\sqrt{[n/2]}}}dy_1 
  \le n^{d/2}  
\end{multline*}
and, symmetrically, 
\[
  \int_{W^d\cap \{\max_j (y_j-y_{j-1})\le \sqrt n\},y_1-x_1-n<-\sqrt n} 
  \mathrm{e}^{-d\frac{|y_1-x_1-n|}{\sqrt{[n/2]}}}
  dy 
  \le n^{d/2}. 
\]  

To show the bound for other terms we analyse more carefully $P_d$, as it is notationally 
easier. 
Denote $y_{[i,j]} = (y_i,\ldots,y_j)$ 
and $\rho_{k}$ the stopping time $\rho$ 
corresponding to the Weyl Chamber $W^{k}$.  
We have, using induction and the Chebyshev inequality,
\begin{align*}
  P_d&\le  
  \int_{W^d\cap \{ (y_d-y_{d-1})> \sqrt n\} } \P_x(\rho>[n/2], S_{[n/2]}\in dy )
  \P_{y_{[1,\ldots,d-1]}}(\rho_{d-1}>n-[n/2])\\ 
  &\le 
  C\int_{W^d\cap \{ (y_d-y_{d-1})> \sqrt n\} } \P_x(\rho>[n/2], S_{[n/2]}\in dy )
  \frac{\Delta(y_{[1,\ldots,d-1]})}{n^{\frac{(d-1)(d-2)}{4}}}\\ 
&\le C 
\int_{W^d\cap \{ (y_d-y_{d-1})> \sqrt n\} } \P_x(\rho>[n/2], S_{[n/2]}\in dy )
\frac{\Delta(y_{[1,\ldots,d-1]})}{n^{\frac{(d-1)(d-2)}{4}}}
\frac{\prod_{j=1}^{d-1}(y_d-y_j)}{n^{(d-1)/2}}
\\
  &\le C 
\int_{W^d} \P_x(\rho>[n/2], S_{[n/2]}\in dy )
  \frac{\Delta(y)}{n^{\frac{d(d-1)}{4}}}
  =C
  \frac{\E_x [\Delta(S_{[n/2]});\rho>[n/2]]}{n^{\frac{d(d-1)}{4}}}
  \\
 &=C \frac{\Delta(x)}{n^{\frac{d(d-1)}{4}}},
\end{align*}   
where we used the harmonicity of $\Delta $ at the last step.
Other terms $P_j$ are analysed similarly using the bound 
\[
  \P_{y}(\rho_{d-1}>n-[n/2])
  \le 
  \P_{(y_{[1,j-1]})}(\rho_{j-1}>n-[n/2])
  \P_{(y_{[j,\ldots,d]}}(\rho_{d-j+1}>n-[n/2]).
\]

We are left to consider the case $\max_j(x_{j}-x_{j-1})>n^{1/2}$. 
Here, we can proceed similarly to the above. 
Suppose that $(x_d-x_{d-1})>\sqrt n$. Then, 
by the induction assumption, 
\[
  \P_{x}(\rho_{d}>n)
  \le 
  \P_{x_{[1,\ldots,d-1]}}(\rho_{d-1}>n)
  \le C \frac{\Delta(x_{[1,\ldots,d-1]})}{n^{\frac{(d-1)(d-2)}{4}}} 
  \le C \frac{\Delta(x)}{n^{\frac{d(d-1)}{4}}}. 
\] 
The other cases can be considered similarly. The proof of the uniform bound for $\tau$ can be done in a similar way 
or proved using the coupling 
between interlaced and ordered random walks 
discussed in subsection~\ref{coupling}. 
\end{proof}

\section{Tail asymptotics}
\label{sec:tail_asy}

%\textcolor{blue}{
%Tail asymptotics can be computed from the large $n$ limit of the integrals 
%\begin{align*}
%\P(\tau > n) & = \int_{W^d} \det(q_{n+i-j}(y_j - x_i))_{i, j = 1}^d \prod_{i=1}^d \lambda_i^n e^{-\sum_{i=1}^d \lambda_i(y_i - x_i)} dy_1 \ldots dy_d, \\
%\P(\rho > n) & = \int_{W^d} \det(q_{n}(y_j - x_i))_{i, j = 1}^d \prod_{i=1}^d \lambda_i^n e^{-\sum_{i=1}^d \lambda_i(y_i - x_i)} dy_1 \ldots dy_d.
%%& = \int_{W^d} \det \left(\frac{(y_j - x_i)^{n-1+i-j}}{(n-1+i - j)!} 1_{\{y_j \geq x_i\}}\right)_{i, j = 1}^d \prod_{i=1}^d \lambda_i^n e^{-\sum_{i=1}^d \lambda_i(y_i - x_i)} dy_1 \ldots dy_d.
%\end{align*}
%We consider this in three different cases: (A) $\lambda_1 > \ldots > \lambda_d$, (B) $\lambda_1 = \ldots = \lambda_d = 1$
%and (C) $\lambda_1 < \ldots < \lambda_d$. Other cases of inequalities between $\lambda_1, \ldots, \lambda_d$ could be treated by similar techniques. 
%In some cases, we will first find asymptotics for $\rho$ and then use the coupling in Section \ref{subsec_coupling} to extend to $\tau$. 
%}

\subsection{Proof of Theorem \ref{tail_asy} for $\lambda_1 > \ldots > \lambda_d.$}
By integrating the formula from Proposition \ref{transition_density},
\[
\P_x(\tau > n) = \int_{W^d} \left( \prod_{j=1}^d \lambda_j^n \right) e^{-\sum_{j=1}^d \lambda_j(y_j - x_j)}  \det(q_{n+i-j}(y_j - x_i))_{i, j = 1}^d 
dy_1\ldots dy_d.
\]
Change variables $\lambda_j y_j = n + \sqrt{n} z_j$ for each $j = 1, \ldots, d$ and apply Stirling's formula to obtain the large $n$ asymptotics,
\begin{align*}
& \P_x(\tau > n)\\
& \sim (2\pi)^{-d/2} \int_{\mathbb{R}^d} \det \left(\lambda_j^{j-i} e^{-\sqrt{n} z_j}(1+z_j/\sqrt{n} - x_i \lambda_j/n)^{n-1+i-j}\right)_{i, j = 1}^d e^{\sum_{i=1}^d \lambda_i x_i} dz_1 \ldots dz_d \\
& \sim(2\pi)^{-d/2}  e^{\sum_{i=1}^d \lambda_i x_i} \det(\lambda_j^{j-i} e^{-x_i \lambda_j})_{i, j = 1}^d \int_{\mathbb{R}^d} e^{-\sum_{j=1}^d z_j^2/2} dz_1 \ldots dz_d \\
& =  e^{\sum_{i=1}^d \lambda_i x_i} \det(\lambda_j^{j-i} e^{-x_i \lambda_j})_{i, j = 1}^d \\
& = h(x).
\end{align*}

%This argument can also be used to prove local limit theorems. 
%\begin{align*}
%\P_x(n^{-1/2}S_n \in dy, \rho > n)  & \sim (2\pi)^{-d/2} e^{-\sum_{j=1}^d y_j^2} h(x), \quad n \rightarrow \infty \\
%\P_x(n^{-1/2}S_n \in dy, \tau > n)  & \sim (2\pi)^{-d/2} \prod_{j = 1}^d \lambda_j^{j-1} e^{-\sum_{j=1}^d y_j^2} h(x), \quad n \rightarrow \infty. 
%\end{align*}

\subsection{Tail asymptotics for equal rates.} 
We set $\lambda_1 = \ldots = \lambda_d = 1$ and the general case can be recovered by scaling. 
We first consider the case $d = 2$. 
For $x = (x_1, x_2) \in W^2$ and $n \geq 1$ let
\begin{equation}\label{defn.p.2}
  p_{x_1,x_2}(n) =  (-1)^n
  \sum_{k=1}^\infty (-1)^{k+1}\binom{k/2 -1}{n}
\frac{(x_2 - x_1)^k}{k!}.  
\end{equation}
We extend the definition to all of $\mathbb{R}^2$ by antisymmetry $p_{x_1, x_2}(n) = -p_{x_2, x_1}(n)$.  
\begin{lemma}\label{lem:rho.2}
For $x=(x_1,x_2)\in W^2$ and $n \geq 1$,
  \begin{equation}\label{eq.exact.d.2}
  \pr_x(\rho>n+1)  = p_{x_1,x_2} (n). 
  \end{equation}
  Moreover, for any fixed $N\ge 1$ and $C>0$,  uniformly in 
  $x\in W^2$ with $x_2-x_1\le C\sqrt{n}$,  
  the following asymptotic expansion is valid,
  \begin{equation*}
  %\label{eq.asym.d.2}
    \left|\pr_x(\rho>n+1) -  (-1)^n \sum_{k=0}^{N-1}  \binom{k-1/2}{n}
    %{k-1/2 \choose  n} 
\frac{(x_2 - x_1)^{2k+1}}{(2k+1)!} 
    \right| 
\le C_0 \left(\frac{(x_2 - x_1)^{2N + 1}}{n^{N + 1/2}}\right)  
  \end{equation*}
  for some $C_0<\infty$. 
\end{lemma}
\begin{proof}
For  $s: 0<s<1$   consider the following 
sequence 
$ (s^ne^{-\sqrt{1-s} (S_2(n-1)-S_1(n))})_{n\ge 1}$, 
which forms a martingale  with respect to the filtration 
$\mathcal{F}_n = \sigma(S_2(0), \ldots, S_2(n-1), S_1(1), \ldots, S_1(n)).$  
In this case 
\[
\rho = \inf\{n\ge 1\colon S_1(n)>S_2(n-1)\}  
\]
and $\rho$ is a stopping time with respect to $\mathcal F_n$.  
An application of the optional stopping theorem gives 
$$
\E_x [s^\rho e^{-\sqrt{1-s} (S_2(\rho-1) -S_1(\rho) }] = 
s\E_x[e^{-\sqrt{1-s} (x_2-S_1(1))}].
$$
To justify the use of the optional stopping theorem note that 
$S_2(n-1)-S_1(n)\ge 0$ for $n<\rho$ and 
$S_1(\rho)-S_2(\rho-1)$ has an exponential distribution with parameter $1$ 
by the memoryless property of the exponential distribution. 
Hence, for $s:0<s<1$, 
\[ 
0\le s^{\rho \wedge n} e^{-\sqrt{1-s} (S_2(\rho\wedge n-1) -S_1(\rho \wedge n)) }
\le e^{-\sqrt{1-s}(S_2(\rho-1) -S_1(\rho))}, 
\] 
which is an integrable random variable. 

Using again the lack of memory of the exponential distribution 
we note that the overshoot  $S_1(\rho)-S_2(\rho-1) $ has exponential distribution with parameter $1$ and is
independent of $\rho$.  Therefore
\begin{align*}
  \E_x [s^{\rho}] 
  &= s e^{-\sqrt{1-s} (x_2-x_1)}.
\end{align*}
Then
\begin{align*}
  \sum_{n=0}^\infty
  s^n \pr_x(\rho>n+1) &= 
  \frac{1-\E_x s^{\rho-1}}{1-s}  
  =
  \sum_{k=1}^\infty  (-1)^{k+1}(1-s)^{k/2-1}\frac{(x_2 - x_1)^k}{k!}.
  \end{align*}
 Applying  the binomial theorem, 
\begin{align*}
  \sum_{n=0}^\infty
  s^n \pr_x(\rho>n+1)
%  &=
%\sum_{n=1}^\infty  (1-s)^{k/2-1}(-1)^{k+1}\frac{(x_2 - x_1)^k}{k!}
%  \\
&=\sum_{k=1}^\infty (-1)^{k+1}
\sum_{n=0}^\infty \binom{k/2-1}{n}(-1)^n s^n
\frac{(x_2 - x_1)^k}{k!}
  \\
&=\sum_{n=0}^\infty (-1)^n s^n 
  \sum_{k=1}^\infty (-1)^{k+1}  \binom{k/2-1}{n}
  \frac{(x_2 - x_1)^k}{k!}.
\end{align*}
Equating powers of $s$ gives~\eqref{eq.exact.d.2}. 

To obtain the asymptotic expansion
note first the representation 
\begin{multline*} 
  p_{x_1,x_2}(n) 
  =(-1)^n 
  \sum_{j=0}^{\infty} \binom{j-1/2}{n}
  \frac{(x_2-x_1)^{2j+1}}{(2j+1)!} \\ 
  - 
  (-1)^n  
  \sum_{j=n+1}^{\infty} \binom{j-1}{n}
  \frac{(x_2-x_1)^{2j}}{(2j)!}.  
\end{multline*}
Using the Stirling approximation we can estimate the second series 
and obtain the required bound. 
% as follows 
% \begin{multline*}
%   \sum_{j=n+1}^{\infty} \binom{j-1}{n}
%   \frac{(x_2-x_1)^{2j}}{(2j)!}
%   \le 
%   \frac{C}{n!e^n}
%   \sum_{j=n+1}^{\infty}  
%   \frac{(x_2-x_1)^{2j}}{4^j j^j 
%   (j-n)^{j-n-1/2}}\\
%   \le 
%   \frac{C(x_2-x_1)^{2n+2}}{n!e^n}
%   \sum_{j=0}^{\infty}  
%   \frac{(x_2-x_1)^{2j}}{4^{j+n+1} (j+n+1)^{j+n+1} 
%   (j+1)^{j+1/2}}\\ 
%   \le 
%   \frac{C(x_2-x_1)^{2n+2}}{n!e^n}
%   \sum_{j=0}^{\infty}  
%   \frac{C^{2j}}{4^{j+n+1} (j+n+1)^{n+1} 
%   (j+1)^{j+1/2}}
%   \le 
%   \frac{C(x_2-x_1)^{2n+2}}{n!e^n}. 
% \end{multline*}  
% The remainder of the first series can be estimated as follows 
% \begin{multline*}
%   \sum_{j=N}^{\infty} \binom{j-1/2}{n}
%   \frac{(x_2-x_1)^{2j+1}}{(2j+1)!}  
%   \le  
%   (x_2-x_1)^{2N+1} 
%   \biggl(
%   \sum_{j=N}^n 
%   \frac{j!(n-j)!}{n!(2j+1)!} C^j n^{j-N-1/2}
%   \\ 
%   +
%   \sum_{j=n+1}^\infty 
%   \frac{j!}{n!(2j+1)!(j-n)!} C^j n^{j-N-1/2}
%   \biggr).  
% \end{multline*}  
% Then, using the Stirling approximation, 
% \[ 
%   \sum_{j=N}^n 
%   \frac{j!(n-j)!}{n!(2j+1)!} C^j n^{j-N-1/2}
%   \le 
%   \frac{C}{n^{1/2}}  
%   \sum_{j=N}^n 
%   \frac{(Ce^2)^j}{n^nj^j} 
%    n^{j-N} 
%    \le 
%    \frac{C}{n^{1/2+N}}
%    \sum_{j=1}^\infty 
%    \left(
%     \frac{Ce^2}{j}
%    \right)^j
% \] 
% and 
% \begin{multline*}
%   \sum_{j=n+1}^\infty 
%   \frac{j!}{n!(2j+1)!(j-n)!} C^j n^{j-N-1/2} 
% \le \frac{C}{n^{N+1/2}}
% \sum_{j=n+1}^\infty  
% \frac{(4Ce^4)^j n^j }{n^n j^j(j-n)^{j-n}}. 
% \end{multline*}  
% As the last series is uniformly bounded in $n$ we arrive at the conclusion. 
\end{proof}

The first step in the analysis for general $d$ is an expression for $\P(\rho > n)$ as a Pfaffian. 
Let $A = (a_{ij})_{i, j = 1}^{2m}$ be a $2m \times 2m$ antisymmetric matrix. 
Let $\Pi_{2m}$ be the set of partitions of $\{1, \ldots, 2m\}$ with the property that $\sigma(2i-1) < \sigma(2i)$ for each
$i = 1, \ldots, m$ 
and $\sigma(1) < \sigma(3) < \ldots < \sigma(2m-1)$. 
Define the Pfaffian of $A$ to be
\[
\Pfaff(A) = \sum_{\sigma \in \Pi_{2m}} \text{sgn}(\sigma) \prod_{i=1}^m a_{\sigma(2i-1), \sigma(2i)}. 
\]
%There is a standard definition of the Pfaffian for matrices of even size while if the matrix is of odd size then there is no standard definition. 
%If $d$ is odd the definition that is convenient for us is the following. Let $A = (a_{ij})$ be a $(2m+1) \times (2m+1)$ antisymmetric matrix and define
%\begin{equation*}\label{eq.pf.recursive}
%\Pfaff(A) = 
%\sum_{l=1}^{2m+1} (-1)^{l+1}\Pfaff(a_{ij})_{i,j\in [2m+1]\setminus\{l\}}. 
%\end{equation*}
%This is also the definition used in \cite{doumerc_o_connell}. 

\begin{lemma}
\label{lem.pfaff}
For all $x \in W^d$ and $n, d \geq 1$
\[
\P_x(\rho > n) = 
\begin{cases}
\Pfaff(p_{x_i, x_j}(n-1))_{i, j = 1}^d & \text{ if $d$ is even,} \\
\sum_{l=1}^{d} (-1)^{l+1}\Pfaff(p_{x_i, x_j}(n-1))_{i, j \in [d-1]\setminus\{l\}} & \text{ if $d$ is odd.}
\end{cases}
\]
\end{lemma}

\begin{proof}
We first suppose that $d$ is even. The transition density in \eqref{eq.km.rho} can be integrated to give
\[
\P_x(\rho > n) = \int_{W^d} \det(f_{n}(y_j - x_i))_{i, j = 1}^d dy_1 \ldots dy_d.
\]
This can be expressed as a Pfaffian by using de Bruijn's integral formula \cite{bruijn}. The $(i, j)$ entry in the Pfaffian is given for $i < j$ and $x_j > x_i$ by
\[
\int_{\mathbb{R}^2} \text{sgn}(y_j-y_i) f_n(y_i - x_i) f_n(y_j - x_j) dy_i dy_j = 2\P_{(x_i, x_j)}(S_2(n) > S_1(n)) 
- 1.
\]
We have
for $x_i < x_j$,
\begin{align}
\label{refl.principle}
\P_{(x_i, x_j)}(\rho \leq n) 
% = \P_{(x_1, x_2)}(\rho \leq n, S_2(n) < S_1(n)) + \P_{(x_1, x_2)}(\rho \leq n, S_2(n) > S_1(n)) \nonumber \\
& = \P_{(x_i, x_j)}(S_2(n) < S_1(n), \rho \leq n) + \P_{(x_i, x_j)}(S_2(n) > S_1(n), \rho \leq n) \nonumber \\
& = \P_{(x_i, x_j)}(S_2(n) < S_1(n)) + \P_{(x_i, x_j)}(S_2(n) > S_1(n), \rho \leq n).
\end{align}
On the event $\{\rho \leq n\}$ the paths of $S_1$ and $S_2$ can be interchanged after the first time they intersect.
For $0 \leq k \leq n$, 
\begin{align*}
\hat{S}_1(k) = S_1(k) 1_{k < \rho} + S_2(k) 1_{k \geq \rho} \\
\hat{S}_2(k) = S_2(k) 1_{k < \rho} + S_1(k) 1_{k \geq \rho}. 
\end{align*}
Then $(S_1, S_2)$ has the same distribution as $(\hat{S}_1, \hat{S}_2)$ using the definition of $\rho$ and lack of memory of exponentials. Moreover $\{S_2(n) > S_1(n)\}$ is equivalent to $\{\hat{S}_2(n) < \hat{S}_1(n)\}$ on 
$\{\rho \leq n \}$. 
Using this in the second term of \eqref{refl.principle} gives
\[
\P_{(x_i, x_j)}(\rho \leq n)  = 2\P_{(x_i, x_j)}(S_2(n) < S_1(n)).
\]

This allows the entries in the Pfaffian to be rewritten in the stated form. For odd $d$, 
a version of the de Bruijn integration formula still holds \cite{bruijn} and gives the stated formula. 
Alternatively, we can add in an extra component to our random walk with starting position $x_{d+1}$, apply a Laplace expansion and let $x_{d+1} \rightarrow \infty$.
%\textcolor{blue}{ Since $d+1$ is even we have proved above that 
%\[
%\P_{(x_1, \ldots, x_{d+1})}(\rho > n)  = \Pfaff(p_{x_i, x_j}(n-1))_{i, j = 1}^{d+1}. 
%\]
%A Laplace expansion of the Pfaffian gives that
%\[
%\P_{(x_1, \ldots, x_{d+1})}(\rho > n)  = \sum_{l=1}^{d} (-1)^{l} p_{x_l, x_{d+1}}(n-1) \Pfaff(p_{x_i, x_j}(n-1))_{i,j\in [d]\setminus\{l\}}.
%\]
%Letting $x_{d+1} \rightarrow \infty$ gives the required result. 
%}
\end{proof}

\begin{lemma}
\label{lem.bin.asy}
For any $k \geq 0$ and any $N \geq 1$ there are coefficients $(a_{j}^{(k)})_{j \geq 0}$  such that
\[
(-1)^n \binom{k-1/2}{n} = \sum_{j=0}^{N-1} a_{j}^{(k)} (n+1)^{-k-1/2-j} + O((n+1)^{-N -1/2}).
\]
Furthermore for any $k \geq 0,$
\[
a_0^{(k)} = \frac{(-1)^k \Gamma(k + 1/2)}{\pi}.
\]
\end{lemma}

\begin{proof}
This is a consequence of an asymptotic expansion of a ratio of Gamma functions in \cite{erdelyi_tricomi}. 
The last paragraph in \cite{erdelyi_tricomi} gives the statement with the coefficient 
$a_0^{(k)} = 1/\Gamma(-k+1/2)$. This is equivalent to the expression for $a_0^{(k)}$ in the statement of the Lemma after using Euler's reflection formula 
\[
\Gamma(-k + 1/2) \Gamma(k+1/2) = (-1)^k \pi. \qedhere
\]
%\textcolor{blue}{We first note that Euler's reflection formula gives
%\[
%(-1)^n {k -1/2 \choose n} = \frac{(-1)^k \Gamma(k+1/2) \Gamma(n + 1/2 - k)}{\pi \Gamma(n+1)}.
%\]
%An asymptotic expansion of a ratio of Gamma functions is given in \cite{erdelyi_tricomi}. 
%One method is to use an integral expression for a ratio of Gamma functions which can be obtained by using two different interpretations of the Beta function:
%\[
%\frac{\Gamma(n + 1/2 - k)}{\Gamma(n +1)} = \frac{1}{\Gamma(1/2 + k)} \int_0^{\infty} e^{-(n + 1/2 - k) v}
%(1- e^{-v})^{k-1/2} dv.
%\]
%Watson's lemma then gives the stated asymptotic expansion. We refer to \cite{erdelyi_tricomi} for more details including expressions for the coefficients. }
%\[
%\frac{\Gamma(1/2 + n - k)}{\Gamma(n +1)} \sim \sum_{j=0}^{\infty} a_j^{(k)} (n+1)^{-k-1/2 - j}.
%\]
\end{proof}

Combining Lemma \ref{lem:rho.2} with Lemma \ref{lem.bin.asy} gives that for any $N$ there exist coefficients $(a_j^{(k)}: j, k = 0, \ldots, N)$ such that as $n \rightarrow \infty$,
 \begin{equation}
 \label{asy.exp.d.2}
    \pr_{(x_1, x_2)}(\rho>n) =  q_{x_1, x_2}(n)    + O\left(\frac{1+(x_2 - x_1)^{2N + 1}}{n^{N+1/2}}\right) 
  \end{equation}
where for any $(y, z) \in \mathbb{R}^2$
  \[
 q_{y, z}(n) = (-1)^n \sum_{k=0}^{N-1}  \sum_{j=0}^{N-1} a_j^{(k)} n^{-k-1/2 - j}\frac{(z - y)^{2k+1}}{(2k+1)!}.
  \]

%\textcolor{blue}{
%\begin{proposition}
%\label{tail.asy.rho}
%If $\lambda_1 = \ldots = \lambda_d = 1$ then uniformly for $x \in \text{int}(W^d)$ with $x_d -x_1 = o(\sqrt{n})$
%\[
%\P_x(\rho > n) \sim \mathfrak{X} \Delta(x) n^{-d(d-1)/4}, \quad n \rightarrow \infty
%\]
%with $\mathfrak{X}$ given in Theorem \ref{tail_asy}.
%\end{proposition} 
%}

\begin{proof}[Proof of Theorem \ref{tail_asy} part (ii)]
%Suppose that $d$ is even and let $l = d/2$. We have
%\begin{align*}
%\P_x(\rho > t) & = \frac{1}{2^d d!}\sum_{\sigma \in S_d} \mathrm{sgn}(\sigma) \prod_{i=1}^{l} \, \sum_{n=0}^{\infty} 
%(-1)^t {n - 1/2 \choose t} 
%\frac{(x_{\sigma(2i-1)} - x_{\sigma(2i)})^{2n+1}}{(2n+1)!} \\
%& =  \frac{1}{2^d d!} \sum_{\sigma \in S_d} \mathrm{sgn}(\sigma) \sum_{p \in \mathbb{N}^l} \prod_{i=1}^l  (-1)^t {p_i - 1/2 \choose t} 
%\frac{(x_{\sigma(2i-1)} - x_{\sigma(2i)})^{2p_i+1}}{(2p_i+1)!}.
%\end{align*}
We first suppose that $d$ is even and let $l = d/2$. Let $[N] = \{0, \ldots, N\}$. 
We use \eqref{asy.exp.d.2}, Lemma \ref{lem.pfaff},   
antisymmetry of $q_{x, y}$ 
and the fact that $q_{x,y}(n)$ is bounded for $|y-x|\le \sqrt n$ 
to obtain that 
\begin{equation}
\label{asy.d}
\P_x(\rho > n) = \Pfaff(q_{x_i, x_j}(n))_{i, j = 1}^d +  O\left((1+(x_d - x_1)^{2N + 1})n^{-N-1/2}\right).
\end{equation}
% has the asymptotic expansion
%\begin{align}
%\label{asy.d}
% \sum_{\sigma \in \Pi_d} \mathrm{sgn}(\sigma) \sum_{p \in [N-1]^l} \sum_{k \in [N-1]^l}
% \prod_{i=1}^l  
% a_{k_i}^{(p_i)} n^{-p_i-1/2 - k_i} \frac{(x_{\sigma(2i-1)} - x_{\sigma(2i)})^{2p_i+1}}{(2p_i+1)!}
%\end{align}
%with an error that is $O\left((1+(x_d - x_1)^{2N + 1})n^{-N-1/2}\right)$.
For all $x \in \mathbb{R}^d$ let
\[
F(x) = \Pfaff(q_{x_i, x_j}(n))_{i, j = 1}^d.
\]
This definition requires that $q_{x, y} = -q_{y, x}$. 
We first show $F$ is an antisymmetric polynomial in $(x_1, \ldots, x_{d})$. 
For each $1 \leq k < l \leq d$ let $D_{kl}$ denote the permutation matrix corresponding to the transposition of 
$x_k$ and $x_l$. 
Let $Q_x = (q_{x_i, x_j}(n))_{i, j = 1}^d$ and $x^{kl}$ be given by the vector $x$ with the $k$-th and $l$-th co-ordinates transposed.  
We use a conjugation formula for Pfaffians: for $d \times d$ matrices $A$ and $B$ such that $A$ is antisymmetric then $\Pfaff(B A B^T) = \Pfaff(A) \text{det}(B)$. Then
\begin{align*}
F(x) & = \Pfaff(Q_x) \\
& = \Pfaff(D_{kl} Q_x D_{kl}) \text{det}(D)^{-1} \\
& = (-1) \Pfaff(Q_{x^{kl}})\\
& = (-1) F(x^{kl}).
\end{align*}
Arguments of this form can be extended to general reflection groups, see Lemma 7.5 of \cite{doumerc_o_connell}. Therefore the Vandermonde determinant divides the first term on the right hand side of \eqref{asy.d}. Without loss of generality set $x_1 :=0$. As we have assumed $x_d - x_1 = o(n^{1/2})$
we can now assume $x_2, \ldots, x_d = o(n^{1/2})$. The relationship between the $x_i$ and $n$ means that for $x_2, \ldots, x_d = o(n^{1/2})$,
\begin{align*}
\P_x(\rho > n) = (\mathfrak{X} + o(1)) \Delta(x) n^{-d(d-1)/4}, \quad n \rightarrow \infty.
% \bigg( 1 + \sum_{r=1}^{\infty} U_r(x) n^{-r} \bigg)
%\sum_{k \in [N]^l} \sum_{r=1}^{\infty}U_r^{(k_1, \ldots, k_l)}(x) t^{-r-k_i}\bigg)  
%+ O\left(\frac{1+(x_d - x_1)^{2N + 1}}{n^{N+1/2}}\right)
\end{align*}
At this stage $\mathfrak{X}$ in unknown and we will determine its value later.

In the case when $d$ is odd,
\begin{equation}
\label{d.odd.eq}
\P_{(x_1, \ldots, x_d)}(\rho > n)  = \sum_{l=1}^d (-1)^{l+1} \P_{(x_1, \ldots, x_{l-1}, x_{l+1}, \ldots, x_d)}(\rho > n). 
\end{equation}
We focus on showing this is an antisymmetric polynomial in $x_1, \ldots, x_{d+1}$. The rest of the argument is same as the case when $d$ is even. 
Let $x_r$ denote $x$ with the $r$-th co-ordinate deleted and $x^{kl}_r$ denote $x$ with the $k$-th and $l$-th co-ordinates transposed before then deleting the $r$-th co-ordinate. 
For $x \in \mathbb{R}^d$ let
\[
F(x) = \sum_{r=1}^d (-1)^{r+1} \Pfaff(q_{x_i, x_j}(n))_{i, j \in [d] \setminus \{r\}}.
\]
Then
\begin{align*}
F(x) & = \sum_{r=1}^d (-1)^{r+1} \Pfaff(Q_{x_r}) \\
& = \sum_{r \neq k, l} (-1)^{r+1} \Pfaff(D_{kl} Q_{x_r} D_{kl}) \text{det}(D)^{-1}
+ (-1)^{k+1} \Pfaff(Q_{x_k}) + (-1)^{l+1} \Pfaff(Q_{x_l}) \\
& =  \sum_{r \neq k, l} (-1)^{r} \Pfaff(Q_{x_r^{kl}}) + (-1)^{k} \Pfaff(Q_{x_k^{kl}}) + (-1)^{l} \Pfaff(Q_{x_l^{kl}})\\
& = (-1) F(x^{kl}).
\end{align*}
The equality between lines 2 and 3 uses the conjugation formula to re-order the rows and column in the Pfaffian. 

%Therefore as before uniformly in $x \in \text{int}(W^d)$ with $x_d - x_1 = o(\sqrt{n})$ 
%\[
%\P_{x}(\rho > n) \sim \mathfrak{X} \Delta(x) n^{-d(d-1)/4}, \quad n \rightarrow \infty.
%\]
%To find $\mathfrak{X}$ we again consider the coefficient of $x_2 \ldots x_d^{d-1} n^{-d(d-1)/4}$. For this coefficient only the term with $l = 1$ in \eqref{d.odd.eq} contributes. Therefore we use the simplifications in \eqref{identify.constant} as before and find
%\begin{align*}
%\mathfrak{X} & =  \frac{1}{\pi^{d/2-1/2} \prod_{j=1}^{d-1} j!} D(d/2 - 1/2, -1/2) \\
%& = \frac{1}{\pi^{d/2-1/2} \prod_{j=1}^{d-1} j!} I(d/2 - 1/2, 1/2).
%\end{align*}
%In this case Eq 3.134 in \cite{forrester} gives $I(d/2 - 1/2, 1/2) = \pi^{-1/2} \prod_{j=1}^d \Gamma(j/2).$
%This gives the stated formula for $\mathfrak{X}$ in the case when $d$ is odd.  

We now consider the tail asymptotics for the ordering condition. 
We use part (ii) of Lemma \ref{coupling_lemma}, the above asymptotics for $\rho$ then part (i) of Lemma \ref{harmonic_simplify2} to obtain that as $n \rightarrow \infty$, uniformly for $x \in W^d$ with $x_d - x_1 = o(\sqrt{n})$,
\begin{align*}
\P_x(\tau > n ) \sim \E_x[\P_{x + \Psi}(\rho > n); A] & \sim \mathfrak{X} \E_x[\Delta(x + \Psi); A] n^{-d(d-1)/4} \\
& = \mathfrak{X} h(x) n^{-d(d-1)/4}. 
\end{align*}
The constant $\mathfrak{X}$ does not depend on the increment distribution~\cite{denisov_wachtel10} and therefore
agrees with the constant computed in the case of nearest-neighbour random walks, in particular (1.2) and (1.3) of~\cite{puchala_rolski_05}. The constant $\mathfrak{X}$ could also be found directly by analysing particular coefficients. 
\end{proof}

%
%
%\textcolor{blue}{
%\begin{remark}
%Another useful form for the case $d = 2$ can be found in terms of Bessel functions. 
%By \eqref{refl.principle} for $x_1 < x_2$,
%\begin{align*}
%\P_{(x_1, x_2)}(\rho \leq n) 
%& = 2\P_{(x_1, x_2)}(S_2(n) < S_1(n)).
%\end{align*}
%This is a one-dimensional problem for a random walk with the Laplace distribution which can be solved exactly (see for example the variance-gamma distribution). This leads to
%\[
%\P_{(x_1, x_2)}(\rho > n) = 2\int_0^{x_2 - x_1} \frac{1}{2^{n-1/2} \sqrt{\pi} \Gamma(n)} z^{n-1/2} K_{n - 1/2}(z) dz
%\]
%where $K_{n-1/2}(z)$ is a modified Bessel function of the second kind. This means that for general $d$ we can express $\P(\rho > n)$ as a Pfaffian of a matrix with entries given by special functions. 
%\end{remark}
%}

 \subsection{Proof of Theorem \ref{tail_asy} for $\lambda_1 < \ldots < \lambda_d$.} 
Let $\gamma = d \log(\bar{\lambda}/\lambda^*)$ where $\lambda^* = (\prod_{i=1}^d \lambda_i)^{1/d}$. 
By Proposition \ref{transition_density}, 
\[
\P_x(\tau > n)  = \int_{W^d} G_n^{(\lambda_1, \ldots, \lambda_d)}(x, z) dz.
\]
We first change variables $z_j \rightarrow n/\bar{\lambda} + z_j$ and then apply a change of measure
\begin{align}
\label{integral_expression}
\P_x(\tau > n) & = \int_{W^d} G_n^{(\lambda_1, \ldots, \lambda_d)}(x, n/\bar{\lambda} + z) dz \nonumber \\
& = \int_{W^d} \prod_{j=1}^d \left(\frac{\lambda_j}{\bar{\lambda}}\right)^n e^{-\sum_{i=1}^d (\lambda_i - \bar{\lambda})(n/\bar{\lambda} + z_i - x_i)} G_n^{(\bar{\lambda}, \ldots, \bar{\lambda})}(x, n/\bar{\lambda} + z) dz \nonumber \\
& = e^{-\gamma n}  \int_{W^d} e^{-\sum_{i=1}^d (\lambda_i - \bar{\lambda})( z_i - x_i)} G_n^{(\bar{\lambda}, \ldots, \bar{\lambda})}(x, n/\bar{\lambda} + z) dz.
\end{align}

We first consider the pointwise limit of the transition density.  
\begin{theorem}
\label{LLT}
Let $\lambda_1 = \ldots = \lambda_d = 1.$
For all $x, z \in \text{int}(W^d)$ and $x, z \in W^d$ respectively, uniformly in $x_d - x_1 = o(\sqrt{n}),$ $x_1 = o(\sqrt{n})$, $z_d - z_1 = o(\sqrt{n})$ and $z_1 = O(\sqrt{n})$, 
\begin{align*}
\widetilde{G}_n(x, n +z) & \sim \chi \Delta(x) \Delta(z) n^{-d^2/2} e^{- \frac{1}{2n}\sum_{j=1}^d z_j^2}, \quad n \rightarrow \infty, \\
G_n(x, n + z) & \sim\chi h(x) \hat{h}(z) n^{-d^2/2} e^{- \frac{1}{2n}\sum_{j=1}^d z_j^2}, \quad n \rightarrow \infty
\end{align*}
where $\chi = (2\pi)^{-d/2} \left(\prod_{j=1}^{d-1} j!\right)^{-1}.$
\end{theorem}

%\textcolor{blue}{
%Consider \begin{align*}
%\widetilde{G}_n(x, n +z)
% & = \frac{n^{nd}}{(n!)^d}\det((1 + z_j/n - x_i/n)^{n-1} e^{-(n + z_j - x_i)})_{i, j= 1}^d \\
% & =  \frac{n^{nd}}{(n!)^d}\det(e^{(n-1)\log (1+z_j/n - x_i/n)} e^{-(n + z_j - x_i)})_{i, j= 1}^d.
% \end{align*}
% A heuristic argument would be to just use the leading terms in the Taylor expansion of the logarithm 
%and apply the Harish-Chandra-Itzykson-Zuber integral which would give the asymptotics
% \begin{align*} 
%\det(e^{x_i z_j/\sqrt{n}})_{i, j = 1}^d
%& \sim C n^{-d(d-1)/2-d/2}  \Delta(x) \Delta(z).
%\end{align*}
%for some constant $C$. However, the result is $O(n^{-d(d-1)/2-d/2})$ so that we cannot
%immediately justify discarding terms in the determinant. The following argument gives a justification 
%of the idea above. 
%}

 \begin{proof}[Proof of Theorem \ref{LLT}]
 The transition density can be expressed for $x, z \in W^d$ and $n + z_1 \geq x_1, \ldots, n+z_d \geq x_1$ as
\begin{align*}
  \widetilde{G}_n(x, n +z) & = e^{-\sum_{j=1}^d (n + z_j - x_j)} \det\left(\frac{(n + z_j - x_i)^{n-1}}{(n-1)!}\right)_{i, j = 1}^d \\
  & = \frac{n^{nd} e^{-nd}}{(n!)^d} e^{-\sum_{j=1}^d (z_j - x_j)} \det\left(e^{(n-1)\log(1 + \frac{z_j - x_i}{n})}\right)_{i, j = 1}^d.
 \end{align*}
 Let
 % $e_M = \sum_{j = 0}^M x^j/j!$ and 
  $L_M = -\sum_{j=1}^M  (-1)^j x^j/j$.
We truncate the Taylor series of the logarithm to obtain that for any $\alpha > 0$ we can choose $M$ large enough such that
 \begin{align*}
 & \widetilde{G}_n(x, n +z) = \frac{n^{nd}e^{-nd}}{(n!)^d}\det(e^{(n-1) L_M (z_j/n - x_i/n)-(z_j - x_i)}))_{i, j= 1}^d
 + O(n^{-\alpha})
% & \sim (2\pi n)^{-d/2} \det\left(\sum_{p=0}^M \frac{1}{p!} \left[(n-1)\sum_{k=2}^M \frac{(x_i - z_j)^k}{k n^k} + \frac{x_i - z_j}{n^2} \right]^p\right)_{i, j= 1}^d + O(n^{-\alpha}). 
 \end{align*}
 The terms which only depend on the index of either the row or the column can be brought outside of the determinant as prefactors. 
 Therefore since $z_1, \ldots, z_d = O(n^{1/2})$ and $x_1, \ldots, x_d = o(n^{1/2})$ \[
  \widetilde{G}_n(x, n +z) = \frac{n^{nd}e^{-nd}}{(n!)^d} e^{-\frac{1}{2n}\sum_{j=1}^d z_j^2 + o(1)}
  \det\left(e^{\frac{x_i z_j}{n}(1 + O(n^{-1/2}))}\right)_{i, j= 1}^d
 + O(n^{-\alpha}).
 \]
 It is known that for $z_1, \ldots, z_d = O(n^{1/2})$ and $x_1, \ldots, x_d = o(n^{1/2})$
 \[
   \det\left(e^{\frac{x_i z_j}{n}(1 + O(n^{-1/2}))}\right)_{i, j= 1}^d \sim \frac{1}{\prod_{j=1}^{d-1} j!} n^{-d(d-1)/2} \Delta(x) \Delta(z), \quad n \rightarrow \infty.
 \]
For example, this follows from Equation 3.4 in \cite{shatashvili_93} and noting that the integral in that equation
converges to 1. 
Therefore 
\begin{equation}
\label{llt_rho}
\widetilde{G}_n(x, n +z) \sim \chi \Delta(x) \Delta(z) n^{-d^2/2} e^{- \frac{1}{2n}\sum_{j=1}^d z_j^2}, \quad n \rightarrow \infty. 
\end{equation}

% 
%The purpose of this is that the determinant on the right hand side is a polynomial in $x_1, \ldots, x_d, z_1, \ldots, z_d.$
%Moreover, it vanishes when $x_j = x_i$ and $z_j = z_i$ for any $i \neq j$. 
%Therefore there exists a polynomial $Q_n(x, z)$ with coefficients depending on $n$ such that
%\[
%\widetilde{G}_n(x, n +z) \sim \Delta\left(\frac{x}{\sqrt{n}}\right) \Delta\left(\frac{z}{\sqrt{n}}\right) Q_n\left(\frac{x}{\sqrt{n}}, \frac{z}{\sqrt{n}}\right) e^{- \frac{d z_1^2}{2n}} + O(n^{-\alpha}).
%\]
%The reason for choosing the scaling $x/\sqrt{n}$ and $z/\sqrt{n}$ is that the coefficients in $Q_n$ are uniformly bounded in $n$. We can find the coefficient of $\prod_i z_i^0 x_i^0$ to be 
%$\sim (2\pi n)^{-d/2} \left(\prod_{j=1}^{d-1} j!\right)^{-1}$ by comparing the coefficient of a particular term such as 
%$z_1^0 z_2^1 \ldots z_d^{d-1} x_1^0 x_2^1 \ldots x_d^{d-1}$. Therefore 
%\begin{equation}
%\label{llt_rho}
%\widetilde{G}_n(x, n +z) \sim \chi \Delta(x) \Delta(z) n^{-d^2/2} e^{- \frac{d z_1^2}{2n}}, \quad n \rightarrow \infty. 
%\end{equation}
% 

%A difficulty is that we have discarded all terms of $O(1/n)$ in the third line but due to cancellations in the determinant we end up with an answer which is $O(n^{-d(d-1)/4})$ so that this does not immediately appear to be valid. 
%However, I think that all of the lower order terms will not help the cancellation in the determinant.
We can then extend to the ordered case using the coupling from Section \ref{coupling}.
Lemma \ref{coupling_lemma} part (i) states that
\[ G_n(x, n +z) 
 = \E[\widetilde{G}_{n-d-1}(x +\Psi, n + z - \Phi); A, B]
\]
where $\Psi, \Phi, A, B$ are all defined in Section \ref{coupling}. 
Therefore from \eqref{llt_rho} and interchanging the limit and expectations using Lemma
\ref{lem.large.deviations}
\begin{equation}
\label{LLT1}
G_n(x, n +z) \sim  \chi n^{-d^2/2} e^{-\frac{1}{2n}\sum_{j=1}^d z_j^2} \E[\Delta(x + \Psi) \Delta(z - \Phi); A, B].
\end{equation}
As remarked in (and using notation from) Section \ref{coupling} the definition of $\Phi$ and event $B$ 
correspond to 
the definition of $\Psi$ and event $A$ with the choices that 
$x_j = -z_{d+1-j}$ along with $V_j^i = U_{d+1-j}^i$ and 
$\Psi_j = \Phi_{d+1-j}$ for $j =1, \ldots, d$. Therefore Lemma \ref{harmonic_simplify2} shows that
\begin{align}
\label{h_hat_eqn}
 \E[ \Delta(z_1 - \Phi_1, \ldots, z_d - \Phi_d); B]  & = \E[\Delta(-z_d+\Phi_d, \ldots, -z_1+\Phi_1 ; A] \nonumber \\
& =  h(-z_d, \ldots, -z_1) \nonumber \\
& = \hat{h}(z_1, \ldots, z_d). 
\end{align} 
Lemma \ref{harmonic_simplify2} can also be applied to simplify $\E[\Delta(x + \Psi); A] = h(x)$. 
Therefore \eqref{LLT1} simplifies to 
\[
G_n(x, n +z) \sim  \chi n^{-d^2/2} e^{-\frac{1}{2n}\sum_{j=1}^d z_j^2} h(x) \hat{h}(z), \quad n \rightarrow \infty. \qedhere
\]
%Recall the definitions of $\Phi$ 
% by using the coupling in Section and
% interchanging limits using Proposition \ref{prop.uniform.bounds}.
% Let $(\gamma_j^i : i+j \leq d)$ be an independent collection of exponential random variables with rate $1$.  
% Let $(U_j^i : i+j  \leq d)$ be defined inductively by $U_j^0 := 0$ and $U_j^i = U_j^{i-1} + \gamma_j^i.$ 
%Let $B$ denote the event that $z_j - U_j^i \leq z_{j+1} - U_{j+1}^{i}$ for all $i+ j < d$. 
%Let $\Phi = (U_1^{d-1}, \ldots, U_{d-1}^1, 0)$. 
%If we reverse signs then the series of inequalities become
%$-z_{j+1} + U_{j+1}^i \leq -z_j + U_j^i$.  
%These inequalities correspond to 
%the event $A$ in Lemma \ref{harmonic_simplify2} with the choices that 
%$x_j = -z_{d+1-j}$ along with $V_j^i = U_{d+1-j}^i$ and 
%$\Psi_j = \Phi_{d+1-j}$ for $j =1, \ldots, d$. Thus Lemma \ref{harmonic_simplify2} shows that
%\begin{align*}
%& \E[ \Delta(z_1 - \Phi_1, \ldots, z_d - \Phi_d); B]  = \E[\Delta(-z_d+\Phi_d, \ldots, -z_1+\Phi_1 ; A] \\
%& \quad =  h(-z_d, \ldots, -z_1) = \hat{h}(z_1, \ldots, z_d). 
%\end{align*} 
%This allows us to extend the local limit theorem to the case with $\tau$ rather than $\rho$. 
% For all $x \in W^d$ as $n \rightarrow \infty$, using Proposition \ref{prop.uniform.bounds}
%\begin{align*}
%G_n(x, n +z) 
%& = \E[G_{n-d}(x, +\Psi, n-d + z - \Phi); A, B] \\
%%& = \E[\P_{x + \Psi}(S_{n-d} - (n-d) + \Phi \in dz); A, B]   \\
%& \sim \chi n^{-d(d-1)/2-d/2} \E[\Delta(x + \Psi) \Delta(z - \Phi); A, B] \\
%& = \chi h(x) \hat{h}(z) n^{-d(d-1)/2-d/2}. \qedhere
%\end{align*}
 \end{proof}

\begin{proof}[Proof of Theorem \ref{tail_asy} part (iii)]
Recall \eqref{integral_expression},
\begin{equation}
\label{drift_integral}
\P_x(\tau > n) = e^{-\gamma n}  \int_{W^d} e^{-\sum_{i=1}^d (\lambda_i - \bar{\lambda})( z_i - x_i)} G_n^{(\bar{\lambda}, \ldots, \bar{\lambda})}(x, n/\bar{\lambda} + z) dz.
\end{equation}
We change variables $r_1 = z_2 - z_1, \ldots, r_{d-1} = z_d - z_{d-1}$ and 
$\theta = \frac{1}{\sqrt{n} d} \sum_{j=1}^d z_j$. 
Use that 
\begin{align}
\label{xi_eqn}
\sum_{i=1}^d z_i (\bar{\lambda} - \lambda_i) & = 
\frac{1}{d} \sum_{1 \leq i < j \leq d} (z_j - z_i)(\lambda_i - \lambda_j)  \nonumber \\
& = \frac{1}{d} \sum_{1 \leq i < j \leq d} (r_i + r_{i+1} + \ldots + r_{j-1}) (\lambda_i - \lambda_j).
\end{align}
Let $r = (r_1, \ldots, r_{d-1})$ and define 
\[
H(r) = \E \left( \prod_{1 \leq i < j \leq d} (r_i + \ldots + r_{j-1} + \eta_{d-i+1} - \eta_{d-j+1}) \right).
\]
In a similar way to \eqref{h_hat_eqn}, 
\begin{align*}
\hat{h}(z_1, \ldots, z_d) %& =  \E \Delta(-z_d + \eta_1, \ldots, -z_1 + \eta_d) \\
%& = 
%\E[ \prod_{1 \leq i < j \leq d} (-z_{d-j+1} + \eta_j + z_{d-i+1} - \eta_i)] \\
& = \E[ \prod_{1 \leq i < j \leq d} (z_{j} - \eta_{d-j+1} - z_{i} + \eta_{d-i+1})]\\
& = \E[ \prod_{1 \leq i < j \leq d}  (r_i + r_{i+1} + \ldots + r_{j-1} + \eta_{d-i+1} - \eta_{d-j+1})].
\end{align*}
We use Lemma \ref{lem.large.deviations} to justify interchanging limits in \eqref{drift_integral} after the change of variables above. 
First note that \eqref{xi_eqn} gives exponential decay in $r_1, \ldots, r_{d-1}$ for $r_1 > 0, \ldots, r_{d-1} > 0$ and dominates the polynomial factors in Lemma \ref{lem.large.deviations}. 
Then note that the second statement in part (i) of Lemma \ref{lem.large.deviations} gives the required decay in $\theta$.
After interchanging limits we then use the asymptotics in Theorem \ref{LLT}. Note that 
$\frac{1}{n} \sum_{j=1}^d z_j^2 = d \theta^2 + o(1).$ 
Therefore
\begin{align*}
&\P(\tau > n) \sim \chi  n^{-d^2/2 + 1/2} e^{-\gamma n} e^{\sum_{i=1}^d (\lambda_i - \bar{\lambda}) x_i}h(x) \\
& \int_{-\infty}^{\infty} d\theta \int_{0}^{\infty} d\xi_1 \ldots \int_0^{\infty} d\xi_{d-1} 
e^{\frac{1}{d} \sum_{1 \leq i < j \leq d} (r_i + r_{i+1} + \ldots + r_{j-1}) (\lambda_i - \lambda_j)} H(r) e^{-d\theta^2/2}.
\end{align*}
After performing the integral in $\theta$ we have the stated asymptotics for $\tau$ with 
$c_d = (2\pi)^{-d/2+1/2} (\prod_{j=1}^{d-1} j!)^{-1} d^{-1/2}$ and
\begin{align}
\label{factor1}
K_{\lambda} & = c_d \int_{0}^{\infty} dr_1 \ldots \int_0^{\infty} dr_{d-1} 
e^{\frac{1}{d} \sum_{1 \leq i < j \leq d} (r_i + r_{i+1} + \ldots + r_{j-1}) (\lambda_i - \lambda_j)} H(r).
\end{align}
%The bounds in Lemma \ref{prop.uniform.bounds} and the local limit theorem in Theorem \ref{LLT} now 
%prove that uniformly for $x \in W^d$ with $x_d - x_1 = o(\sqrt{n})$,
%\[
%\P_x(\tau > n) \sim   K_{\lambda}n^{-d(d-1)/2 - d/2} e^{-\gamma n} e^{\sum_{i=1}^d ( \lambda_i- \bar{\lambda} ) x_i} h^{(\bar{\lambda}, \ldots, \bar{\lambda})}(x), 
%\quad n \rightarrow \infty
%\]
%with \[
%K_{\lambda}  = \chi \int_{W^d} e^{-\sum_{j=1}^d (\lambda_j - \bar{\lambda} )z_j} \hat{h}^{(\bar{\lambda}, \ldots, \bar{\lambda})}(z) dz_1 \ldots dz_d. 
%\]
%The integral is finite due to the condition $\lambda_1 < \ldots < \lambda_d$.
Let \[
\widetilde{H}(r) = \prod_{1 \leq i < j \leq d} (r_i + \ldots + r_{j-1}).
\]
The same argument also gives the stated tail asymptotics for $\rho$ with constant factor 
\begin{align}
\label{factor2}
C_{\lambda} & = c_d \int_{0}^{\infty} dr_1 \ldots \int_0^{\infty} dr_{d-1} 
e^{\frac{1}{d} \sum_{1 \leq i < j \leq d} (r_i + r_{i+1} + \ldots + r_{j-1}) (\lambda_i - \lambda_j)} \widetilde{H}(r).
\end{align}
\end{proof}

\section{The smallest and largest particles} 
\label{sec.largest.particle}

 In this section our aim is to find the distribution of the smallest and largest particles when $(Z(n))_{n \geq 0}$ has general starting positions. 
Suppose that $\lambda_1 = \ldots = \lambda_d = 1$ and let $x = (x_1, \ldots, x_d)$ and $z = (z_1, \ldots, z_d)$. Then applying the $h$-transform from Theorem \ref{harmonic} to Proposition \ref{transition_density} gives
\begin{equation*}
\P_x(Z(n) \in dz) = e^{-\sum_{j=1}^d (z_j - x_j)}  \det(q_{n+i-j}(z_j - x_i))_{i, j = 1}^d \frac{h(z)}{h(x)} dz, \quad x, z \in W^d.
\end{equation*}

%
%For any $\eta > 0$ let $I^{\eta} = \{{0 \leq z_1 \leq \ldots \leq z_d \leq \eta} \}$ and
%$I_{\eta} = \{\eta \leq z_1 \leq \ldots \leq z_d\}$. 
%In the fixed time case we have
%\begin{align*}
%\P_x(Z_d(n) \leq \eta) & = 
%\int_{\upp} \Delta(z)  \det(q_n^{(d-i)}(z_j - x_i))_{i, j = 1}^d 
%e^{-\sum_{i=1}^d (z_i-x_i)} dz_1 \ldots dz_d \\
%\P_x(Z_1(n) \geq \eta) & = \int_{\low} \Delta(z)  \det(q_n^{(1-i)}(z_j - x_i))_{i, j = 1}^d 
%e^{-\sum_{i=1}^d (z_i-x_i)} dz_1 \ldots dz_d. 
%\end{align*}

\begin{proof}[Proof of Theorem \ref{largest_particle}]
For any $a \in \mathbb{R}$ let $I^{a} = \{{x_1 \leq \ldots \leq x_d \leq a} \}$ and
$I_{a} = \{a \leq x_1 \leq \ldots \leq x_d\}$. 
We will use the representation $h_2$ for the harmonic function from Section \ref{harmonic_properties}. 
All of the matrices defined in the determinants in this proof are indexed by $i, j = 1, \ldots, d$ and we omit this from the notation.

Proposition \ref{transition_density} and Theorem \ref{harmonic} give that 
\begin{align*}
& \P_x(Z_d(n_1) \leq \xi_1, \ldots, Z_d(n_m) \leq \xi_m) \\
& =  \frac{e^{\sum_i \lambda_i x_i^0}}{h(x^0)}  \int_{\uppone} \ldots \int_{\uppm} \det(q_{n_1+i-j}(x_j^1 - x_i^0))\det(q_{n_2-n_1+i-j}(x_j^2 - x_i^1)) \\
& \ldots \det(q_{n_m-n_{m-1}+i-j}(x_j^m - x_i^{m-1})) \text{det}((-1)^{d-j}\phi_i^{(d-j)}(x_j^m))  \prod_{k=1}^m \prod_{j=1}^d dx_j^k 
\end{align*}
where $x^0 : = x$.
The main problem which prevents us immediately applying the Eynard-Mehta theorem is the dependence on $i$ and $j$ in the functions such as $q_{n_2 - n_1 +i-j}$ appearing in the determinants. 
We use the integral and derivative relations \eqref{derivative_relations} and \eqref{integral_relations}
to remove this dependency on $i$ and $j$. 

We start with smooth approximations $q_k^{(\epsilon)}$ of the functions appearing above before passing to a limit. 
We integrate by parts for $k = 1, \ldots, m$ in the order $x_1^k, x_2^k, \ldots, x_{d-1}^k$ then $x_1^k,  \ldots, x_{d-2}^k$ and so on until finally $x_1^k.$ This ensures that there are no boundary conditions due to the determinants having equal rows or columns at each boundary as in Lemma 2 of \cite{FW}. 
The limit as $\epsilon \rightarrow 0$ can then be taken in a similar way to Lemma 5 from \cite{FW}. We give more details in Section \ref{ibp_justify}. The condition $n_1 \geq d-1$ is needed to justify taking this limit.  
Therefore
\begin{align*}
&  \P_x(Z_d(n_1) \leq \xi_1, \ldots, Z_d(n_m) \leq \xi_m)\\
& =  \frac{e^{\sum_i \lambda_i x_i^0}}{h(x^0)} \int_{\uppone} \ldots \int_{\uppm} \det(q_{n_1+i-d}(x_j^1 - x_i^0))\det(q_{n_2-n_1}(x_j^2 - x_i^1)) \\
& \ldots \det(q_{n_m-n_{m-1}}(x_j^m - x_i^{m-1})) \text{det}(\phi_i(x_j^m))  \prod_{k=1}^m \prod_{j=1}^d dx_j^k .
\end{align*}
%\begin{align*}
%\P_x(Z_d(n) \leq z) &= e^{\sum_i \lambda_i x_i} \int_{\upp} \det(\phi_i^{(d-j)}(z_j))   \det(q_n^{(j-i)}(z_j - x_i)) dz_1 \ldots dz_d \\
%& = e^{\sum_i \lambda_i x_i} \int_{\upp} \det(\phi_i(z_j))   \det(q_n^{(d-i)}(z_j - x_i)) 
%dz_1 \ldots dz_d. 
%\end{align*}

%The distribution of $Z(n)$ is continuous in the $\lambda_i$ and so after taking the limit $\lambda_i \rightarrow 1$,
%\begin{align*}
%& \P_x(Z_d(n) \leq z, Z_d(m+n) \leq \xi) \\
%& =  \int_{\uppone} \ldots \int_{\uppm} e^{-\sum_{i=1}^d (z_i-x_i)}  \det(q_n^{(d-i)}(y_j - x_i))\det(q_m(z_j - y_i)) \Delta(z) dy dz \\
%& =   \int_{\uppone} \ldots \int_{\uppm}  \det(f_{n-d+i}(y_j - x_i))\det(f_m(z_j - y_i)) \Delta(z) dy dz.
%\end{align*}
%It is convenient to rewrite in terms of the probability density function of a Gamma$(n, 1)$ random variables denoted $f_n$. Then for each $n \geq k,$ $q_n^{k}(x)e^{-x} = f_{n-k}(x)$.  
Rewriting in terms of $f_n$ we have
\begin{align*}
&  \P_x(Z_d(n_1) \leq \xi_1, \ldots, Z_d(n_m) \leq \xi_m) \\
& = \frac{1}{h(x^0)}\int_{I^{\xi_1}} \ldots \int_{I^{\xi_m}}
\det(f_{n_1-d+i}(x_j^1 - x_i^0))\det(f_{n_2 - n_1}(x_j^2 - x_i^1)) \\
&  \ldots \det(f_{n_m - n_{m-1}}(x_j^m - x_i^{m-1}))  \Delta(x^m)  \prod_{k=1}^m \prod_{j=1}^d dx_j^k. 
\end{align*}
From the Eynard-Mehta theorem the right hand side is given by a Fredholm determinant
with the stated extended kernel eg. \cite{johansson_survey, tracy_widom}. The fact that $A$ is an invertible matrix can be seen as follows. For each $j = 1, \ldots, d$ define independent random variables $\xi_{n-d}^{(j)} \sim \text{Gamma}(n-d, 1)$. Then
\begin{align*}
\det(A) & = \E[ \det((x_k + \eta_{n-d+k})^{l-1})_{k, l = 1}^d] \\
& = \E^{(\eta)} \E^{(\xi)} \Delta(x_1 + \eta_1 + \xi_{n-d}^{(1)}, \ldots, x_d + \eta_d + \xi_{n-d}^{(d)}). 
\end{align*}
The Vandermonde determinant is harmonic for an increment with 
distribution $(\xi_{n-d}^{(1)}, \ldots,  \xi_{n-d}^{(d)})$ by Corollary 2.2 of \cite{konig2002}. 
Therefore $\det(A) = h(x) > 0.$

For the distribution of the smallest particle the same argument shows that
\begin{align*}
& \P_x(Z_1(n_1) \geq \xi_1, \ldots, Z_1(n_m) \geq \xi_m) \\
& =  \frac{e^{\sum_i \lambda_i x_i^0}}{h(x^0)}  \int_{\lowone} \ldots \int_{\lowm} \det(q_{n_1+i-j}(x_j^1 - x_i^0))\det(q_{n_2-n_1+i-j}(x_j^2 - x_i^1)) \\
& \ldots \det(q_{n_m-n_{m-1}+i-j}(x_j^m - x_i^{m-1})) \text{det}((-1)^{d-j}\phi_i^{(d-j)}(x_j^m))  \prod_{k=1}^m \prod_{j=1}^d dx_j^k. 
\end{align*}
Again we start with a smooth approximation, apply an integration by parts and then take a limit. This time we need to integrate by parts for $k = 1, \ldots, m$ in the order $x_d^k, x_{d-1}^k, \ldots, x_1^k, x_{d}^k, \ldots, x_2^k, \ldots, x_d^k$ which ensures there are no boundary conditions. This requires the condition $n_m - n_{m-1} \geq d-1$, see Section \ref{ibp_justify}.
Therefore 
\begin{align*}
&  \P_x(Z_1(n_1) \geq \xi_1, \ldots, Z_1(n_m) \geq \xi_m)\\
& =  \frac{e^{\sum_i \lambda_i x_i^0}}{h(x^0)}  \int_{\lowone} \ldots \int_{\lowm} \det(q_{n_1+i-1}(x_j^1 - x_i^0))\det(q_{n_2-n_1}(x_j^2 - x_i^1)) \\
& \ldots \det(q_{n_m-n_{m-1}}(x_j^m - x_i^{m-1})) \text{det}(\phi_i^{(d-1)}(x_j^m))  \prod_{k=1}^m \prod_{j=1}^d dx_j^k .
\end{align*}
We use the reduction that $ \text{det}(\phi_i^{(d-1)}(z_j)) = e^{-\sum_{j=1}^d z_j}\Delta(z)$.  
Therefore
\begin{align*}
&  \P_x(Z_1(n_1) \geq \xi_1, \ldots, Z_1(n_m) \geq \xi_m) \\
& = \frac{1}{h(x^0)}\int_{I_{\xi_1}} \ldots \int_{I_{\xi_m}}
\det(f_{n_1-1+i}(x_j^1 - x_i^0))\det(f_{n_2 - n_1}(x_j^2 - x_i^1)) \\
&  \ldots \det(f_{n_m - n_{m-1}}(x_j^m - x_i^{m-1}))  \Delta(x^m)  \prod_{k=1}^m \prod_{j=1}^d dx_j^k. 
\end{align*}
The stated formula now follows from the Eynard-Mehta theorem. The argument used for $A$ also shows that $B$ is invertible. 
%
%  \begin{align*}
%\P_x(Z_1(n) \geq \eta)
%& =  e^{\sum_i \lambda_i x_i} \int_{\low} \det(\phi_i^{(d-j)}(z_j))   \det(q_n^{(j-i)}(z_j - x_i)) dz_1 \ldots dz_d.
%\end{align*}
%Again we apply an integration by parts lemma but this time need to integrate by parts in the order $z_d, z_{d-1}, \ldots, z_1, z_{d}, \ldots, z_2, \ldots, z_d$ which ensures there are no boundary conditions. Therefore
%\[
%\P_x(Z_1(n) \geq \eta)  = e^{\sum_{i=1}^d x_i} \int_{\low} \det(\phi_i^{(d-1)}(z_j))  \det(q_n^{(1-i)}(z_j - x_i))
% dz_1 \ldots dz_d. 
%\]
%Then $\phi_i^{(d-1)}(z) = \lambda^{d-1}_ie^{-\lambda_i z}$ and after taking a limit $\lambda_i \rightarrow 1$  
%\[
%\P_x(Z_1(n) \geq \eta)  = \int_{\low} \Delta(z)  \det(q_n^{(1-i)}(z_j - x_i))
%e^{-\sum_{i=1}^d (z_i-x_i)} dz_1 \ldots dz_d. 
%\]
%In the case $x_i = 0$ for $i = 1, \ldots, d$ we equality in distribution to the smallest eigenvalue of LUE with $a = n-1$,
%\[
%\P_x(Z_1(n) \geq \eta)  = \int_{\low} \prod_{i = 1}^d z_i^{n-1} \Delta(z)^2 e^{-\sum_{i=1}^d z_i} dz_1 \ldots dz_d. 
%\] 
\end{proof}

\subsection{Integration by parts} 
\label{ibp_justify}
Let $q_n^{(\epsilon)}$ be the smooth approximations defined in the proof of Theorem \ref{largest_particle}.
%satisfying that 
%$\tilde{q}_n(x) = 0$ for $x \leq 0$, 
%that
%$\tilde{q}_n(x) = q_n(x)$ for $x \geq \epsilon$ and that $\tilde{q}_n$ is uniformly bounded for all $x \in [0, 1]$. 
%As $\tilde{q}_n$ is smooth then for $k > 0$ we define $\tilde{q}_n^{(k)}$ as the $k$-th derivative of $\tilde{q}_n$ and
%for $k < 0$ we define $\tilde{q}_n^{(-k)}(x) = \int_0^x \frac{(x-u)^{k-1}}{(k-1)!} \tilde{q}(u) du$.
%\begin{align}
%\label{ibp_smooth}
%& \int_{\uppone} \ldots \int_{\uppm} \det(\tilde{q}_{n_1}^{(j-i)}(x_j^1 - x_i^0))\det(\tilde{q}_{n_2-n_1}^{(j-i)}(x_j^2 - x_i^1))\nonumber  \\
%& \ldots \det(\tilde{q}_{n_m-n_{m-1}}^{(j-i)}(x_j^m - x_i^{m-1})) \text{det}(\phi_i^{(d-j)}(x_j^m))  \prod_{k=1}^m \prod_{j=1}^d dx_j^k \nonumber \\
%& = 
%\int_{\uppone} \ldots \int_{\uppm} \det(\tilde{q}_{n_1}^{(d-i)}(x_j^1 - x_i^0))\det(\tilde{q}_{n_2-n_1}(x_j^2 - x_i^1)) \nonumber \\
%& \ldots \det(\tilde{q}_{n_m-n_{m-1}}(x_j^m - x_i^{m-1})) \text{det}(\phi_i(x_j^m))  \prod_{k=1}^m \prod_{j=1}^d dx_j^k. 
%\end{align}
%We want to pass to a limit in \eqref{ibp_smooth} to remove the smooth approximation. 
As discussed in the proof of 
Theorem \ref{largest_particle} we can establish for this smooth approximation that
\begin{align}
\label{ibp_smooth1}
& \int_{\uppone} \det(q^{(\epsilon)}_{n_1+i-j}(y_j - x_i))\det(q^{(\epsilon)}_{n_2 - n_1+i-j}(z_j - y_i)) dy_1 \ldots dy_d \nonumber \\
& = 
\int_{\uppone}\det(q^{(\epsilon)}_{n_1+i-d}(y_j - x_i))\det(q^{(\epsilon)}_{n_2-n_1+d-j}(z_j - y_i)) dy_1 \ldots dy_d. 
\end{align}
We now take a limit in $\epsilon$ of both sides of the equation. This follows as in Lemma 5 of \cite{FW}
except with the following additional complication when taking the limit of the right hand side. 
A term in the Laplace expansion of the right hand side of \eqref{ibp_smooth1} corresponding to permutations $\sigma$ and $\rho$ is 
 \begin{equation*}
\int_{\uppone} \prod_{i=1}^d q^{(\epsilon)}_{n_1+\sigma(i) -d}(y_i - x_{\sigma(i)})
q^{(\epsilon)}_{n_2 - n_1+i - \rho(i)}(z_{\rho(i)} - y_i) dy_1 \ldots dy_d.
\end{equation*}
If $\sigma$ is the identity then  
$\prod_{i=1}^d q^{(\epsilon)}_{n_1+i-d}(y_{i} - x_i)$ is bounded uniformly in $\epsilon$ for $0 \leq x_i \leq y_{i}$ if and only if 
 $n_1 \geq d - 1$. This is the reason for the condition $n_1 \geq d -1$. Once this is imposed the limit in $\epsilon$ can be taken as in Lemma 5 of \cite{FW}.
 %Thus we can pass to a limit in \eqref{ibp_smooth1} when 
%$x_1 < \ldots < x_n$ and $z_1 < \ldots < z_n$. In \cite{FW} it also shown that both sides 
%are continuous in $x$ and $z$ for $x, z \in W_n^+$ so that we can remove the condition 
%$x_1 < \ldots < x_n$ and $z_1 < \ldots < z_n$. 
By the same method, we can establish that
\begin{align*}
& \int_{\upptwo} \det(q_{n_2-n_1+d - j}(y_j - x_i))\det(q_{n_3 - n_2+i-j}(z_j - y_i)) dy_1 \ldots dy_d \nonumber \\
& = 
\int_{\upptwo}\det(q_{n_2-n_1}(y_j - x_i))\det(q_{n_3 - n_2+d -j}(z_j - y_i)) dy_1 \ldots dy_d. 
\end{align*}
In this case there is no need for a constraint on $n_2 - n_1.$
Finally we pass to the limit in 
%\begin{align}
%& \int_{\uppone} \ldots \int_{\uppm} \det(q_{n_1}^{(j-i)}(x_j^1 - x_i^0))\det(q_{n_2-n_1}^{(j-i)}(x_j^2 - x_i^1))\nonumber  \\
%& \ldots \det(\tilde{q}_{n_m-n_{m-1}}^{(j-i)}(x_j^m - x_i^{m-1}))  \prod_{k=1}^{m-1} \prod_{j=1}^d dx_j^k \nonumber \\
%& = 
%\int_{\uppone} \ldots \int_{\uppm} \det(q_{n_1}^{(d-i)}(x_j^1 - x_i^0))\det(q_{n_2-n_1}(x_j^2 - x_i^1)) \nonumber \\
%& \ldots \det(\tilde{q}_{n_m-n_{m-1}}(x_j^m - x_i^{m-1})) \prod_{k=1}^{m-1} \prod_{j=1}^d dx_j^k. 
%\end{align}
%First we have
%\begin{align}
%\label{ibp_smooth3}
%& \int_{\uppone} \det(\tilde{q}_{n_1}^{(j-i)}(y_j - x_i))\det(\tilde{q}_{n_2-n_1}^{(j-i)}(z_j - y_i)) dy_1 \ldots dy_d \nonumber \\
%& = 
%\int_{\uppone}\det(\tilde{q}_{n_1}^{(d-i)}(y_j - x_i))\det(\tilde{q}_{n_2-n_1}(z_j - y_i)) dy_1 \ldots dy_d 
%\end{align}
\begin{align*}
& \int_{\uppm} \det(q^{(\epsilon)}_{n_m-n_{m-1}+d-j}(x_j^m - x_i^{m-1})) \text{det}(\phi_i^{(d-j)}(x_j^m)) dx^m_1 \ldots dx^m_d\\
& = \int_{\uppm} \det(q^{(\epsilon)}_{n_m-n_{m-1}}(x_j^m - x_i^{m-1})) \text{det}(\phi_i(x_j^m)) dx^m_1 \ldots dx^m_d
\end{align*}
which is straightforward since every function is smooth. 
The justification for the smallest particles is similar except we start with the $x^m_j$ and end with the $x^1_j$.  The condition $n_1 \geq d -1$ is replaced by the condition $n_m - n_{m-1} \geq d-1$.

\paragraph{Acknowledgment.} 
We are very grateful to the reviewer for their detailed reading and for their helpful and constructive comments. 

\appendix

\section{Doob h-transforms for ordering and interlacing}
\label{h_transform}
The harmonic function in Theorem \ref{harmonic} and tail asymptotics in Theorem \ref{tail_asy} give two ways of defining an exponential random walk conditioned to stay ordered. 
%The resulting processes coincide for $\lambda_1 \geq \ldots \geq \lambda_d$ 
%but will be different otherwise.
Suppose first either that $\lambda_1 > \ldots > \lambda_d$ or that all rates are equal.
Recall the function $h$ from Theorem \ref{harmonic}
 satisfies $\E_x(h(S(1)) 1_{\tau > 1}) = h(x)$ and $h(x) > 0$ for all $x \in W^d$.
We can define
$(Z(n))_{n \geq 0} = (Z_1(n), \ldots, Z_d(n))_{n \geq 0}$ as a change of measure of $(S(n))_{n \geq 0}$ using the harmonic function $h$. For bounded measurable $f$,
\begin{align*}
\E_x[f(Z(k) : 0 \leq k \leq n)] & = \E_x\left[\frac{h(S(n))}{h(x)} f(S(k) : 0 \leq k \leq n) 1_{\{ \tau > n\}}\right].
\end{align*}
This defines a transformed process which is a Markov chain on $W^d$
with transition densities 
\[
\P_x(Z(n) \in dz) = \frac{h(z)}{h(x)} \P_x(S(n) \in dz, \tau > n), \quad x, z \in W^d. 
\]
We refer to $(Z(n))_{n \geq 0}$ as a 
(Doob) $h$-transform.

In the case $\lambda_1 < \ldots < \lambda_d$ we still have
$\E_x [h(S(1)) 1_{\tau > 1}] = h(x)$ but now $h(x) < 0$ on $W^d$. Hence we can use $(-h)$ to define a Doob $h$-transform. The transition densities of the $h$-transformed process are given by using the definition of $h$, Proposition \ref{transition_density} and cancelling the terms in $\lambda_i$ which can be brought outside of the determinant as prefactors. This gives
\[
%\frac{(-h(z))}{(-h(x))} \P_x(S(n) \in dz, \tau > n) = 
\prod_{j=1}^d \lambda_j^n \frac{\det(\lambda_i^{-j} e^{-\lambda_i z_j})_{i, j = 1}^d}{\det(\lambda_i^{-j} e^{-\lambda_i x_j})_{i, j = 1}^d} \det(q_{n+i-j}(z_j - x_i))_{i, j = 1}^d.
\]
This is invariant under permutations of the $\lambda_i$. Thus the $h$-transformed process in the case 
$\lambda_1 < \ldots < \lambda_d$ agrees with the case $\lambda_1 > \ldots > \lambda_d$.

Alternatively we can define $(\hat{Z}(n))_{n \geq 0} = (\hat{Z}_1(n), \ldots, \hat{Z}_d(n))_{n \geq 0}$ by conditioning on $\{\tau > m\}$ and then taking the limit $m \rightarrow \infty.$
For bounded measurable $f$, 
\begin{align*}
 \E_x [f(\hat{Z}(k) : k \leq n)] = 
 \lim_{m \rightarrow \infty} \E_x \left[f(S(k) : k \leq n) 1_{\{\tau > n\}} \frac{\P_{S(n)}(\tau > m - n)}{\P_x(\tau > m)}\right]. 
\end{align*}
Theorem \ref{tail_asy} gives the asymptotics of the ratio on the right hand side. 
In the case when either $\lambda_1 > \ldots > \lambda_d$ or all rates are equal, 
then this definition of $(\hat{Z}(n))_{n \geq 0}$ coincides with the definition of $(Z(n))_{n \geq 0}$ as an $h$-transform.

If $\lambda_1 < \ldots < \lambda_d$ then using part (iii) of Theorem \ref{tail_asy},
\[
\frac{\P_{z}(\tau > m - n)}{\P_x(\tau > m)} \sim e^{\gamma n} \frac{e^{\sum_{i=1}^d ( \lambda_i- \bar{\lambda} ) z_i} h^{(\bar{\lambda})}(z)}{e^{\sum_{i=1}^d ( \lambda_i- \bar{\lambda} ) x_i} h^{(\bar{\lambda})}(x)}, \quad m \rightarrow \infty.
\]
Therefore $(\hat{Z}(n)_{n \geq 0}$ has transition densities
\begin{align*}
& e^{\gamma n}\frac{e^{\sum_{i=1}^d ( \lambda_i- \bar{\lambda} ) z_i} h^{(\bar{\lambda})}(z)}{e^{\sum_{i=1}^d ( \lambda_i- \bar{\lambda} ) x_i} h^{(\bar{\lambda})}(x)} \prod_{j=1}^d \lambda_j^n e^{-\sum_{j=1}^d \lambda_j(z_j - x_j)}  \det(q_{n+i-j}(z_j - x_i))_{i, j = 1}^d \\
& = \frac{ h^{(\bar{\lambda})}(z)}{ h^{(\bar{\lambda})}(x)} \bar{\lambda}^{nd} e^{-\sum_{j=1}^d \bar{\lambda}(z_j - x_j)}  \det(q_{n+i-j}(z_j - x_i))_{i, j = 1}^d.
\end{align*}
This agrees with a Doob $h$-transform of an exponential random walk with equal rates all given by 
$\bar{\lambda}$ and using $h^{(\bar{\lambda})}$ as the harmonic function. 
Thus the definitions of  $(Z(n))_{n \geq 0}$ and $(\hat{Z}(n))_{n \geq 0}$ do not coincide in the case $\lambda_1 < \ldots < \lambda_d$. This has been observed for one-dimensional random walks, see \cite{bertoin_doney}. 

%The process $(Z(n))_{n \geq 0}$ with $h$-transform $\lvert h \rvert$ is invariant under permutations of $\lambda_1, \ldots, \lambda_d$.
%Theorem \ref{tail_asy} (iii) shows that if $\lambda_1 < \ldots < \lambda_d$ then $(\hat{Z}(n))_{n \geq 0}$ is an exponential random walk with equal rates 
%$\bar{\lambda}$ conditioned to stay positive using an $h$-transform in the case of equal rates. 

All of the above has an analogue where ordering is replaced by interlacing. The only difference comes from the fact that $h$ has been defined on all of $W^d$ while $\har$ has been defined only on $\text{int}(W^d)$.
Suppose either that $\lambda_1 > \ldots > \lambda_d$ or that all rates are equal. We define an \emph{interlaced exponential random walk} as an $h$-transform
$(Y(n))_{n \geq 0} = (Y_1(n), \ldots, Y_d(n))_{n \geq 0}$ satisfying for $x \in \text{int}(W^d)$ and bounded measurable $f$ that
\begin{align*}
\E_x[f(Y(k) : 0 \leq k \leq n)] & = \E_x\left[\frac{\har(S(n))}{\har(x)} f(S(k) : 0 \leq k \leq n) 1_{\{ \rho > n\}}\right].
\end{align*}
This defines a Markov chain on $\text{int}(W^d)$. The reason that $\har$ has been defined on $\text{int}(W^d)$ is that if the starting points coincide then almost surely 
the interlacing condition will not be satisfied even after a single step. This corresponds to the fact that $\har(x) \rightarrow 0$ 
as $x \rightarrow \partial W^d$. It is therefore not immediately obvious how to start $(Y(n))_{n \geq 0}$ 
from the boundary of $W^d$. We will focus on the case where $Y(0) \equiv 0$. 

For $x \in \text{int}(W^d)$ 
the transition densities of $Y$ are given by 
\begin{align*}
\P_x(Y(n) \in dz) & = \frac{\har(z)}{\har(x)} \prod_{j=1}^d \lambda_j^n e^{-\sum_{j=1}^d \lambda_j(z_j - x_j)} \det(q_n(z_j - x_i))_{i, j = 1}^d dz \\
& = \frac{\det(e^{-\lambda_i z_j})_{i, j = 1}^d}{\det(e^{-\lambda_i x_j})_{i, j = 1}^d}  \prod_{j=1}^d \lambda_j^n  \det(q_n(z_j - x_i))_{i, j = 1}^d dz.
\end{align*}
For $n \geq d$ take a limit as $x \rightarrow 0$ using \eqref{singular_det_limit} to find
\[
\lim_{x \rightarrow 0} \P_x(Y(n) \in dz) =  \frac{ \prod_{j=1}^d \lambda_j^n \prod_{j=1}^d z_j^{n-d} \Delta(z) \det(e^{-\lambda_i z_j})_{i, j = 1}^d}{\prod_{j=1}^d (n-j)! \Delta(\lambda)} dz.
\]
The condition that $n \geq d$ ensures differentiability of the functions inside the matrix in order to apply 
\eqref{singular_det_limit}. 
This defines an entrance law for the interlaced random walk $(Y(n))_{n \geq d}$ started from zero. 

\section{Connections to other models}
\label{exact_identities}
Ordered exponential random walks can be connected to a variety of other models. All of these connections rely on the initial condition being zero. 

%As discussed in Section \ref{intro} this is the output process of applying RSK to last passage percolation with exponential data. 
%%\textcolor{blue}{In Proposition \ref{tail.asy.rho} we prove tail asymptotics for $\rho$ which give an alternative interpretation
%%to the process $(Y(n))_{n \geq 0}$ as a random walk with exponential increments conditioned on the event that $\{\rho > m\}$ followed by taking the limit $m \rightarrow \infty$. }
%
%In this section we adapt various constructions which are known to be related to \emph{interlaced} exponential random walks to the case of \emph{ordered} exponential random walks.
%% \textcolor{blue}{One consequence is to provide three different proofs of \eqref{process_identities}.}

\subsection{Last passage percolation}
\label{zero_section}
In Section \ref{coupling} we defined a coupling that represents an ordered random walk as an interlaced random walk started from a random initial condition. There is a variant of this coupling that we only use in this subsection where we instead represent an interlaced random walk as an ordered random walk started from a random initial condition. 
We consider this only started from zero.

From the same independent collection of exponential random variables $(X_{ij})_{i \geq 1, 1 \leq j \leq d}$ with rates $\lambda_j > 0$
we define 
\begin{align*}
S_j(0) & = 0, &  1 \leq j \leq d, \\
S_j(k) & = S_j(k-1) + X_{k j}, & k \geq 1, 1 \leq j \leq d,
\end{align*}
and 
\begin{align*}
\mathcal{S}_j(k) & = 0, & 0 \leq k \leq d-j, 1 \leq j \leq d, \\
\mathcal{S}_j(k) & = \mathcal{S}_j(k-1) + X_{k-d+j, j}, & k \geq d - j +1, 1 \leq j \leq d.
\end{align*}
We have, see Figure \ref{second_coupling} for an illustration,
\[
S_j(k) = \mathcal{S}_j(k + d - j), \qquad 1 \leq j \leq d, k \geq 0. 
\]

\begin{figure}
\centering
 \begin{tikzpicture}[scale = 1]
 \node at (0, -2) {$\mathcal{S}_1(1) = 0$};
 \node at (0, -1) {$\mathcal{S}_1(2) = 0$};
  \node at (4, -1) {$\mathcal{S}_{2}(2) = S_2(1)$};
    \node at (0, 0) {$\mathcal{S}_{1}(3) = S_1(1)$};
        \node at (4, 0) {$\mathcal{S}_{2}(3) = S_2(2)$};
          %  \node at (8, 0) {$\mathcal{S}_{3}(3) = S_3(3)$};
                        \node at (0, 1) {$\mathcal{S}_{1}(4) = S_1(2)$};
                                  %  \node at (4, 1) {$\mathcal{S}_{2}(4) = 0$};
                                             %   \node at (8, 1) {$\mathcal{S}_{3}(4) = S_3(4)$};
\node at (2, 0.5) {$\leq$};
% \node at (6, 0.5) {$\leq$};
\node at (6, -0.5) {$\leq$};
\node at (2, -0.5) {$\leq$};
\node at (2, -1.5) {$\leq$};
\node at (6, -1.5) {$\leq$};
\node at (2, -2.5) {$\leq$};
\node at (6, -2.5) {$\leq$};

\node at (8, -1) {$\mathcal{S}_3(2) = S_3(2)$};
\node at (8, -2) {$\mathcal{S}_3(1) = S_3(1)$};
\node at (4, -2){$\mathcal{S}_2(1) = 0$};
\node at (0, -3){$0$};
\node at (4, -3){$0$};
\node at (8, -3){$0$};

% 
% \node at (4, 4) {$X_{11}$};
% \draw[->] (3.3, 4) -- (3.7, 4);
%  \draw[->] (4, 3.7) -- (4, 3.3);
% \node at (3, 4) {$X_{12}$};
% \draw[->] (2.3, 4) -- (2.7, 4);
%   \draw[->] (4, 2.7) -- (4, 2.3);
% \node at (2, 4) {$X_{13}$};
% \draw[->] (1.3, 4) -- (1.7, 4);
%   \draw[->] (4, 1.7) -- (4, 1.3);
% \node at (1, 4) {$X_{14}$};
% \draw[->] (3.3, 3) -- (3.7, 3);
%   \draw[->] (3, 3.7) -- (3, 3.3);
% \node at (4, 3) {$X_{21}$};
% \draw[->] (2.3, 3) -- (2.7, 3);
%   \draw[->] (3, 2.7) -- (3, 2.3);
% \node at (3, 3) {$X_{22}$};
% \draw[->] (3.3, 2) -- (3.7, 2);
% \node at (2, 3) {$X_{23}$};
%   \draw[->] (2, 3.7) -- (2, 3.3);
% \node at (4, 2) {$X_{31}$};
% \node at (3, 2) {$X_{32}$};
% \node at (4, 1) {$X_{41}$};
% \draw (0, 4) -- (4, 0);
%  \draw[->] (0.6, 3.6) -- (0.8, 3.8);
%  \draw[->] (1.6, 2.6) -- (1.8, 2.8);
%   \draw[->] (2.6, 1.6) -- (2.8, 1.8);
%    \draw[->] (3.6, 0.6) -- (3.8, 0.8);
%  \draw[snake=coil,segment aspect=1]           (3.7,1.3)  -- (3.3,1.7); 
 \end{tikzpicture}
\caption{The coupling between ordered and interlaced processes used in Section \ref{zero_section}.} 
%The inequalities give the conditioning that is present in \emph{both} interlaced and ordered processes. The $V_j^i$ are a random initial condition which makes the conditionings match.}
\label{second_coupling}
\end{figure}

%
%
%have defined two random walks $(S(n))_{\geq 0}$ started from $x$ and 
%$(\mathcal{S}(n))_{n \geq 0}$ started from a random initial condition $\Psi$. 
%
%We now 
%consider this in the case $x = 0$. 
%We suppose also that the event $A$ holds which means $\Psi \equiv 0$.  

In the case where the rates are ordered as $\lambda_1 > \ldots > \lambda_d$ then the event of positive probability that
$\bigcap_{i \geq 1} \bigcap_{j=2}^d\{S_{j-1}(i) \leq S_j(i)\}$ occurs if and only if the event $\bigcap_{i \geq 1} \{\mathcal{S}(i-1) \prec \mathcal{S}(i)\}$ occurs. Therefore the conditional laws also agree. 
This means that for all $d, n \geq 1$, if $Y_1(0) = \ldots = Y_d(0) = 0$ and $Z_1(0) = \ldots = Z_d(0) = 0$ we have
\begin{equation}
\label{identity_cond_laws2}
(Z_1(n), Z_{2}(n), \ldots, Z_d(n))_{n \geq d} \stackrel{d}{=} 
(Y_1(n+d-1), Y_2(n+d-2), \ldots, Y_d(n))_{n \geq d}. 
\end{equation}
This has been observed in \cite{o_connell_2003} and is related to a bijection between Young tableaux and reverse plane partitions. 
The restriction $n \geq d$ could be removed by modifying the definition of the entrance law for $Y$ in Appendix 
\ref{h_transform}.
In the case of equal rates we use that, 
\[
\lim_{\lambda_1, \ldots, \lambda_d \rightarrow 1} \frac{h^{(\lambda_1, \ldots, \lambda_d)}(x)}{\Delta(\lambda)} = 
\frac{h(x)}{\prod_{j=1}^{d-1} j!}, \quad \lim_{\lambda_1, \ldots, \lambda_d \rightarrow 1} \frac{\mathfrak{h}^{(\lambda_1, \ldots, \lambda_d)}(x)}{\Delta(\lambda)} = 
\frac{\mathfrak{h}(x)}{\prod_{j=1}^{d-1} j!}.
\]
For $h$ this is Lemma \ref{h_continuity}. It can be proven in a similar way for 
$\mathfrak{h}$ using \eqref{singular_det_limit}. This can be used to prove weak convergence of the Doob $h$-transforms as
$\lambda_1, \ldots, \lambda_d \rightarrow 1$. Therefore \eqref{identity_cond_laws2} also holds with $\lambda_1 = \ldots = \lambda_d = 1$.
%If $\lambda_1 < \ldots < \lambda_d$  then
%\eqref{identity_cond_laws2} still holds as long as we use the same definition for the conditioned processes, see Appendix \ref{h_transform}.

Equation \eqref{identity_cond_laws2} connects ordered exponential random walks to last passage percolation. 
It was shown in \cite{johansson_2} for equal rates that the output process of applying the Robinson-Schensted-Knuth (RSK) correspondence to last passage percolation is given by the process $(Y(n))_{n \geq 0}$. 
In particular,
\begin{equation}
\label{johansson_result}
(Y_d(n))_{n \geq d} \stackrel{d}{=} (L(n, d))_{n \geq d}.
\end{equation}
For general rates, see for example \cite{dieker_warren}.
This can be combined with \eqref{identity_cond_laws2} to give
\[
(Z_d(n))_{n \geq d} \stackrel{d}{=} (L(n, d))_{n \geq d}.
\] 
The restriction $n \geq d$ is unnecessary and is removed in the next subsection.

\subsection{Queueing theory}
\label{queueing}
Suppose that $\lambda_1 > \ldots > \lambda_d > 0$ and let $(N_1(t), \ldots, N_d(t))_{t \geq 0}$ be independent Poisson point processes where $N_j$ has rate $\lambda_{d-j+1}$ for $j = 1, \ldots, d$. 
Let $(M_1(t), \ldots, M_d(t))_{t \geq 0}$ denote $(N_1(t), \ldots, N_d(t))_{t \geq 0}$ conditoned on the event that $N_1(t) \leq \ldots \leq N_d(t)$ for all $t \geq 0$. 

O'Connell and Yor \cite{o_connell2002}
% prove that 
%\begin{equation}
%\label{poisson_lpp}
%M_1(t)\stackrel{d}{=} \inf_{0 = t_0 < t_1 < \ldots < t_n = t} \sum_{j=1}^d N_j(t_j) - N_j(t_{j-1})
%\end{equation}
%with equality as processes in $t$. 
%Moreover, they 
proved a representation for $(M_1(t))_{t \geq 0}$ in terms of a queueing network. Consider a series of $(d-1)$ tandem queues. Customers arrive at rate $\lambda_d$ at the first queue which has exponentially distributed services with rate $\lambda_{d-1}$. After departing from the first queue they immediately join the second queue which has service rate 
$\lambda_{d-2}$. This continues until the customer departs from the $(d-1)$-th queue and exits the system. It was shown in \cite{o_connell2002} that $M_1(t)$ counts the number of customers who have departed from the $(d-1)$-th queue by time $t$. 

By reversing the role of space and time, it is possible to give queueing interpretations to ordered exponential random walks. 
For $j = 1, \ldots, d$ define \[
S_j(n) = \inf\{t \geq 0: N_{d-j+1}(t) \geq n\}, \quad n \geq 0 .
\] 
Then $S_1, \ldots, S_d$ are independent random walks with exponential increments with rates $\lambda_1, \ldots, \lambda_d$ started from $S_1(0) = \ldots = S_d(0) = 0$. Moreover, the event $\{N_1(t) \leq \ldots \leq N_d(t) \text{ for all } t \geq 0 \}$ is the same as the event that $\{S_1(n) \leq \ldots \leq S_d(n) \text{ for all } n \geq 0\}$. 
Let $Z_j(n) = \inf\{t \geq 0: M_{d-j+1}(t) \geq n\}$. 
Then $(Z_1(n), \ldots, Z_d(n))_{n \geq 0}$ started from $Z_1(0) = \ldots = Z_d(0) = 0$ is equal in distribution to the times at which jumps occur in Poisson point processes conditioned not to collide. 
In particular, the queueing interpretation of $(M_1(t))_{t \geq 0}$ 
gives a queueing interpretation of $(Z_d(n))_{n \geq 0}$ as the process in $n$ of the departure times of the $n$-th customer from the $(d-1)$-th queue in the series of tandem queues defined above.

This queueing interpretation of $(Z_d(n))_{n \geq 0}$ can then be further connected with
last passage percolation and Equation \eqref{process_identities}. It is known that departure times from tandem queueing networks satisfy the same recursion equation as last passage percolation.
For $k \geq 0$ let $D(k, 1)$ denote the $k$-th arrival time at the first queue and $D(k, j+1)$ denote the $k$-th departure from the $j$-th queue for $j = 1, \ldots, d-1$. The structure of the queueing network means that
\[
D(k, j) = \max(D(k, j-1), D(k-1, j)) + e_{kj}, \quad k \geq 0, j = 1, \ldots, d.
\]
Note that last passage percolation times satisfy the same equation.

Therefore we can observe that
\[
(Z_d(n))_{n \geq 0} \stackrel{d}{=} (D(n, d))_{n \geq 0}
\]
in two different ways:
\begin{enumerate}[(i)]
\item Apply the result of O'Connell Yor \cite{o_connell2002} and reverse the role of space and time as described  
in this subsection.
\item Apply the connection between interlaced and exponential random walks in Equation \eqref{identity_cond_laws2}, 
the result of Johansson \cite{johansson_2} stated in \eqref{johansson_result} and then the above connection between last passage percolation and departure times in queues. This argument adds in an extra constraint $n \geq d$ but more careful arguments of this type could remove this.
\end{enumerate}

The case of equal rates can then be established by taking limits as in Section \ref{zero_section}.
 
%Moreover, \eqref{poisson_lpp} can be used to give a direct proof of the relation between ordered exponential random walks and last passage percolation. We have for all $n \geq 0$ and $x > 0$,
%\begin{align*}
%\P(Z_d(n) > x) & = \P(M_1(x) < n) \\
%& = \P\big(\inf_{0 = t_0 < t_1 < \ldots < t_d = x} \sum_{j=1}^d N_j(t_j) - N_j(t_{j-1}) < n\big) \\
%& = \P\big(\sup_{0 = k_0 < k_1 < \ldots < k_d = n} \sum_{j=1}^d X_j(k_j) - X_j(k_{j-1}-1) > x\big) \\
%& = \P(G(d, n) > x).
%\end{align*}
%The penultimate equality can be observed from the fact that XXX. 
%

\subsection{Push-block dynamics}
%\textcolor{blue}{A useful property is that the times at which particles jump in the Totally Asymmetric Exclusion Process (TASEP) can be interpreted through  last passage percolation. To our knowledge, it is not known that this can be applied at the level of Gelfand-Tsetlin patterns to construct a process with a bottom layer given by an ordered exponential random walk. } 

Processes on Gelfand-Tsetlin patterns where particles attempt to make independent geometrically distributed jumps while experiencing pushing and blocking interactions have been constructed in \cite{bf_anisotropic} and Section 2.2 of \cite{warren_windridge}.
Both involve particles being blocked by the positions of other particles at the \emph{previous time step}. The bottom layer evolves as
an \emph{interlaced} exponential random walk. The example below does not immediately appear to fit into the general framework in \cite{bf_anisotropic}.

Suppose that $\lambda_1 \geq \ldots \geq \lambda_d > 0$.
We will consider processes on Gelfand-Tsetlin patterns taking values in  
the state space 
\[
\K_d = \{x^k_j \in \mathbb{R} : 1 \leq j \leq k \leq d \text{ with } x^{k-1}_{j-1} \leq x^{k}_j \leq x^{k-1}_j\} 
\]
with the conventions that $x_0^k :=-\infty$ and $x_{k+1}^k = \infty$. 

We start by defining a process considered in Section 2.1 of \cite{warren_windridge} taking values in $\K_d$  and 
denoted by $(M_j^k(t): 1 \leq j \leq k \leq d, t \geq 0)$ started from $M_j^k(0) = 0$. 
Each particle $M^k_j$ attempts a nearest-neighbour jump to the right at rate $\lambda_{d-k+1}$ that may be subject to two possible interactions. Suppose the particle with position $M^k_j(t_{-})$ before the possible jump attempts to jump at time $t$.
\begin{itemize}
\item \emph{Blocking}. If $M^k_j(t_{-}) = M^{k-1}_{j}(t_{-})$ then any rightward jump is suppressed so that $M^k_j(t) = M^k_j(t_{-})$. 
\item \emph{Pushing}. If $M^k_j(t_{-}) = M^{k+1}_{j+1}(t_{-})$ and $M^{k}_{j}(t) = M^{k}_j(t_{-}) + 1$ then this pushes the particle in level $k+1$ so that $M^{k+1}_{j+1}(t) = M^{k+1}_{j+1}(t_{-}) + 1.$ This jump may then cause further jumps in levels $k+2, ..., d$. 
\end{itemize}
An argument involving intertwinings shows, for example in Theorem 2.1 of \cite{warren_windridge}, that $(M^{d}_1(t), \ldots, M^{d}_d(t))_{t \geq 0}$ is a 
collection of Poisson point process with rates $\lambda_{d} \leq \ldots \leq \lambda_1$ 
conditioned to satisfy $M_1^{d}(t) \leq \ldots \leq M_d^{d}(t)$ for all $t \geq 0$ using the harmonic function $\mathfrak{h}$.

We now construct a second process on $\K_d$ with push-block interactions by reversing the role of space and time. 
For $1 \leq j \leq k \leq d$ let
\begin{equation}
\label{defnZ}
Z^k_j(n) = \inf\{t \geq 0: M_{k-j+1}^k(t) \geq n \}, \quad n \geq 0.
\end{equation}
This defines a discrete-time process on $\K_d$ denoted by $(Z^k_j(n): 1 \leq j \leq k \leq d, n \in \mathbb{N}_0)$ 
and started from $Z^k_j(0) = 0$. We first describe the dynamics on this array before then justifying that this dynamics arises from \eqref{defnZ}. 

At time $n$ we update each layer starting with $Z^1_1$, then $Z^2_1, Z^2_2$, and so on until $Z^d_1, \ldots, Z^d_d$. 
Let $(e^k_j(n) : 1 \leq j \leq k \leq d, n \geq 0)$ be independent exponential random variables with rate $\lambda_k$. 
Suppose we have updated the positions of $Z^1_1, Z^2_1, Z^2_2, \ldots, Z^{k-1}_1, \ldots, Z^{k-1}_{k-1}$. Then for $j = 1, \ldots, k$ each $Z^k_j$ attempts an independent jump according to an exponential random variable with rate 
$\lambda_{k}$ subject to two types of interaction:
\begin{itemize}
\item \emph{Pushing}. If $Z^{k-1}_{j-1}(n) > Z^{k}_{j}(n-1)$ then $Z^k_j$ is pushed to position $Z^{k-1}_{j-1}(n)$ before performing its exponential jump.
\item \emph{Blocking}. The proposed exponential jump from this pushed position takes value 
$\max(Z_{j-1}^{k-1}(n), Z_j^k(n-1)) + e_j^k(n)$.
If this exceeds $Z^{k-1}_{j}(n)$ then the overshoot is blocked and we set 
$Z^k_j(n) = Z^{k-1}_{j}(n)$. 
\end{itemize}
Therefore the combination of pushing and blocking interactions involves setting
\begin{equation}
\label{push_block_recursion}
Z_j^k(n) = \min(Z^{k-1}_j(n), \max(Z_{j-1}^{k-1}(n), Z_j^k(n-1)) + e_j^k(n)).
\end{equation}
%\begin{itemize}
%\item \emph{Blocking}. If $Z^k_j(n-1) + e^k_j(n) > Z^{k-1}_{j}(n)$ then the overshoot is blocked and we set 
%$Z^k_j(n) = Z^{k-1}_{j}(n)$. 
%\item \emph{Pushing}. If $Z^{k-1}_{j-1}(n) > Z^{k}_{j}(n-1)$ then $Z^k_j$ is pushed to position $Z^{k-1}_{j-1}(n)$ before performing its exponential jump. This jump may still be blocked as above so we need to further split into two cases. 
%If $Z^{k-1}_{j-1}(n) + e^k_j(n) \leq Z^{k-1}_j(n)$ then there is no blocking and we set $Z^k_j(n) = Z^{k-1}_{j-1}(n) + e^k_j(n).$
%If $Z^{k-1}_{j-1}(n) + e^k_j(n) > Z^{k-1}_j(n)$ then the overshoot is blocked and we set  $Z^k_j(n) = Z^{k-1}_{j}(n)$.
%\end{itemize}
%Otherwise there are no interactions and we set $Z^{k}_j(n) = Z^{k-1}_j(n-1) + e^k_j(n).$
We now explain how these interactions are a consequence of the push-block interactions in the definition of the $M^k_j$ and the definition of the $Z^k_j$ in terms of $M^k_j$ given in \eqref{defnZ}.

Suppose first that $\inf\{t \geq 0: M_{k-j+1}^k(t) \geq n \}$ is attained without occurring due to a push by 
$M_{k-j}^{k-1}$. 
This jump in the particle labelled $M_{k-j+1}^k$ to site $n$ becomes possible after both $M_{k-j+1}^k$ has reached site $n-1$ and $M_{k-j+1}^{k-1}$ has reached site $n$ (so that the jump is not blocked). 
Thus the jump becomes possible at the time given by the maximum of $Z_{j}^k(n-1) = \inf\{t \geq 0: M_{k-j+1}^k(t) \geq n -1\}$ and
$Z_{j-1}^{k-1}(n) =  \inf\{t \geq 0: M_{k-j+1}^{k-1}(t) \geq n\}$. 
The jump then occurs after a waiting time given by an
exponential random variable denoted $e_j^k(n)$ that is independent of all other random variables. 
The other option is that $M_{k-j}^{k-1}$ jumps to site $n$ and pushes $M_{k-j+1}^k$. 
This occurs at time $Z^{k-1}_{j}(n)$. The minimum over these two possibilities gives the first  
time that $M_{k-j+1}^k$ jumps to site $n$. Therefore
\[
Z_j^k(n) = \min(Z^{k-1}_j(n), \max(Z_{j-1}^{k-1}(n), Z_j^k(n-1)) + e_j^k(n)).
\]
This agrees with \eqref{push_block_recursion}.

Suppose that $\lambda_d > \ldots > \lambda_1$. 
As the $(M^{d}_1(t), \ldots, M^{d}_d(t))_{t \geq 0}$ are 
Poisson point process with rates $\lambda_{1} < \ldots < \lambda_d$ conditioned on the event that 
$\{M^d_1(t) \leq \ldots \leq M^d_d(t) \text{ for all } t \geq 0\}$ then 
$(Z^{d}_1(n), \ldots, Z^{d}_d(n))_{n \geq 0}$ are exponential random walks with rates $\lambda_d > \ldots > \lambda_1$ conditioned on the event that $\{Z^d_1(n) \leq \ldots Z^d_d(n) \text{ for all } n \geq 0\}$. 
The two interpretations of $(Z^d_d(n))_{n \geq 0}$ as either the top particle in an ordered exponential random walk
or as the 
top particle in a system with pushing interactions give another proof of Equation \eqref{process_identities}. The case of equal rates can be established by taking limits as in Section \ref{zero_section}. 
The point of this Section is that the underlying dynamics on the Gelfand Tsetlin pattern involves a bottom layer evolving as an ordered rather than interlaced exponential random walk.

\bibliographystyle{abbrv}
 \bibliography{ExpOrdWalk}

\end{document}